\documentclass{amsart}

\usepackage{enumerate}
\usepackage{enumitem}
\usepackage{moreenum}
\usepackage{amsmath} 
\usepackage{amssymb}
\usepackage{mathrsfs}
\usepackage{hyperref}

\usepackage{lmodern}
\usepackage{textcomp}

\usepackage{xcolor}
\usepackage{enumitem}
\usepackage{wasysym}

\newtheorem{theorem}{Theorem}[section] 
\newtheorem{claim}[theorem]{Claim}

\newtheorem{lemma}[theorem]{Lemma} 
\newtheorem{proposition}[theorem]{Proposition} 
 
\newtheorem{corollary}[theorem]{Corollary}

\newcommand{\thistheoremname}{}
\newtheorem*{genericthm*}{\thistheoremname}
\newenvironment{namedthm*}[1]
{\renewcommand{\thistheoremname}{#1}%
	\begin{genericthm*}}
	{\end{genericthm*}}
\newcommand{\ov}{\overline}

\theoremstyle{definition}
\newtheorem{definition}[theorem]{Definition}

\theoremstyle{remark}
\newtheorem{remark}[theorem]{Remark}

\newcommand{\rest}{{\restriction}}

\newcommand{\cB}{{\mathcal B}}

\newcommand{\cC}{{\mathcal C}}
\newcommand{\bbC}{{\mathbb C}}

\newcommand{\cK}{{\mathcal K}}

\newcommand{\bbP}{{\mathbb P}}

\newcommand{\bbK}{{\mathbb K}}
\newcommand{\cP}{{\mathcal P}}

\newcommand{\bbN}{{\mathbb N}}

\newcommand{\cU}{{\mathcal U}}

\newcommand{\pr}{{\rm pr}}

\newcommand{\dotieconcat}[2]{
	\text{\raisebox{.8ex}{$\smallfrown$}}%
}

\newcount\skewfactor
\def\mathunderaccent#1#2 {\let\theaccent#1\skewfactor#2
	\mathpalette\putaccentunder}
\def\putaccentunder#1#2{\oalign{$#1#2$\crcr\hidewidth
		\vbox to.2ex{\hbox{$#1\skew\skewfactor\theaccent{}$}\vss}\hidewidth}}

\newenvironment{PROOF}[2][\proofname.]
{\begin{proof}[#1]\renewcommand{\qedsymbol}{\textsquare$_{\rm #2}$}}
	{\end{proof}}

\usepackage{hyperref}

\begin{document}
	
	\title[Homeomorphisms through projective Fra{\"i}ss{\'e} limits]{Homeomorphisms of continua through\\ projective Fra{\"i}ss{\'e} limits}
	\author {Márk Poór}
	\address{Department of Mathematics\\
	Cornell University\\
	Ithaca, NY 14853 USA}
	\email{mp2264@cornell.edu}

	\author {S\l{}awomir Solecki}
	\address{Department of Mathematics\\
	Cornell University\\
		Ithaca, NY 14853 USA}
	\email{ss377@cornell.edu}
	\thanks{The first author was supported by the National Research, Development and Innovation Office -- NKFIH, grants no.~146922, 129211. The second author was supported by NSF grant DMS-2246873. }



	
	\begin{abstract}
		We study homeomorphisms and the homeomorphism groups of compact metric spaces using the automorphism groups of projective Fra{\"i}ss{\'e} limits. In our applications, we investigate the Polish group ${\rm Homeo}(P)$ of all homeomorphisms of the pseudoarc $P$ using the automorphism group ${\rm Aut}(\bbP)$ of the pre-pseudoarc $\bbP$. 
Strengthening results from the literature, we show that the diagonal conjugacy action of ${\rm Homeo}(P)$ on ${\rm Homeo}(P)^{\bbN}$ has a dense orbit. In our second application, we show that there exists a homeomorphism of $P$ that is not conjugate in ${\rm Homeo}(P)$ to an element of ${\rm Aut}(\bbP)$. 
	\end{abstract}
	
	\maketitle
	\numberwithin{equation}{section}

	\section{Introduction}	\label{S:intro} 
	
	As an application of our methods based on projective Fra{\"i}ss{\'e} theory, we prove the following theorem that is phrased purely in terms of the homeomorphism group of the pseudoarc. (We provide some information on the pseudoarc later in the introduction.) 

	\begin{theorem}\label{T:homeo} Let $P$ be the pseudoarc, and let ${\rm Homeo}(P)$ be the group of all homeomorphisms of $P$ equipped with the uniform convergence topology. 
	
	For each natural number $n\geq 1$, the diagonal conjugacy action of ${\rm Homeo}(P)$ on ${\rm Homeo}(P)^n$ has a dense orbit. In fact, the diagonal conjugacy action of ${\rm Homeo}(P)$ on ${\rm Homeo}(P)^\bbN$ has a dense orbit.
	\end{theorem}
	
The statement above was conjectured for $n=1$ by Kwiatkowska \cite{KW} and was established, again for $n=1$, by Bice and Malicki \cite{B-M}. We prove the  general statement for arbitrary $n$, and also for $\bbN$, using an argument based on projective Fra{\"i}ss{\'e} theory. Our methods differ from those of \cite{B-M} and also appear to be simpler. Theorem~\ref{T:homeo} follows from the more refined Theorem~\ref{T:homeoprec}. In fact, in light of Theorem~\ref{T:hdiffa}, 
Theorem~\ref{T:homeoprec} yields a stronger conclusion even in the case $n=1$, compared with \cite{B-M}---see the discussion following Theorem~\ref{T:homeoprec}.

Theorem~\ref{T:homeo} contributes to the study of the existence of dense orbits in the diagonal conjugacy actions of Polish groups $G$ on their finite $G^n$, $n\in \bbN$, $n\geq 1$, and infinite 
$G^\bbN$ products. Such actions are defined by the formula 
\[
G\times G^I \ni (g, (h_i)_{i\in I}) \mapsto (gh_ig^{-1})_{i\in I}, 
\]
where $I=n\geq 1$ or $I= \bbN$. This line of investigation was initiated by Glasner and Weiss \cite{GW}.
When $G = {\rm Homeo}(X)$ is the group of all homeomorphisms of a compact metric space $X$, the existence of dense orbits of the diagonal conjugacy actions translates into a statement about continuous actions of free groups on $X$ as follows. Let $\Gamma$ be a countable group. A continuous action of $\Gamma$ on $X$ is identified with a homomorphism from $\Gamma$ to $G$, so it is an element of $G^\Gamma$. Consequently, the space of all continuous actions of $\Gamma$ on $X$ is a subset of $G^\Gamma$. The homeomorphism group $G$ acts naturally by conjugation on this space of continuous action of $\Gamma$ on $X$. This type of framework was investigated, for example, in \cite{DMT} and \cite{GK}. 
In this context, Theorem \ref{T:homeo} amounts to asserting that the space of continuous actions of the free group $\mathbb{F}_I$ on the pseudoarc has a dense orbit, where $I=n\in \bbN$, $n>0$, or $I=\bbN$. 
A general theory behind the existence of dense orbits for diagonal conjugacy actions for automorphism groups of countable structures was developed by Kechris and Rosendal in \cite{KR}. A survey of the area can be found in \cite{GWsurvey}. The reader may consult the papers \cite{BK}, \cite{B-M}, \cite{GK}, \cite{GWsurvey}, \cite{GW}, \cite{KR}, \cite{KT}, \cite{KW}, \cite{KWampl}, \cite{U0ample} for examples of Polish groups with dense orbits in diagonal conjugacy actions.

	The goal of this paper is broader than merely proving Theorem~\ref{T:homeo}—we investigate the relationship between the automorphism groups of projective 
	Fra{\"i}ss{\'e} limits and the homeomorphism groups of compact spaces that are canonical quotients of those limits. The mathematical context for this investigation is as follows. (For detailed information on projective Fra{\"i}ss{\'e} theory, see Appendix~\ref{A:projf}.) One starts with a countable (up to isomorphism) family $\cK$ of finite reflexive graphs and a family of epimorphisms among graphs in $\cK$. (A {\bf reflexive graph} is a binary symmetric reflexive relation. An {\bf epimorphism} is a function from the vertices of one reflexive graph to the vertices of another that preserves the edge relation and is surjective on edges and so, by reflexivity, also on vertices.) When writing $\cK$ we have in mind the family of reflexive graphs together with the family of epimorphisms. The reflexive graph relation on the structures in $\cK$ is denoted by $R$. Assuming that the epimorphisms in $\cK$ fulfill a projective amalgamation condition, a canonical projective limit $\mathbb K$ of $\cK$ exists. The object $\mathbb K$ is a compact totally disconnected metric space equipped with a set of continuous functions from $\bbK$ to the reflexive graphs in $\mathcal K$, with the reflexive graphs in $\cK$ carrying the discrete topology. This structure induces a canonical interpretation $R^{\mathbb K}$ of $R$ on $\mathbb K$, which is a compact reflexive graph relation. If $R^{\mathbb K}$ is transitive, then it is a compact equivalence relation. In that situation, we form the quotient space $K= {\mathbb K}/R^{\mathbb K}$, which is a compact metric space. Even though $\mathbb K$  as a topological space is totally disconnected, the quotient space $K$ is often connected, that is, it is a continuum---in fact, this is the most interesting situation from the topologically point of view. 
	
	We consider two groups associated with the objects above, the automorphism group ${\rm Aut}({\mathbb K})$ of $\bbK$ and the homeomorphism group ${\rm Homeo}(K)$ of the quotient space $K$. Both these groups come with natural Polish group topologies on them, namely, the topologies of uniform convergence. Furthermore, there is a natural continuous homomorphism 
	\begin{equation}\label{E:prpre} 
	{\rm pr}\colon {\rm Aut}(\bbK)\to {\rm Homeo}(K). 
	\end{equation}
	We sometimes identify elements of ${\rm Aut}(\bbK)$ with their images under $\rm pr$ in ${\rm Homeo}(K)$. For more information on this identification, see the section Notation and Conventions below in the introduction. 

Elements of ${\rm Aut}(\bbK)$ are easier to deal with than elements of ${\rm Homeo}(K)$, as they are essentially combinatorial objects and their properties are closely related to combinatorial properties of the family $\cK$. Thus, there is a drop in complexity when passing from ${\rm Homeo}(K)$ to ${\rm Aut}(\bbK)$. (Theorem~\ref{T:hdiffa} below is an example of a statement substantiating this claim.) 
We exploit this reduction in complexity in order to study ${\rm Homeo}(K)$ through ${\rm Aut}(\bbK)$, and ultimately through $\cK$. This is implemented in three steps. First, we transfer the uniform convergence topology on ${\rm Homeo}(K)$ to ${\rm Aut}(\bbK)$ using the continuous homomorphism $\rm pr$ in \eqref{E:prpre} and we give a combinatorial description of this transferred topology. Second, dualizing \cite{KR}, we develop a combinatorial approach to the existence of dense orbits in the projective setting of ${\rm Aut}(\bbK)$ and ${\rm Homeo}(K)$. A new aspect of this development is the consideration of two distinct topologies on ${\rm Aut}(\bbK)$, which leads to three distinct versions of orbit density. And third, we give applications to the homeomorphism group of the pseudoarc and to the automorphism group of the pre-pseudoarc.

	Thus, our first step is to consider the topology on ${\rm Aut}(\bbK)$ that is obtained by pulling back the topology on ${\rm Homeo}(K)$ using the continuous homomorphism ${\rm pr}$ from \eqref{E:prpre}. Since we will be using both these topologies on the group ${\rm Aut}(\bbK)$, we introduce notation for each of them. Let 
	\[
	\tau_u\;\hbox{ and }\;\tau_i
	\]
	stand, respectively, for the uniform convergence topology and the topology inherited from ${\rm Homeo}(K)$ via the homomorphism $\rm pr$ in \eqref{E:prpre} on ${\rm Aut}(\bbK)$. We also refer to the product topologies on ${\rm Aut}(\bbK)^n$, with $n\in \bbN$, $n\geq 1$,  as $\tau_u$ and $\tau_i$.  By continuity of $\rm pr$, the topology $\tau_i$ is weaker than $\tau_u$. 
In Section~\ref{S:top}, Theorem~\ref{T:thun} and Corollaries~\ref{cauchychar} and \ref{cortop}, we give combinatorial descriptions of the two topologies $\tau_u$ and $\tau_i$ on ${\rm Aut}(\bbK)$.

After that, we move to finding combinatorial conditions corresponding to the existence of dense orbits under diagonal conjugacy actions. We note that the group ${\rm Homeo}(K)$ will always be considered with its natural topology  (corresponding to uniform convergence on $K$) and the products ${\rm Homeo}(K)^n$ and ${\rm Homeo}(K)^\bbN$ with the products of the uniform convergence topology. 
Since we will be assuming that ${\rm Aut}(\bbK)$ is dense in ${\rm Homeo}(K)$, when considering density of orbits of diagonal conjugacy actions, it suffices to consider conjugacy by elements of ${\rm Aut}(\bbK)$. 
This leads to three types of conjugacy actions. Given $n\in \bbN$, $n\geq 1$, we consider:
	\begin{enumerate}
	\item[(a)] conjugacy by elements of ${\rm Aut}(\bbK)$ of tuples in ${\rm Aut}(\bbK)^n$ taken with $\tau_u$; 
	
	\item[(b)] conjugacy by elements of ${\rm Aut}(\bbK)$ of tuples in ${\rm Aut}(\bbK)^n$ taken with $\tau_i$; 
	
	\item[(c)] conjugacy by elements of ${\rm Aut}(\bbK)$ of tuples in ${\rm Homeo}(K)^n$. 
	\end{enumerate} 
	Observe that having a dense orbit with respect to an action of type (a) implies the existence of a dense orbit in the corresponding action of type (b), which, in turn, implies the existence of a dense orbit with respect to the action of type (c).  
	
	We obtain combinatorial conditions equivalent to the following properties of the actions as in (a)--(c):
	\begin{enumerate}
	\item[(a$'$)] there exist $\tau_u$-comeagerly many $\ov{\gamma}\in {\rm Aut}(\bbK)^n$ with conjugacy orbits $\tau_u$-dense in ${\rm Aut}(\bbK)^n$; 

	\item[(b$'$)] there exist $\tau_u$-comeagerly many $\ov{\gamma}\in {\rm Aut}(\bbK)^n$ with conjugacy orbits $\tau_i$-dense in ${\rm Aut}(\bbK)^n$; 
		
	\item[(c$'$)]  there exist comeagerly many $\ov{\gamma}\in {\rm Homeo}(K)^n$ with conjugacy orbits dense in ${\rm Homeo}(K)^n$. 
		\end{enumerate} 
The three combinatorial conditions are related but distinct. They are versions of the Joint Projection Property for suitable categories and are obtained by dualizing to the projective setting and adapting to the mix of two topologies (one coming from ${\rm Aut}(\bbK)$ and the other one from ${\rm Homeo}(K)$) of the Joint Embedding Property in the paper by Kechris and Rosendal \cite{KR}. Our conditions are stated, and the theorems establishing the equivalence between the appropriate Joint Projection Property and the existence of dense orbits in the corresponding conjugacy action are proved in Section~\ref{S:jpp}, Theorem~\ref{mcor} and Corollaries~\ref{m2cor} and \ref{m2cor2}.
We note that the case (b)/(b$'$) exhibits the most interesting interaction between the two topologies on ${\rm Aut}(\bbK)$: comeagerness in (b$'$) refers to the uniform convergence topology $\tau_u$ on ${\rm Aut}(\bbK)$, while density of orbits concerns the topology $\tau_i$ inherited from ${\rm Homeo}(K)$. This is precisely the case applied to the pseudoarc---see below.
Finally, note that since the property of a point having a dense orbit under a continuous action of a Polish group is $G_\delta$, the properties in (a$'$) and (c$'$) are equivalent to the existence of a single dense orbit in ${\rm Aut}(\bbK)^n$ and ${\rm Homeo}(K)^n$, respectively.

	For our applications, recall that the {\bf pseudoarc} is a {\bf continuum}, that is, a compact connected metric space, constructed by Knaster \cite{Kn} and characterized by Bing \cite{B2} as the unique chainable hereditarily indecomposable continuum. In the same paper \cite{B2}, Bing gave another compelling characterization of the pseudoarc. It is 
	the unique continuum that is generic in the following sense. Let $\cC$ be the space of all continua included in the Hilbert cube $[0,1]^\bbN$ endowed with the Vietoris topology, that is, the topology induced by the Hausdorff metric. The space $\cC$ is a compact metric space. As proved in \cite{B2}, there exists a continuum $P$ such that ``topologically most'' elements of $\cC$ are homeomorphic to $P$, that is, the subset of $\cC$ consisting of copies of $P$ is comeager in $\cC$, in fact, it is a dense $G_\delta$. This continuum $P$ is the pseudoarc. For more information on the pseudoarc, the reader may consult the survey \cite{Le} or the book \cite{Na}. 
	
	Consider the family $\cP$ consisting of all structures isomorphic to reflexive graphs of the following form: 
	\[
	L= \{ 0, \dots, n\}\, \hbox{ with }\, xR^Ly \Leftrightarrow |x-y|\leq 1, \hbox{ for }x,y\in L,
	\]
	where $n\in \bbN$. Morphisms in $\cP$ are all epimorphisms among structures in $\cP$. As proved in \cite{IS}, the family $\cP$ forms a transitive projective Fra{\"i}ss{\'e} class and its limit $\bbP$ is such that $P= \bbP/R^{\bbP}$ is homeomorphic to the pseudoarc. 
	In Section~\ref{S:psco}, we use this representation of the pseudoarc from \cite{IS} to apply our results of Section~\ref{S:jpp} to the pseudoarc. We prove the appropriate Joint Projection Property for the class $\cP$, which by the results of Section~\ref{S:jpp}, yields the theorem below. A still stronger version of this result is stated and proved as Theorem~\ref{S3mainT}.
	
	\begin{theorem}\label{T:homeoprec}
	For each natural number $n\geq 1$, there exists an element of ${\rm Aut}(\bbP)^n$ whose orbit with respect to the diagonal conjugacy action of ${\rm Aut}(\bbP)$ 
	is dense in ${\rm Homeo}(P)^n$.
		 Furthermore, there exists an element of ${\rm Aut}(\bbP)^\bbN$ whose orbit with respect to the diagonal conjugacy action ${\rm Aut}(\bbP)$ 
		 is dense in ${\rm Homeo}(P)^\bbN$.
	 \end{theorem} 	  
	  
	  Theorem~\ref{T:homeo} follows immediately from Theorem~\ref{T:homeoprec}. Moreover, we note that Theorem~\ref{T:homeoprec} provides additional information compared to Theorem~\ref{T:homeo}, as it asserts that the element with a dense conjugacy class lies in ${\rm Aut}(\bbP)^n$ or ${\rm Aut}(\bbP)^\bbN$, rather than merely in ${\rm Homeo}(P)^n$ or ${\rm Homeo}(P)^\bbN$. For $n = 1$, this yields an additional improvement over the theorem proved in \cite{B-M}. Of course, this is an improvement only if one can show that not every homeomorphism of $P$ can be conjugated to an element of ${\rm Aut}(\bbP)$. We now turn to this issue.

Since ${\rm Aut}(\bbP)$ is dense in ${\rm Homeo}(P)$, each homeomorphism of $P$ can be approximated by an automorphism of $\bbP$ to an arbitrary degree of precision. A natural problem is whether a homeomorphism of $P$ can in fact be realized as an automorphism of $\bbP$. This problem can be stated precisely in terms of the conjugacy action of ${\rm Homeo}(P)$ on itself: is every homeomorphism in ${\rm Homeo}(P)$ conjugate, within ${\rm Homeo}(P)$, to an element of ${\rm Aut}(\bbP)$? This question is further motivated by the remarks following Theorem~\ref{T:homeoprec} above. 
In Section~\ref{s4}, we answer it in the negative by exhibiting a homeomorphism whose conjugacy class does not intersect ${\rm Aut}(\bbP)$. The following result is restated and proved as Theorem~\ref{T:homeout}.

\begin{theorem}\label{T:hdiffa} 
There exists an element of ${\rm Homeo}(P)$ that is not conjugate in ${\rm Homeo}(P)$ to an element of ${\rm Aut}(\bbP)$. 
\end{theorem}

Theorem~\ref{T:hdiffa} above is the first example of the phenomenon described in it in the projective Fra{\"i}ss{\'e} context. The construction uses our analysis of topologies in Section~\ref{S:top}. The homeomorphism of $P$ in the conclusion of this theorem is obtained as the limit of a sequence of automorphisms in ${\rm Aut}(\bbP)$ that is Cauchy with respect to the uniformity inducing the topology $\tau_i$ on ${\rm Aut}(\bbP)$ inherited from ${\rm Homeo}(P)$. We expect this method to be useful in constructing homeomorphisms of $P$ with other properties.

	\subsection*{Notation and conventions} 
	A short exposition of the projective Fra{\"i}ss{\'e} theory is given in Appendix~\ref{A:projf}. With this in mind and following Appendix~\ref{A:projf}, 
	we fix notation and some conventions for the paper. 
	\begin{enumerate} 
	\item[---] $\mathcal K$ is a countable projective Fra{\"i}ss{\'e} family; 
	
	\item[---] $\mathbb{K}$ with the binary relation $R^{\mathbb K}$ is the projective Fra{\" i}ssé limit of $\mathcal{K}$; 
	
	\item[---] if $R^{\mathbb{K}}$ is transitive, then $K = \mathbb{K} /R^{\mathbb{K}}$ is the canonical quotient topological space and $\pr: \mathbb{K} \to K$ is the canonical projection.
	\end{enumerate} 
	By ${\rm Aut}(\bbK)$ we denote the automorphism group induced by $\cK$ and
	again by $\pr$ we denote the continuous homomorphism $\pr\colon {\rm Aut}(\bbK)\to {\rm Homeo}(K)$ induced by the quotient map $\pr: \mathbb{K} \to K$. 

We often identify elements of ${\rm Aut}({\mathbb K})$ with their images under the homomorphism $\pr$ in ${\rm Homeo}(K)$, that is, we consider $f\in {\rm Aut}(\bbK)$ as the  element ${\rm pr}(f)$ of ${\rm Homeo}(K)$ keeping in mind the identification spelled out in \eqref{E:preq} in Appendix~\ref{A:projf}. 
Such a move is justified by $\rm pr$ being injective under mild assumptions on $\cK$, see Proposition~\ref{P:inj}. In particular, this is true in all the topological situations considered in the literature, for example, in the case of the pseudoarc and the associated with it projective Fra{\"i}ss{\'e} family $\cP$.

	\section{Topologies on  ${\rm Aut}(\bbK)$}\label{S:top}
	
	For notion not defined in this section, the reader should refer to Appendix~\ref{A:projf}.

	We start with a lemma that will be useful in several places. 
	
	\begin{lemma}\label{hyple} Assume that $\mathcal K$ is transitive. 
		Fix $k\in {\mathbb N}$, $k\geq 1$. For each epimorphism $\phi\colon \mathbb{K}\to A$, there exist epimorphisms $\psi\colon \mathbb{K}\to B$ and $f\colon B\to A$ such that 
		\begin{enumerate} 
		\item[(i)] $\phi= f\circ \psi$;
		
		\item[(ii)] if $a, b\in B$ and $aR^kb$, then $f(a)Rf(b)$ .
		\end{enumerate}
	\end{lemma}
	\begin{PROOF}{Lemma \ref{hyple}}
	We can assume that there exists a generic sequence $\pi_{i, i+1}\colon A_{i+1}\to A_i$ with $\mathbb{K}=\projlim_i (A_i, \pi_{i,i+1})$ and $\phi= \pi_0\colon \mathbb{K}\to A_0=A$. 
		Assume we have a sequence $(j_l)$ of natural numbers with $0<j_l<j_{l+1}$ and points $a_l, b_l\in A_{j_l}$ with 
		\begin{equation}\label{E:qw}
		a_lR^kb_l\;\hbox{ and }\;\neg\big( \pi_{0, j_l}(a_l)\, R\,\pi_{0,j_l}(b_l)\big). 
		\end{equation}
		By going to subsequences, we can assume that for some $a,b\in A_0$ and all $l$, 
		\begin{equation}\label{E:er}
		\pi_{j_l, j_{l+1}}(a_{l+1})=a_l,\;  \pi_{j_l, j_{l+1}}(b_{l+1})=b_l,\;\pi_{0,j_l}(a_l)=a,\hbox{ and }\pi_{0,j_l}(b_l)=b. 
		\end{equation}
		We now fix sequences $x_l,y_l\in \mathbb{K}$, $l\in \mathbb{N}$, with $\pi_{j_l}(x_l)=a_l$ and $\pi_{j_l}(y_l)=b_l$, and, by compactness, assume that they converge to $x$ and $y$, respectively. By \eqref{E:qw} and \eqref{E:er}, we have that $\pi_{j_l}(x)R^k\pi_{j_l}(y)$, for all $l$, so $xR^ky$. Since $R$ is transitive on $\mathbb{K}$, we get $xRy$. This condition implies, by \eqref{E:er}, that $aRb$, which contradicts \eqref{E:qw}, and the lemma follows. 		
	\end{PROOF}

Fix a metric $d^0$ on $K$. The metric induces the supremum metric
	$$ d(f,g) = \sup\{d^0(f(x),g(x)): \ x \in K\} \ \ (f,g \in \mathcal{C}(K,K)).$$
	By the same letter $d$, we denote the pseudometric on ${\rm Aut}(\bbK)$ given by 
	\[
	d(f,g) = d\big(\pr(f), \pr(g)\big). 
	\]
	Let 
	\[
	{\rm Epi}_\bbK = \bigcup_{A\in {\mathcal K}} {\rm Epi}(\bbK, A).
	\]	
	With each epimorphism $\phi\in {\rm Epi}_\bbK$, we associate the set
	  \[
	  U_\phi = \{ (f,g)\in {\rm Aut}(\bbK)\times {\rm Aut}(\bbK)\mid \phi\circ f = \phi\circ g\}
	  \]
	  and for $k\in \bbN$, $k>0$, the set 
	   \[
	  U^{(k)}_\phi = \{ (f,g)\in {\rm Aut}(\bbK)\times {\rm Aut}(\bbK)\mid \phi\circ f \,R^k \,\phi\circ g\}
	  \]
	  We will write $U^{(0)}_\phi$ for $U_\phi$. For $k\in \bbN$, let 
	  \[
	  {\mathcal U}^{(k)}
	  \]
	  consist of all subsets of ${\rm Aut}(\bbK)\times {\rm Aut}(\bbK)$ containing a set of the form $U^{(k)}_\phi$.

	\begin{lemma}\label{L:low} 
	Fix $k\in \bbN$, $k\geq 1$. 
	\begin{enumerate} 
	\item[(i)] $U^{(1)}_\phi\subseteq U^{(k)}_\phi$, for each $\phi\in {\rm Epi}_\bbK$.  
	
	\item[(ii)] For each $\phi\in {\rm Epi}_\bbK$, there exists $\psi\in {\rm Epi}_\bbK$ such that 
	\[
	U^{(k)}_\psi\subseteq U^{(1)}_\phi. 
	\]
	\end{enumerate} 
	\end{lemma}
	
	\begin{proof} Point (i) is immediate from the definition of $U^{(k)}_\phi$. Point (ii) is a consequence of Lemma~\ref{hyple}. 
	\end{proof}

	Recall that a {\bf uniformity} on a set $X$ is a family $\cU$ of subsets of $X\times X$ such that 
	\begin{enumerate} 
	\item[---] $\{ (x,x)\mid x\in X\}\subseteq U$ for each $U\in \cU$; 
	
	\item[---] for each $V\in \cU$, there exists $U\in \cU$, such that $U\circ U\subseteq V$; 
	
	\item[---] if $U\in \cU$, then $U^{-1}\in \cU$; 
	
	\item[---] if $U\in \cU$ and $U\subseteq V$, then $V\in \cU$. 
	\end{enumerate}  
	If $\rho$ is a pseudometric on a set $X$, the {\bf uniformity induced by $\rho$} is the family of all subsets of $X\times X$ containing sets of the form 
	\[
	\{ (x,y)\in X\times X\mid \rho(x,y)<r\}, \hbox{ for }r>0. 
	\]

	 \begin{theorem}\label{T:thun} Assume $\cK$ is transitive. 
	  \begin{enumerate} 
	  \item[(i)] ${\mathcal U}^{(k)}$ is a uniformity on ${\rm Aut}(\bbK)$, for each $k\in \bbN$.

	  \item[(ii)] The uniformity $\,{\mathcal U}^{(0)}$ is equal to the uniformity induced by the uniform metric on ${\rm Aut}(\bbK)$. 	  
	  
	  \item[(iii)] For each $k\geq 1$, the uniformity $\,{\mathcal U}^{(k)}$ is equal to the uniformity induced by the pseudometric $d$ on ${\rm Aut}(\bbK)$. 
	  \end{enumerate} 
	  \end{theorem}

	\begin{proof}	
	We only handle the case $k>0$, which is somewhat trickier than $k=0$. To see point (i) for $k>0$ and point (iii), it suffices to show that
	\begin{equation}\label{E:epep1}
	\forall \epsilon>0\;\exists \phi \in {\rm Epi}_\bbK \;\;  U^{(k)}_\phi\subseteq \{ (\sigma,\tau)\in {\rm Aut}(\bbK)\times {\rm Aut}(\bbK)\mid d(\sigma,\tau)<\epsilon\}
	\end{equation} 
	and, conversely, 
	\begin{equation}\label{E:epep2} 
	\forall\phi\in {\rm Epi}_\bbK\; \exists \epsilon >0\;\; \{ (\sigma,\tau)\in {\rm Aut}(\bbK)\times {\rm Aut}(\bbK)\mid d(\sigma,\tau)<\epsilon\}\subseteq U^{(k)}_\phi. 
	\end{equation} 
	
	By Lemma~\ref{L:low}, it suffices to show the statements above for $k=1$ only. 
	
	We show \eqref{E:epep1} for $k=1$ first. Fix $\epsilon>0$. We prove that there exists $\phi\in {\rm Epi}_\bbK$ such that 
	\begin{equation}\label{E:varo} 
	\forall x,y\in \mathbb{K}\; \big(\phi(x)R\phi(y)\Rightarrow d^0(\pr(x), \pr(y))<\epsilon/2\big).
	\end{equation} 
	Note that this statement implies that, for all $\sigma,\tau\in {\rm Aut}(\bbK)$, 
	\[
	\phi\circ \sigma \,R\, \phi\circ \tau \Rightarrow d(\sigma,\tau)\leq\epsilon/2<\epsilon, 
	\]
	and \eqref{E:epep1} for $k=1$ is proved. 
	 
	 We proceed to proving \eqref{E:varo}. Since $\pr$ is continuous and $\mathbb{K}$ is compact, there exists $\delta>0$ such that for $x,y\in \mathbb{K}$, if $d(x,y)<\delta$, then $d^0(\pr(x),\pr(y))<\epsilon/2$. Let now $\phi\colon \mathbb{K}\to A$ be an epimorphism such that preimages of points have diameter $<\delta$. We claim that this $\phi$ works.  Let $x,y\in \mathbb{K}$ be such that $\phi(x)R\phi(y)$. Since $\phi$ is an epimorphism, there exist $x',y'\in \mathbb{K}$ such that 
	\[
	\phi(x')=\phi(x),\; \phi(y')=\phi(y), \hbox{ and } x'Ry'.
	\]
	Then, by our choice of $\phi$, $d(x,x')<\delta$ and $d(y,y')<\delta$, so $d^0(\pr(x), \pr(x'))<\epsilon/2$ and $d^0(\pr(y), \pr(y'))<\epsilon/2$. Since $x'Ry'$, we have $\pr(x')=\pr(y')$. It follows that $d^0(\pr(x), \pr(y))<2\epsilon/2=\epsilon$. 

	Now, we show \eqref{E:epep2} for $k=1$. Fix an epimorphism $\phi\colon \mathbb{K}\to A$. We are looking for $\epsilon>0$ to satisfy the inclusion in \eqref{E:epep2}. Consider $a,b\in A$ such that $\neg(aRb)$. Note that $\phi^{-1}(a)$ and $\phi^{-1}(b)$ are clopen subsets of $\mathbb{K}$ with 
	\[
	R\big( \phi^{-1}(a)\big)\cap \phi^{-1}(b)=\emptyset. 
	\]
	It follows that $\pr\big(\phi^{-1}(a)\big)$ and $\pr\big( \phi^{-1}(b)\big)$ are disjoint compact subsets of $K$, so they are at a positive $d^0$ distance $\epsilon_{a,b}>0$ from each other. Since $A$ is finite, we can let $\epsilon>0$ be the minimum of all $\epsilon_{a,b}$ for $a,b\in A$ with $\neg (aRb)$. Now, if $\sigma,\tau\in {\rm Aut}(\bbK)$ are such that $d(\sigma,\tau)<\epsilon$, then, for each $x\in \bbK$, $d^0\big(\pr(\sigma(x)), \pr(\tau(x))\big)<\epsilon$, so, by our choice of $\epsilon$, we have $\sigma(x)R \tau(x)$, for each $x\in \bbK$; thus, 
	$(\sigma,\tau)\in U^{(1)}_\phi$, as required. 
		\end{proof}

	We state two immediate corollaries of Theorem~\ref{T:thun}. 
	
	\begin{corollary}\label{cauchychar} Assume $\mathcal K$ is transitive. 
		Let $(\sigma_n)$ be a sequence in ${\rm Aut}(\mathbb{K})$. 
		\begin{enumerate} 
		\item[(i)] If $(\sigma_n)$ is $d$-Cauchy in $\mathcal{C}(K,K)$, then for each $\varphi \in {\rm Epi}(\mathbb{K}, A)$, there exists $N$ such that 
		$$\forall n,m \geq N\; \forall x \in \mathbb{K}\ \big(
		\varphi(\sigma_n(x)) R \varphi(\sigma_m(y))\big).$$
				
		\item[(ii)] \label{2ndclause} If for each $\varphi \in {\rm Epi}(\mathbb{K}, A)$, there exists $N$ such that 
		$$\forall n,m \geq N\; \forall x \in \mathbb{K} \ \big(\varphi(\sigma_n(x)) R \varphi(\sigma_m(x))\big),$$
		then the sequence $(\sigma_n)$ is $d$-Cauchy  in $\mathcal{C}(K,K)$.
		\end{enumerate} 
	\end{corollary}

	\begin{proof} The conclusions are immediate from Theorem~\ref{T:thun} (ii) and (iii). 
	\end{proof}

%
%
%

With each pair of epimorphisms $\phi,\psi \in {\rm Epi}_\bbK$, we associate  the following subsets of ${\rm Aut}(\bbK)$:
		$$B_{\psi,\phi} = \{\tau \in {\rm Aut}(\bbK): \ \psi = \phi \circ \tau\},$$
		and for $k\in {\mathbb N}$, $k\geq 1$, 
		$$B^{(k)}_{\psi,\phi} = \{ \tau \in {\rm Aut}(\bbK): \psi \ R^k \ \phi \circ \tau \}.$$ 
		We may also  write $B^{(0)}_{\psi,\phi}$ instead of $B_{\psi,\phi}$. 
		
		Recall that a family $\mathcal B$ of subsets of a topological space $X$ is called a {\bf neighborhood basis} if for each $x\in X$ and open set $U\subseteq X$ with $x\in U$ there exists $B\in \cB$ such that $B\subseteq U$ and $x$ is in the interior of $B$.


		\begin{corollary} {}\ \label{cortop}  \label{topobs} Assume $\cK$ is transitive. 
	\begin{enumerate} 
	\item[(i)]  Sets $B_{\phi, \psi}$, with $\phi, \psi\in {\rm Epi}_\bbK$, form a clopen neighborhood basis of the topology on ${\rm Aut}(\bbK)$.
	
	\item[(ii)] Fix $d\geq 1$. Sets $B^{(d)}_{\phi,\psi}$, with $\phi, \psi\in {\rm Epi}_\bbK$, form a neighborhood basis of the topology on 
	${\rm Aut}(\bbK)$ induced by the pseudometric $d$, that is, the topology inherited from ${\rm Homeo}(K)$. 
	\end{enumerate} 
	\end{corollary}
	\begin{PROOF}{Corollary \ref{cortop}} Point (i) follows from Theorem~\ref{T:thun}~(ii) and point (ii) from Theorem~\ref{T:thun}~(iii). We give details for the latter argument. By Theorem~\ref{T:thun}~(iii), for $\sigma\in {\rm Aut}(\bbK)$, sets of the form 
	\[
	\{ \tau\in {\rm Aut}(\bbK)\mid (\phi\circ \sigma)\, R^d\,  (\phi \circ \tau)\}, \hbox{ for } \phi\in {\rm Epi}_\bbK, 
	\] 
	are a neighborhood basis at $\sigma$ of the topology induced by $d$ containing $\sigma$ in their interiors. Setting $\psi= \phi\circ \sigma$, the conclusion follows. 
	\end{PROOF}

%

	\section{Joint Projection Properties}\label{S:jpp} 
	
	\subsection{The formulation of Joint Projection Properties} 
	
	Let $\mathcal K$ be a category. 
	Following Kechris--Rosendal \cite[Section~2]{KR}, and adapting their work to the projective setting, we define the category $\cK_{\textrm{p}}$ as follows.
	
	\begin{definition} {}\
		\begin{enumerate}
			\item[] {\em Objects}: $(A,B,f,g) \in \cK_{\textrm{p}}$ iff $A,B \in \cK$, $f,g \in {\rm Epi}(A,B)$.
			\item[] {\em Morphisms}: The $\cK$-epimorphism  $\alpha:A \to A'$ is an epimorphism in $\cK_{\mathrm{p}}$ between $(A,B,f,g)$ and $(A',B',f',g')$ iff
			\[
			\forall a_0,a_1 \in A \ \big( f(a_0) = g(a_1) \ \Rightarrow \ f'(\alpha(a_0)) = g'(\alpha(a_1))\big).
			\]
		\end{enumerate}
		\end{definition}
		
		It is not difficult to see that $\alpha$ is an epimorphism in $\cK_{\mathrm{p}}$ between $(A,B,f,g)$ and $(A',B',f',g')$ precisely when there exists an $\cK$-epimorphism $\beta: B \to B'$, such that 
		\[
		\beta \circ f = f' \circ \alpha\;\hbox{ and }\; \beta \circ g = g' \circ \alpha.
		\]

We now modify the definition of morphism in $\cK_p$ to obtained the definition of an approximate morphism in this category. 
		
	\begin{definition}
		Let $(A,B,f,g)$ and $(A',B',f',g') \in \cK_{\mathrm{p}}$.
		The $\mathcal K$-epimorphism $\alpha$ between $A$ and $A'$ is an \emph{approximate} epimorphism in $\cK_{\mathrm{p}}$ between $(A,B,f,g)$ and $(A',B',f',g')$ iff
			$$\forall a_0,a_1 \in A: \ f(a_0) = g(a_1) \ \Rightarrow \ f'(\alpha(a_0))\  R \ g'(\alpha(a_1)).$$
	\end{definition}	
		\begin{definition} \label{xndf}
		For $n \in \bbN$, we let  $\cK^{\times n}_{\textrm{p}}$ denote the class of objects of the form $(A,B,f_0,\ldots, f_{n-1},g_0, \ldots, g_{n-1})$, where  $(A,B,f_i,g_i) \in \cK_{\textrm{p}}$, for each $i<n$.
		
		We call $\alpha \in {\rm Epi}(A,A')$ an epimorphism (an approximate epimorphism, respectively) between 
		$$(A,B,f_0,\ldots , f_{n-1}, g_0, \ldots,g_{n-1})$$ and 
		$$(A',B',f'_0,\ldots, f_{n-1}', g'_0, \ldots, g'_{n-1})$$ if for each $i<n$ the map $\alpha$ is an epimorphism (an approximate epimorphism, respectively) between $(A,B,f_i,g_i)$ and $(A',B',f'_i,g'_i)$.
	\end{definition}
	When $n$ is clear from the context, we write 
	\[
	(A,B, \ov{f}, \ov{g})\;\hbox{ for } \; 
	(A,B,f_0,\ldots , f_{n-1}, g_0, \ldots,g_{n-1}). 
	\]
	Fix $n$. We say that $\cK^{\times n}_{\textrm{p}}$ has 
	\begin{enumerate}
	\item[(a)] JPP, 
	
	\item[(b)] half-approximate JPP, 
	
	\item[(c)] approximate JPP, 
	\end{enumerate} 
	if for all $(A,B, \ov{f}, \ov{g})$ and $(A',B', \ov{f'}, \ov{g'})$ in $\cK^{\times n}_{\textrm{p}}$, there exist $(A^+,B^+, \ov{f^+}, \ov{g^+}\,)$ in $\cK^{\times n}_{\textrm{p}}$ and 
	\[
	\alpha\colon (A^+,B^+, \ov{f^+}, \ov{g^+}\,)\to (A,B, \ov{f}, \ov{g})\;\hbox{ and }\;\alpha' \colon (A^+,B^+, \ov{f^+}, \ov{g^+}\,)\to(A',B', \ov{f'}, \ov{g'})
	\]
	such that  
	\begin{enumerate} 
	\item[(a)] $\alpha$ and $\alpha'$ are epimorphisms, 
	
	\item[(b)] $\alpha$ is an epimorphism and $\alpha'$ is an approximate epimorphism, 
	
	\item[(c)] $\alpha$ and $\alpha'$ are approximate epimorphisms. 
	\end{enumerate}

	\subsection{Joint Projection Properties and density of orbits for diagonal conjugacy actions} 
	
	We now come to the main theorem of this section.
	For a sequence $\overline{\gamma} = (\gamma_j)_{j<n}$ of elements of a group $G$, we write 
	\[
	\overline{\gamma}^G= \{ (g\gamma_jg^{-1})_{j<n}\mid g\in G\} \subseteq G^n. 
	\]

		\begin{theorem} \label{mcor}
		Suppose that $\cK$ is a transitive projective Fra{\"i}ssé class with the property that the image of ${\rm Aut(\bbK)}$ under the canonical homomorphism $\rm pr$ is dense in ${\rm Homeo}(K)$. Let $n \in \bbN$. 
		\begin{enumerate}[label = $(\roman*)$, ref = $(\roman*)$]
			\item \label{ge} If $\cK^{\times n}_{\textrm{p}}$ has the approximate JPP, then, for a non-empty open set $U \subseteq {\rm Homeo}(K)^n$,  
			\[
			\{ \overline{\gamma} \in {\rm Aut}(\bbK)^n: \ \overline{\gamma}^{{\rm Aut}(\bbK)} \cap U  \neq \emptyset \} \ \textrm{ is dense in } {\rm Homeo}(K)^n.
			\]
						
			\item \label{HalfAppr} If $\cK^{\times n}_{\textrm{p}}$ has the half-approximate JPP, 
			then, for a non-empty open set $U \subseteq {\rm Homeo}(K)^n$, 
			$$\{ \overline{\gamma} \in {\rm Aut}(\bbK)^n: \ \overline{\gamma}^{{\rm Aut}(\bbK)} \cap U  \neq \emptyset \} \ \textrm{ is dense in } {\rm Aut}(\bbK)^n.$$
			
			\item \label{KR}	If $\cK^{\times n}_{\textrm{p}}$ has the JPP, then for a non-empty open set $U \subseteq {\rm Aut}(\bbK)^n$, 
			$$\{ \overline{\gamma} \in {\rm Aut}(\bbK)^n: \ \overline{\gamma}^{{\rm Aut}(\bbK)} \cap U  \neq \emptyset \} \ \textrm{ is dense in } {\rm Aut}(\bbK)^n.$$
		\end{enumerate}
	\end{theorem}
	
	Note that \ref{KR} is the dualized version of \cite{KR}.
	
	First we note that with epimorphisms $f,g\colon A\to B$ and $\varphi \in {\rm Epi}(\bbK, A)$ we can naturally associate the clopen set $B_{f \circ \varphi,g \circ \varphi}$, for which we introduce the following shorthand notation.
	\begin{definition}
	  If $f,g\colon A\to B$ are epimorphisms and $\varphi \in {\rm Epi}(\bbK, A)$, then we let
	 \begin{enumerate}
	 	\item[---] $B_{f,g; \varphi} = B_{f \circ \varphi,g \circ \varphi},$ that is, 
	 	$\sigma \in B_{f,g; \varphi}$ iff $(g\circ \varphi) \circ \sigma = f\circ \varphi$,
	 	\item[---]  $B^{(k)}_{f,g; \varphi} = B^{(k)}_{f \circ \varphi,g \circ \varphi},$ that is, 
	 	$\sigma \in B^{(k)}_{f,g; \varphi}$ iff $(g\circ \varphi) \circ \sigma \ R^k \ f\circ \varphi$.
	 \end{enumerate} 
	\end{definition}

	Before embarking on proving the theorem, we need several lemmas, the first one of which describes the behavior, relevant to our proof, of sets of the form $B^{(k)}_{f,g;\varphi}$ under conjugation.

		\begin{lemma} \label{conju}
		Let $(A,B,f,g)$ and $(A',B',f',g') \in \cK_{\mathrm{p}}$. Suppose that 
		\[
		\alpha\in {\rm Epi}(A, A'),\;\, \varphi \in {\rm Epi}(\bbK, A),\;\, \varphi' \in {\rm Epi}(\bbK, A'), \hbox{ and }\, \sigma \in B_{\alpha \circ \varphi, \varphi'}.
		\] 
		If $\alpha$  is an approximate ${\mathcal K}_p$-epimorphism between $(A,B,f,g)$ and $(A',B',f',g')$, then 
		\begin{equation} \label{eegy} 
		\sigma B_{f,g;\varphi}^{(0)} \sigma^{-1} \subseteq B^{(1)}_{f',g';\varphi'} 
		\end{equation}
				If $\alpha$ is a $\cK_{\textrm{p}}$-epimorphism, then
		\begin{equation}\label{E:ddgy} 
		\sigma B_{f, g;  \varphi}^{(1)} \sigma^{-1} \subseteq B^{(1)}_{f',g'; \varphi'},
		\end{equation} 
		and
		\begin{equation}\label{E:ffgy} 
			\sigma B^{(0)}_{f, g;  \varphi} \sigma^{-1} \subseteq B^{(0)}_{f',g'; \varphi'}.
		\end{equation} 
		\end{lemma}

	\begin{PROOF}{Lemma \ref{conju}} We prove \eqref{eegy} first. 
		Given $\sigma^* \in B_{f \circ  \varphi,g \circ \varphi}^{(0)}$, we need to check that 
		\[
		\sigma \sigma^* \sigma^{-1} \in B^{(1)}_{f',g'; \varphi'}.
		\]
		Since $\sigma^* \in B_{f,g;\varphi}^{(0)}$ means that $f\circ \varphi \ =\ g\circ \varphi\circ \sigma^*$, we have
		\begin{equation}\label{E:alp} 
		(f'\circ \alpha\circ \varphi )\ R \ (g'\circ \alpha\circ \varphi\circ \sigma^*) .
		\end{equation} 
		On the other hand, $\sigma \in B_{\alpha \circ \varphi, \varphi'}$ means that 
		$\alpha \circ \varphi = \varphi'\circ \sigma$, which gives 
		\begin{equation}\label{E:bet}
		f'\circ \alpha\circ \varphi = f'\circ \varphi'\circ \sigma \ \hbox{ and } \ g'\circ \alpha\circ \varphi\circ \sigma^* = g'\circ \varphi'\circ \sigma\circ \sigma^*. 
		\end{equation} 
		Putting together \eqref{E:alp} and \eqref{E:bet}, we get 
		\[
		(f'\circ \varphi'\circ \sigma) \ R \ (g'\circ \varphi'\circ \sigma\circ \sigma^*),
		\]
		which implies 
		\[
		f'\circ \varphi' = (f'\circ \varphi'\circ \sigma\circ \sigma^{-1}) \ R \ (g'\circ \varphi'\circ \sigma\circ \sigma^*\circ \sigma^{-1}),
		\]
		as desired.

		We show \eqref{E:ddgy} assuming that $\alpha$ is a $\cK_{\textrm{p}}$-epimorphism, $\sigma^* \in {\rm Aut}(\bbK)$. First, we observe that 
		\begin{equation}\label{E:alm}
		 (f\circ \varphi) \ R \ (g \circ \varphi \circ \sigma^*)\;\Rightarrow \, (f'\circ\alpha \circ \varphi)\, R\, (g'\circ \alpha \circ \varphi \circ \sigma^*). 
		 \end{equation}
		 Indeed, given $x \in \bbK$ let  $a = \varphi(x)$,  $b = \varphi(\sigma^*(x))\in A$, so $f(a) R g(b)$. We find $r,s\in A$ and then $t\in A$, with
		\[
			f(a)= g(r),\, rRs,\,  g(s)= f(t),\, f(t)= g(b), 
			\]
			from which we get 
			\begin{equation}\notag
			f'(\alpha(a)) = g'(\alpha(r))\, R \,  g'(\alpha(s))  =  f'(\alpha(t))=  g'(\alpha(b)). 
			\end{equation}
			
			To see \eqref{E:ddgy}, we repeat the argument for \eqref{eegy}. This argument goes through since, by \eqref{E:alm}, we get \eqref{E:alp} assuming $\sigma^* \in B_{f,g;\varphi}^{(1)}$, that is, $(f\circ \varphi) \, R\, (g\circ \varphi\circ \sigma^*)$.
			
			Finally, towards \eqref{E:ffgy} assuming $\alpha$ is a $\cK_{\textrm{p}}$-epimorphism, note that
				\begin{equation} \label{E:alm2}
				(f\circ \varphi) \ = \ (g \circ \varphi \circ \sigma^*)\;\Rightarrow \, (f'\circ\alpha \circ \varphi)\, = \, (g'\circ \alpha \circ \varphi \circ \sigma^*). 
			\end{equation}
			Therefore, if $\sigma^* \in B_{f,g;\varphi}^{(0)}$, so  $f\circ \varphi \ =\ g\circ \varphi\circ \sigma^*$, then
			\begin{equation} \label{E:alm3} (f'\circ\alpha \circ \varphi)\, = \, (g'\circ \alpha \circ \varphi \circ \sigma^*). \end{equation}
			Putting together \eqref{E:alm3} and \eqref{E:bet} the same way implies
			$$(f'\circ \varphi') \ = \ (g'\circ \varphi'\circ \sigma\circ \sigma^*\circ \sigma^{-1}),$$
			as desired.
		\end{PROOF}

		\begin{remark} 	\label{additional}
		With some additional work one can show that in the lemma above, if $\alpha$ is an approximate $\cK_{\textrm{p}}$-epimorphism, then 
		\begin{equation}\notag 
		\begin{split} 
 	\sigma B_{f,g;\varphi}^{(d)} \sigma^{-1} & \subseteq B^{(2d+2)}_{f',g';\varphi'},\; \hbox { for all }d \in \bbN,\hbox{ and }\\
		\sigma B_{f,g; \varphi}^{(d)} \sigma^{-1} & \subseteq B^{(2d+1)}_{f',g';\varphi'},\; \hbox{  if $d$ is even},
		\end{split} 
		\end{equation}
		and if $\alpha$ is a $\cK_{\textrm{p}}$-epimorphism, then
		$$\sigma B_{f, g;  \varphi}^{(d)} \sigma^{-1} \subseteq B^{(d)}_{f',g'; \varphi'}.$$
		\end{remark}

\begin{lemma}\label{L:pro}
		Given epimorphisms $\varphi_i, \psi_i\colon \bbK\to B_i$  with $i<n$, there exist objects 
		$(A,B,f_i,g_i)$ in $\cK_{\textrm{p}}$, for $i<n$, and $\varphi \in {\rm Epi}(\bbK, A)$, such that 
		\begin{equation}\notag
		B^{(1)}_{f_i, g_i; \varphi} \subseteq B^{(1)}_{\varphi_i, \psi_i}\;\hbox{ for all }i<n, 
		\end{equation}
		and
			\begin{equation}\notag
			B^{(0)}_{f_i, g_i; \varphi} \subseteq B^{(0)}_{\varphi_i, \psi_i}\;\hbox{ for all }i<n, 
		\end{equation}
	\end{lemma}
	
	\begin{PROOF}{Lemma~\ref{L:pro}} We start with a claim.
	
	\begin{claim} \label{pdensity}
		Given  $\gamma_i \in {\rm Epi}(\bbK, B)$, for $i<m$, there exist $A\in \cK$ and $h_i \in {\rm Epi}(A,B)$, and $\varphi \in {\rm Epi}(\bbK, A)$ such that
		\[
		\gamma_i = h_i \circ \varphi,\;\hbox{ for all }i<m. 
		\]	
	\end{claim}

	\begin{PROOF}{Claim \ref{pdensity}}
		We will use properties \ref{factor} and \ref{good} of $\bbK$. Consider the function 
		\[
		 \gamma_0\times \cdots \times \gamma_{m-1}: \bbK^m\to B^m , 
		\]
		which is obviously continuous. By property \ref{factor}, there are $A \in \cK$ and $\varphi \in {\rm Epi}(\bbK, A)$ such that $\gamma_0\times \cdots \times \gamma_{m-1}$ factors through $\varphi$, that is, for each $x\in \bbK$, the tuple $(\gamma_0(x),\dots, \gamma_{m-1}(x))$ depends only on $\varphi(x) \in A$.
		Let $h_i: A \to B$, for $i<m$,  be defined by the equalities 
		\[
		h_i(\varphi(x)) = \gamma_i(x), \;\hbox{ for all }i<m.
		\]		
		By property \ref{good} of $\bbK$, $h_0, \dots, h_{m-1}$ are epimorphisms. 
		\end{PROOF}

	We proceed to proving the conclusion of the lemma. Directly from the definition of the sets $B^{(0)}_{f_i, g_i; \varphi}$ and $B^{(0)}_{\varphi_i, \psi_i}$ (and that of the sets $B^{(1)}_{f_i, g_i; \varphi}$ and $B^{(1)}_{\varphi_i, \psi_i}$), one sees that it suffices to find $A, B\in \cK$, $\varphi\in {\rm Epi}(\bbK, A)$, $f_i, g_i\in {\rm Epi}(A, B)$, and $\beta_i\in {\rm Epi}(B, B_i)$ such that 
	\begin{equation}\label{E:bef}
	\beta_i\circ f_i\circ \varphi =\varphi_i\;\hbox{ and }\; \beta_i\circ g_i\circ \varphi = \psi_i, \hbox{ for all }i<n.
	\end{equation} 
	By the joint projection property for $\cK$, there exist $B \in \cK$ and $\beta_i \in {\rm Epi}(B,B_i)$ for $i<n$. Now, since $\bbK$ is the projective Fra{\"i}ss{\'e} limit of $\cK$, there exist $\gamma_i, \xi_i\in {\rm Epi}(\bbK, B)$ such that 
	\begin{equation}\label{E:pork}
	\beta_i \circ \gamma_i = \varphi_i\;\hbox{ and }\; \beta_i\circ \xi_i =\psi_i, \hbox{ for all }i<n. 
	\end{equation} 
	Apply Claim~\ref{pdensity} to the $m=2n$ epimorphisms $\gamma_i, \xi_i$, $i<n$, obtaining $A\in \cK$, $\varphi\in {\rm Epi}(\cK, A)$, and 
	$f_i, g_i\in {\rm Epi}(A, B)$ such that 
	\begin{equation}\label{E:opop}
	\gamma_i  = f_i\circ \varphi \;\hbox{ and }\; \xi_i= g_i\circ \varphi,\; \hbox{ for all }i<n. 
	\end{equation}
	Now \eqref{E:bef} is implied by \eqref{E:pork} and \eqref{E:opop}, and the lemma follows. 
	\end{PROOF}

			\begin{PROOF}[Proof of Theorem \ref{mcor}]{Theorem \ref{mcor}} 
			We first consider \ref{ge}. Fix non-empty open sets $U, V \subseteq {\rm Homeo}(K)^n$. We need to show that 
			\begin{equation}\label{E:confi}
			\sigma\, \ov{\gamma}\,\sigma^{-1}\in U,\; \hbox{ for some }\ov{\gamma}\in V \hbox{ and } \sigma\in {\rm Aut}(\bbK). 
			\end{equation}
			First, by our assumptions, 	$V \cap {\rm Aut}(\bbK) \neq \emptyset \neq  V \cap {\rm Aut}(\bbK)$.
		
		By  Corollary~\ref{topobs}. and Lemma~\ref{L:pro}, we can assume that for some $(A,B,f_i,g_i) \in \cK_{\textrm{p}}$, $i<n$, and $\varphi \in {\rm Epi}(\bbK, A)$, we have
		\[
		V \cap {\rm Aut}(\bbK) = \prod_{i<n} B^{(1)}_{f_i, g_i; \varphi},
		\]
		and, similarly, for some $(A',B',f'_i,g'_i) \in \cK_{\textrm{p}}$, $i<n$, and $\varphi' \in {\rm Epi}(\bbK, A')$, 
		\[
		U \cap {\rm Aut}(\bbK)=\prod_{i<n} B^{(1)}_{f'_i, g'_i; \varphi'}.
		\]
		By the approximate JPP, we obtain $A^+,B^+ \in \cK$, $f^+_i,g^+_i\in {\rm Epi}(A^+, B^+)$, $i <n$, and $\cK$-epimorphisms $\alpha: A^+ \to A$, $\alpha': A^+ \to A'$, such that, for each $i<n$, 
		$\alpha$ and $\alpha'$ are approximate $\cK_{\textrm{p}}$-epimorphisms from $(A^+,B^+,f_i^+,g_i^+)$ to $(A,B,f_i,g_i)$ and to $(A',B',f'_i,g'_i)$, respectively. 
					
		Pick $\varphi^+ \in {\rm Epi}(\bbK, A^+)$ such that $\varphi = \alpha \circ \varphi^+$.
		Then, from the definitions of the two sets in \eqref{E:1conf} and from $\alpha$ being an approximate $\cK_p$-epimorphism, we get 
		\begin{equation}\label{E:1conf}
			B_{f^+_i,g^+_i; \varphi^+}^{(0)} \subseteq B^{(1)}_{f_i,g_i; \varphi},\;\hbox{ for all }i<n.
		\end{equation} 
		Pick $\sigma \in {\rm Aut}(\bbK)$ with $\sigma \in B_{\alpha' \circ \varphi^+, \varphi'}$. By Lemma \ref{conju}, we have 
		\begin{equation}\label{E:2conf}
		 \sigma (B^{(0)}_{f^+_i, g^+_i; \varphi^+ }) \sigma^{-1} \subseteq B^{(1)}_{f'_i, g_i'; \varphi'}, \;\hbox{ for all }i<n. 
		\end{equation}
		Now, \eqref{E:confi} follows from \eqref{E:1conf} and \eqref{E:2conf} since the set $B^{(0)}_{f^+_i, g^+_i; \varphi^+ }$ is non-empty. 
		
		We turn to \ref{HalfAppr}.
		Fix a non-empty open set  $U \subseteq {\rm Homeo}(K)^n$. We can again assume that there exist $(A',B',f'_i,g'_i) \in \cK_{\textrm{p}}$, $i<n$, and $\varphi' \in {\rm Epi}(\bbK, A')$ with
		\[
		U \cap {\rm Aut}(\bbK)=\prod_{i<n} B^{(1)}_{f'_i, g'_i; \varphi'}.
		\]
		
		Fix a non-empty open set $V \subseteq {\rm Aut}(\bbK)$. By  Corollary~\ref{topobs} and Lemma~\ref{L:pro}, we can assume that for some $(A,B,f_i,g_i) \in \cK_{\textrm{p}}$, $i<n$, and $\varphi \in {\rm Epi}(\bbK, A)$, we have
		\[
		V  = \prod_{i<n} B^{(0)}_{f_i, g_i; \varphi}.
		\]
		
		So again it remains to show that 
		\begin{equation}\label{E:confi2}
			\sigma\, \ov{\gamma}\,\sigma^{-1}\in U,\; \hbox{ for some }\ov{\gamma}\in V \hbox{ and } \sigma\in {\rm Aut}(\bbK). 
		\end{equation}
		By the half-approximate JPP, we obtain $A^+,B^+ \in \cK$, $f^+_i,g^+_i\in {\rm Epi}(A^+, B^+)$, $i <n$, and $\cK$-epimorphisms $\alpha: A^+ \to A$, $\alpha': A^+ \to A'$ such that, for each $i<n$, 
		$\alpha$ is a $\cK_{\textrm{p}}$-epimorphism from $(A^+,B^+,f_i^+,g_i^+)$ to $(A,B,f_i,g_i)$ and $\alpha'$ is an approximate $\cK_{\textrm{p}}$-epimorphism to $(A',B',f'_i,g'_i)$.
		
			If $\varphi^+ \in {\rm Epi}(\bbK, A^+)$ is such that $\varphi = \alpha \circ \varphi^+$, then we can proceed as in the proof of \ref{ge}, so
		$$B_{f^+_i,g^+_i; \varphi^+}^{(0)} \subseteq B^{(0)}_{f_i,g_i; \varphi},\;\hbox{ for all }i<n, $$
			and if  $\sigma \in B_{\alpha' \circ \varphi^+, \varphi'}$, then 
				$$\sigma (B^{(0)}_{f^+_i, g^+_i; \varphi^+ }) \sigma^{-1} \subseteq B^{(1)}_{f'_i, g_i'; \varphi'}, \;\hbox{ for all }i<n.  $$
		
		The proof of \ref{KR} is the same. We start from
			\[
		U=\prod_{i<n} B^{(0)}_{f'_i, g'_i; \varphi'},
		\]
		and
			\[
		V  = \prod_{i<n} B^{(0)}_{f_i, g_i; \varphi},
		\]
		 for each $i<n$, 
		$\alpha$ is a $\cK_{\textrm{p}}$-epimorphism from $(A^+,B^+,f_i^+,g_i^+)$ to $(A,B,f_i,g_i)$ as well as $\alpha'$ is a $\cK_{\textrm{p}}$-epimorphism to $(A',B',f'_i,g'_i)$.
			So if $\varphi^+ \in {\rm Epi}(\bbK, A^+)$ is such that $\varphi = \alpha \circ \varphi^+$,  then $$B_{f^+_i,g^+_i; \varphi^+}^{(0)} \subseteq B^{(0)}_{f_i,g_i; \varphi},\;\hbox{ for all }i<n, $$
		and if  $\sigma \in B_{\alpha' \circ \varphi^+, \varphi'}$, then 
		$$\sigma (B^{(0)}_{f^+_i, g^+_i; \varphi^+ }) \sigma^{-1} \subseteq B^{(0)}_{f'_i, g_i'; \varphi'}, \;\hbox{ for all }i<n.  $$
		\end{PROOF}

	The following corollary is an immediate consequence of Theorem~\ref{mcor}. 
	
\begin{corollary} \label{m2cor}
	Suppose that $\cK$ is a transitive projective Fra{\"i}ssé class with the property that ${\rm Aut(\bbK)}$ has a dense image under the canonical homomorphism into ${\rm Homeo}(K)$, let $n \in \bbN$. 
	\begin{enumerate}[label = $(\roman*)$, ref = $(\roman*)$]
		\item \label{geGD} If $\cK^{\times n}_{\textrm{p}}$ has the approximate JPP, then  
		\[
		\{ \overline{\gamma} \in {\rm Homeo}(K)^n: \ \overline{\gamma}^{{\rm Aut}(\bbK)} \textrm{ is dense} \} \ \textrm{ is dense }G_\delta \textrm  { in  } {\rm Homeo}(K)^n.
		\]
		
		\item \label{HalfApprGD} If $\cK^{\times n}_{\textrm{p}}$ has the half-approximate JPP, 
		then	\[
		\{ \overline{\gamma} \in {\rm Aut}(\bbK)^n: \ \overline{\gamma}^{{\rm Aut}(\bbK)} \textrm{ is dense in }{\rm Homeo}(K)^n \} \ \textrm{ is dense }G_\delta \textrm  { in  } {\rm Aut}(\bbK)^n.
		\]
		
		\item \label{KRGD}	If $\cK^{\times n}_{\textrm{p}}$ has the JPP, then 
	\[
	\{ \overline{\gamma} \in {\rm Aut}(\bbK)^n: \ \overline{\gamma}^{{\rm Aut}(\bbK)} \textrm{ is dense in }{\rm Aut}(\bbK)^n \} \ \textrm{ is dense }G_\delta \textrm  { in  } {\rm Aut}(\bbK)^n.
	\]
	\end{enumerate}
\end{corollary}	

 Corollary~\ref{m2cor} implies the following.

	\begin{corollary} \label{m2cor2}
	Suppose that $\cK$ is a transitive projective Fra{\"i}ssé class with the property that ${\rm Aut(\bbK)}$ has a dense image under the canonical homomorphism into ${\rm Homeo}(K)$. 
	\begin{enumerate}[label = $(\roman*)$, ref = $(\roman*)$]
		\item \label{geGD2} If $\cK^{\times n}_{\textrm{p}}$ has the approximate JPP for each $n\in \bbN$, then  
		\[
		\{ \overline{\gamma} \in {\rm Homeo}(K)^\bbN: \ \overline{\gamma}^{{\rm Aut}(\bbK)} \textrm{ is dense} \} \ \textrm{ is dense }G_\delta \textrm  { in  } {\rm Homeo}(K)^\bbN.
		\]
		
		\item \label{HalfApprGD2} If $\cK^{\times n}_{\textrm{p}}$ has the half-approximate JPP for each $n\in \bbN$, 
		then	\[
		\{ \overline{\gamma} \in {\rm Aut}(\bbK)^\bbN: \ \overline{\gamma}^{{\rm Aut}(\bbK)} \textrm{ is dense in }{\rm Homeo}(K)^\bbN \} \ \textrm{ is dense }G_\delta \textrm  { in  } {\rm Aut}(\bbK)^\bbN.
		\]
		
		\item \label{KRGD2}	If $\cK^{\times n}_{\textrm{p}}$ has the JPP for each $n\in \bbN$, then 
	\[
	\{ \overline{\gamma} \in {\rm Aut}(\bbK)^\bbN: \ \overline{\gamma}^{{\rm Aut}(\bbK)} \textrm{ is dense in }{\rm Aut}(\bbK)^\bbN \} \ \textrm{ is dense }G_\delta \textrm  { in  } {\rm Aut}(\bbK)^\bbN.
	\]
	\end{enumerate}
\end{corollary}	

\begin{PROOF}{Corollary~\ref{m2cor2}} We write the proof only for (ii), as (i) and (iii) follow by the same argument. 
Note that $\ov{\gamma} = (\gamma_i)_{i\in \bbN}\in {\rm Aut}(\bbK)^\bbN$ is such that $\overline{\gamma}^{{\rm Aut}(\bbK)}$ is dense in ${\rm Homeo}(K)^\bbN$ precisely when for each $n\in \bbN$, $n\geq 1$, the orbit of the finite tuple $(\gamma_i)_{i<n}$ under the diagonal conjugacy action of ${\rm Aut}(\bbK)$ is dense in ${\rm Homeo}(K)^n$. Now, the conclusion of (ii) follows from Corollary~\ref{m2cor} (ii) since if $G\subseteq {\rm Aut}(\bbK)^n$ is a dense $G_\delta$ in ${\rm Aut}(\bbK)^n$, then the set 
\[
\{ (\gamma_i)_{i\in \bbN}\in {\rm Aut}(\bbK)^\bbN \mid  (\gamma_i)_{i<n}\in G\}
\]
is a dense $G_\delta$ in ${\rm Aut}(\bbK)^\bbN$. 
\end{PROOF}

	Now we state and prove the converse to Theorem~\ref{mcor}.
.

	\begin{theorem} \label{reverse}
			Suppose that $\cK$ is a transitive projective Fra{\"i}ssé class with the property that  the image of ${\rm Aut(\bbK)}$ under the canonical homomorphism $\rm pr$ is dense in ${\rm Homeo}(K)$. Let $n \in \bbN$.
			 
				\begin{enumerate}[label = $(\roman*)$, ref = $(\roman*)$]
				\item \label{bd1} If there exists $\bar \gamma \in {\rm Homeo}(K)^n$ such that $\bar \gamma^{{\rm Homeo}(K)}$ is dense in ${\rm Homeo}(K)^n$, then  $\cK^{\times n}_{\textrm{p}}$ has the approximate JPP.
				\item  \label{bd2} If the set of all $\bar \gamma \in {\rm Aut}(\bbK)^n$ such that $\bar \gamma^{{\rm Homeo}(K)}$ is dense in ${\rm Homeo}(K)^n$ is comeager in ${\rm Aut}(\bbK)^n$, then  $\cK^{\times n}_{\textrm{p}}$ has the half-approximate JPP.
				\item \label{KRR} If there exists $\bar \gamma \in {\rm Aut}(\bbK)^n$ such that $\bar \gamma^{{\rm Aut}(\bbK)}$ is dense in ${\rm Aut}(\bbK)^n$, then  $\cK^{\times n}_{\textrm{p}}$ has the JPP.
				\end{enumerate}
	\end{theorem}
	\begin{PROOF}{Theorem \ref{reverse}}
		For simplicity we assume that $n=1$. First we consider \ref{bd1}.
		
		Fix $(A,B,f,g)$, $(A',B',f',g') \in \cK_{\mathrm{p}}$. Pick $\varphi \in {\rm Epi}(\bbK, A)$, $\varphi' \in {\rm Epi}(\bbK, A')$. By Corollary \ref{cortop} $B^{(1)}_{f,g; \varphi} \subseteq {\rm Aut}(\bbK)$ has non-empty interior with respect to the topology inherited from ${\rm Homeo}(K)$, so for some nonempty open set $U \subseteq {\rm Homeo}(K)$ we have 
		$$U \cap {\rm Aut}(\bbK) \subseteq B^{(1)}_{f,g; \varphi}.$$
		Similarly, for some non-empty open set $U'$ in ${\rm Homeo}(K)$ we have
		\begin{equation} \label{U'}  U' \cap {\rm Aut}(\bbK) \subseteq B^{(1)}_{f',g'; \varphi'}. \end{equation}

		By density of Aut$(\bbK)$ in ${\rm Homeo}(K)$  and continuity we can infer the following.
		\begin{claim} \label{automs}
			There exist $\sigma, \sigma', \tau \in $ Aut$(\bbK)$,  such that
			\begin{equation} \label{sigphigam} \tau \sigma \tau^{-1} = \sigma', \end{equation}
			moreover,
			\begin{equation} \label{phif} \sigma \in B^{(1)}_{f,g; \varphi}, \end{equation}
			\begin{equation} \label{gam} \sigma' \in B^{(1)}_{f',g';\varphi'}. \end{equation}
		\end{claim}
		\begin{PROOF}{Claim \ref{automs}}
			By our assumptions there are $\varsigma \in U$, $\varsigma' \in U'$ that are conjugate to each other, that is, for some $\tau_0 \in {\rm Homeo}(K)$
			$$ \tau_0 \varsigma \tau_0^{-1} = \varsigma' \in U'.$$
			By continuity of the composition and inverse, taking $\tau \in {\rm Aut}(\bbK)$ close enough to $\tau_0$, $\sigma \in  {\rm Aut}(\bbK)$ close enough to $\varsigma$ (and recalling \eqref{U'}) we will have
			$$ \tau \sigma \tau^{-1} \in U' \cap {\rm Aut}(\bbK),$$
			therefore setting $\sigma ' = \tau \sigma \tau^{-1}$ works.
		\end{PROOF}

		We are going to find  $A^+,B^+ \in \cK$, $f^+,g^+ \in {\rm Epi}(A^+,B^+)$, $\varphi^+ \in {\rm Epi}(\bbK, A^+)$ such that
		$$\sigma \in B_{f^+,g^+; \varphi^+},$$
		and $\varphi$, $\varphi' \circ \tau$ factor through  $g^+ \circ \varphi^+$ (therefore they factor through $\varphi^+$, too). This can be done by picking $\psi^* \in {\rm Epi}(\bbK,C)$ such that  $(\varphi \times \varphi' \circ \tau)$ factors through $\psi^*$, and applying Lemma \ref{pdensity} to  $\varphi^* = \psi^* \circ \sigma$ and $\psi^*$.
		If $f^+ \circ \varphi^+ = \psi^* \circ \sigma^{-1}$ and $g^+ \circ \varphi^+ = \psi^*$, then as $\varphi$, $\varphi' \circ \tau$ factor through $\psi^*$, that is, $\varphi = \delta \circ \psi^*$, $\varphi' \circ \tau = \delta' \circ \psi^*$ for some $\delta$ and $\delta'$, clearly
		\begin{equation} \label{gfactors}  \delta \circ g^+ \circ \varphi^+ = \delta \circ \psi^* = \varphi. \end{equation}
		\begin{equation} \label{gfactors'}  \delta' \circ g^+ \circ \varphi^+ = \delta' \circ \psi^* = \varphi' \circ \tau. \end{equation}
		Moreover, it is not difficult to see that that $B_{\varphi^*, \psi^*} = B_{\psi^* \circ \sigma, \psi^*}  \ni \sigma$.

		We let $\alpha = \delta \circ g^+: A^+ \to A$ so that $\varphi = \alpha \circ \varphi^+$, and let  $\alpha' =  \delta' \circ g^+: A^+ \to A'$, so that 
		$\varphi' \circ \tau = \alpha' \circ \varphi^+$. To finish the proof of the theorem, it suffices to show that $\alpha \in {\rm Epi}(A^+, A)$, $\alpha' \in  {\rm Epi}(A^+, A')$, moreover, they are approximate epimorphisms between the respective objects in $\cK_{\rm p}$. First, it is easy to check that $\alpha$ is a strong homomorphism, so by property \ref{good} $\alpha \in {\rm Epi}(A^+, A)$, and similarly, $\varphi' \circ \tau$ and $\varphi^+$ are epimorphisms, so are $\alpha'$.
			
			First we check that $\alpha'$ is an approximate epimorphism.
			Let $a_0,a_1 \in A^+$ with $f^+(a_0) = g^+(a_1)$, we need to argue that $f'(\alpha'(a_0)) \ R\  g'(\alpha'(a_1))$. Pick $x \in \bbK$ with $\varphi^+(\tau^{-1}(x)) = a_0$.
			
			\begin{claim} \label{scl}
				Suppose that $\tau \in {\rm Aut}(\bbK)$, $\delta'$, $g^+$, $f^+$, $\alpha'$ are $\cK$-epimorphisms,  $\varphi^+ \in {\rm Epi}_{\bbK}$ that satisfy
				$$\delta' \circ g^+ = \alpha',$$
				and $\sigma \in B_{f^+,g^+; \varphi^+}$. 
				If $a_0,a_1$ are such that $f^+(a_0) = g^+(a_1)$ and $\varphi^+(\tau^{-1}(x)) = a_0$,
				then
				 $$\alpha'(a_1) = \alpha'(\varphi^+(\sigma \tau^{-1}(x)).$$
			\end{claim}
			\begin{PROOF}{Claim \ref{scl}} Indeed, $\sigma \in B_{f^+,g^+; \varphi^+}$ implies $f^+(\varphi^+(\tau^{-1}(x))) = g^+(\varphi^+(\sigma(\tau^{-1}(x))))$, using $\varphi^+(\tau^{-1}(x)) = a_0$ we get 
			$$ g^+(\varphi^+(\sigma(\tau^{-1}(x)))) = f^+(a_0) = g^+(a_1).$$
			So  applying $\delta'$, and using $\alpha' = \delta' \circ g^+$
		
			$$\alpha' \circ \varphi^+(\sigma \tau^{-1}(x)) =  \alpha'(a_1).$$
			\end{PROOF}
		 So pick $x \in \bbK$ with $\varphi^+(\tau^{-1}(x)) = a_0$ and note that by the claim the condition $f'(\alpha'(a_0)) \ R \ g'(\alpha'(a_1))$ is equivalent to 
			\[
			f'(\alpha'(\varphi^+(\tau^{-1}(x)))) \ R \ g'(\alpha'(\varphi^+(\sigma \tau^{-1}(x)))), 
			\]
			so we need to check this for all $x$, that is,
			\[
			 (f' \circ \alpha' \circ \varphi^+ \circ \tau^{-1}) \ R \ (g' \circ \alpha' \circ \varphi^+ \circ \sigma \tau^{-1}).
			\]
			But using $\alpha' \circ \varphi^+ = \varphi' \circ \tau$, this is equivalent to
			$$ (f' \circ \varphi' \circ \tau \circ \tau^{-1} \ R\  g' \circ \varphi' \circ \tau \circ \sigma \circ \tau^{-1}).$$
			We obtained that
			\begin{equation} \label{pre1}
				(f' \circ \varphi' \circ \tau \circ \tau^{-1} \ R\  g' \circ \varphi' \circ \tau \circ \sigma \circ \tau^{-1}) \ \Rightarrow \ \alpha' \textrm{ is an approximate epimorphism} \end{equation}.
			 But the premise is true, since $\sigma' =  \tau \sigma \tau^{-1} \in B^{(1)}_{f',g';\varphi'}$, we are done. 	
			 It remains to show that $\alpha$ is an approximate epimorphism.
			 \begin{claim} \label{scl2}
			 		Suppose that $\delta$, $g^+$, $f^+$, $\alpha$ are $\cK$-epimorphisms,  $\varphi^+ \in {\rm Epi}_{\bbK}$ that satisfy
			 	$$\delta \circ g^+ = \alpha,$$
			 	and $\sigma \in B_{f^+,g^+; \varphi^+}$. 
			 	If $a_0,a_1$ are such that $f^+(a_0) = g^+(a_1)$ and $\varphi^+(x) = a_0$,
			 	then
			 	$$\alpha(a_1) = \alpha(\varphi^+(\sigma (x)).$$ 
			 \end{claim}
			 \begin{PROOF}{Claim \ref{scl2}} Indeed, $\sigma \in B_{f^+,g^+; \varphi^+}$ implies $f^+(\varphi^+(x)) = g^+(\varphi^+(\sigma(x)))$, using $\varphi^+(x) = a_0$ we get 
			 	$$ g^+(\varphi^+(\sigma(x))) = f^+(a_0) = g^+(a_1).$$
			 	Now  applying $\delta$, and using $\alpha = \delta \circ g^+$
			 	
			 	$$\alpha \circ \varphi^+(\sigma(x)) = \alpha(a_1).$$
			 \end{PROOF}
			 
			 Let $a_0,a_1 \in A^+$ be such that $f^+(a_0) = g^+(a_1)$. We need to argue that we have $f(\alpha(a_0)) \ R\  g(\alpha(a_1))$. Pick $x \in \bbK$ with $\varphi^+((x)) = a_0$.
			 
			 The claim gives that $f(\alpha(a_0)) \ R \ g(\alpha(a_1))$ is equivalent to 
			 \[
			 f(\alpha(\varphi^+(x))) \ R \ g(\alpha(\varphi^+(\sigma (x))), 
			 \]
			 so we need to verify that this holds for all $x$, that is,
			 \[
			 (f \circ \alpha \circ \varphi^+) \ R \ (g \circ \alpha \circ \varphi^+ \circ \sigma ).
			 \]
			 But using $\alpha \circ \varphi^+ = \varphi$, this is equivalent to
			 $$ (f \circ \varphi) \ R \  (g \circ \varphi \circ \sigma ).$$
			 So we got that
			 \begin{equation} \label{pre2}
			 	(f \circ \varphi) \ R \  (g \circ \varphi \circ \sigma ) \ \Rightarrow \ \alpha \textrm{ is an approximate } \cK_{\rm p}\textrm{-epimorphism.}
			 \end{equation}
			 But the premise is true, since $\sigma \in B^{(1)}_{f,g;\varphi}$, we are done.

			 Now we sketch the proof of \ref{bd2}.
			 
			 Fix again $(A,B,f,g)$, $(A',B',f',g') \in \cK_{\mathrm{p}}$, and pick $\varphi \in {\rm Epi}(\bbK, A)$, $\varphi' \in {\rm Epi}(\bbK, A')$. By Corollary \ref{cortop} for some nonempty open set $U \subseteq {\rm Aut}(\bbK)$ we have 
			 $$U  \subseteq B^{(0)}_{f,g; \varphi},$$
			 and for some non-empty open set $U'$ in ${\rm Homeo}(K)$ we have
			 \begin{equation} \label{U''}  U' \cap {\rm Aut}(\bbK) \subseteq B^{(1)}_{f',g'; \varphi'}. \end{equation}
			 	Then we find $\sigma, \sigma', \tau \in $ Aut$(\bbK)$,  such that
			 \begin{itemize}
			 	\item $\tau \sigma \tau^{-1} = \sigma'$,
			 	\item $ \sigma \in B^{(0)}_{f,g; \varphi}$,
			 	\item $\sigma' \in B^{(1)}_{f',g';\varphi'}$.
			 \end{itemize}
			Next we find  $A^+,B^+ \in \cK$, $f^+,g^+ \in {\rm Epi}(A^+,B^+)$, $\varphi^+ \in {\rm Epi}(\bbK, A^+)$ such that
			$$\sigma \in B_{f^+,g^+; \varphi^+},$$
			and $\delta$ and $\delta'$ with
			\begin{equation} \label{gfactors+}  \delta \circ g^+ \circ \varphi^+ =  \varphi. \end{equation}
			\begin{equation} \label{gfactors'+}  \delta' \circ g^+ \circ \varphi^+  = \varphi' \circ \tau. \end{equation}
				We let $\alpha = \delta \circ g^+: A^+ \to A$ so that $\varphi = \alpha \circ \varphi^+$, and let  $\alpha' =  \delta' \circ g^+: A^+ \to A'$, so that 
			$\varphi' \circ \tau = \alpha' \circ \varphi^+$.
			 The same argument as above shows that $\alpha'$ is an approximate epimorphism between $(A^+,B^+,f^+,g^+)$ and $(A',B',f',g')$.
			 
			 It remains to show that $\alpha$ is an epimorphism between $(A^+,B^+,f^+,g^+)$ and $(A,B,f,g)$. This case Claim \ref{scl2} still applies, and modifying the argument following the claim one can get
			  \begin{equation} \label{pre4}
			 	(f \circ \varphi) \ = \  (g \circ \varphi \circ \sigma ) \ \Rightarrow \ \alpha \textrm{ is a } \cK_{\rm p}\textrm{-epimorphism.}
			 \end{equation}
			 	 But the premise is true, since $\sigma \in B^{(0)}_{f,g;\varphi}$, we are done. 
			 
			 To get \ref{KRR} we need to make minor changes to the proof of \ref{bd2}.
			 We have $$U  \subseteq B^{(0)}_{f,g; \varphi},$$
			 and 
			 \begin{equation}\notag 
			 U' =  B^{(0)}_{f',g'; \varphi'}, \end{equation}
			 and then  we find $\sigma, \sigma', \tau \in $ Aut$(\bbK)$,  such that
			 \begin{itemize}
			 	\item $\tau \sigma \tau^{-1} = \sigma'$,
			 	\item $ \sigma \in B^{(0)}_{f,g; \varphi}$,
			 	\item $\sigma' \in B^{(0)}_{f',g';\varphi'}$.
			 \end{itemize}
			 Constructing again $\alpha$ and $\alpha'$, one shows \eqref{pre4}, and
		$$ 
			 	(f' \circ \varphi' \circ \tau \circ \tau^{-1} \ = \  g' \circ \varphi' \circ \tau \circ \sigma \circ \tau^{-1}) \ \Rightarrow \ \alpha' \textrm{ is a } \cK_{\rm p}\textrm{-epimorphism}. $$

	\end{PROOF}

	\subsection{Comeager conjugacy class in ${\rm Homeo}(K)$}
	The following definition and theorem are the dualized versions of the corresponding notions in \cite{KR}.
	\begin{definition}
		Let $\cK$ be a projective Fra\"{i}ssé class.
		We say that $\cK_{\rm p}$ satisfies the {\em Weak Projective Amalgamation Property} (in short, WPAP) if for each $(A,B,f,g) \in \cK_{\rm p}$ there exist  $(A^+,B^+,f^+,g^+) \in \cK_{\rm p}$ and an epimorphism 
		$$\alpha: \ (A^+,B^+,f^+,g^+) \to (A,B,f,g)$$
		such that for every $(A',B',f',g'),(A'',B'',f'',g'') \in \cK$ and epimorphisms
		\[
		\beta' \colon (A',B', f', g',)\to (A^+,B^+, f^+, g^+)\;\hbox{ and }\;\beta'' \colon (A'',B'', f'', g''\,)\to(A^+,B^+, f^+, g^+)
		\]
		there exist $(A^*,B^*,f^*,g^*) \in \cK$ and epimorphisms
		\[
		\gamma' \colon (A^*,B^*, f^*, g^*,)\to (A',B', f', g')\;\hbox{ and }\;\gamma'' \colon (A^*,B^*, f^*, g^*\,)\to(A'',B'', f'', g'')
		\]
		such that 
		$$\alpha \circ \beta' \circ \gamma' = \alpha \circ \beta'' \circ \gamma'' \ \textrm{(as mappings in }{\rm Epi}(A^*,A){\rm)}.$$
	\end{definition}
	The interested reader may dualize the proof \cite[Theorem 3.4]{KR} to obtain the following.
	\begin{theorem}
		Suppose that $\cK$ is a projective Fra\"issé class, let $\bbK$ be its limit. Then the following are equivalent
		\begin{enumerate}[label = $(\roman*)$, ref= $(\roman*)$ ]
			\item There exists a comeager conjugacy class in ${\rm Aut}(\bbK)$.
			\item $\cK_{\rm p}$ has the JPP and WPAP.
		\end{enumerate}
	\end{theorem}
	
	We remark that the above equivalence can be adapted to the approximate setting to give an equivalent condition for the existence of the generic homeomorphism.	
		\begin{definition}
		Let $\cK$ be a projective Fra\"{i}ssé class.
		We say that $\cK_{\rm p}$ satisfies the {\em approximate}  WPAP if for each $(A,B,f,g) \in \cK_{\rm p}$ there exist  $(A^+,B^+,f^+,g^+) \in \cK_{\rm p}$ and an {\rm approximate} epimorphism 
		$$\alpha: \ (A^+,B^+,f^+,g^+) \to (A,B,f,g)$$
		such that for every $(A',B',f',g'),(A'',B'',f'',g'') \in \cK$ and  {\rm approximate} epimorphisms
		\[
		\beta' \colon (A',B', f', g',)\to (A^+,B^+, f^+, g^+)\;\hbox{ and }\;\beta'' \colon (A'',B'', f'', g''\,)\to(A^+,B^+, f^+, g^+)
		\]
		there exist $(A^*,B^*,f^*,g^*) \in \cK$ and {\rm approximate} epimorphisms
		\[
		\gamma' \colon (A^*,B^*, f^*, g^*,)\to (A',B', f', g')\;\hbox{ and }\;\gamma'' \colon (A^*,B^*, f^*, g^*\,)\to(A'',B'', f'', g'')
		\]
		such that 
		$$(\alpha \circ \beta' \circ \gamma') \ R \ (\alpha \circ \beta'' \circ \gamma'') \ \textrm{(in }{\rm Epi}(A^*,A){\rm)}.$$
	\end{definition}
	Using the proof of \cite[Theorem 3.4]{KR} coupled with our ideas it is possible to show the following.
	\begin{theorem}
		Suppose that $\cK$ is a transitive projective Fra\"issé class, let $K = \bbK/R^\bbK$ be the quotient compactum, and assume that ${\rm Aut}(\bbK)$ is dense in ${\rm Homeo}(K)$.
		Then the following are equivalent
		\begin{enumerate}[label = $(\roman*)$, ref= $(\roman*)$ ]
			\item There exists a comeager conjugacy class in ${\rm Homeo}(K)$.  
			\item $\cK_{\rm p}$ has the approximate JPP and the approximate WPAP.
		\end{enumerate}
	\end{theorem}

	\begin{remark}
		To make the above definition meaningful, we point out two facts in regard to composing approximate epimorphisms.  First, it is not difficult to see that Remark \ref{additional} implies that if $\delta$ is the composition of two approximate epimorphism between $(A',B',f',g')$ and $(A,B,f,g)$, then $f'(a_0) = g'(a_1)$ implies $f(\delta(a_0)) R^4 g(\delta(a_1))$.
		
		Second, in the above theorem one can also require that the approximate epimorphism $\alpha: \ (A^+,B^+,f^+,g^+) \to (A,B,f,g)$ satisfy the stronger condition
		$$ (f^+(a_0) R^4 g^+(a_1)) \ \Rightarrow \ (f(\alpha(a_0)) R g(\alpha(a_1)),$$
		therefore $(\alpha \circ \beta' \circ \gamma')$ and $(\alpha \circ \beta'' \circ \gamma'')$ are approximate epimorphisms.
	\end{remark}

	\section{Dense conjugacy classes in products---the pseudoarc case}\label{S:psco} 
	
	We refer the reader to Section~\ref{S:intro} for the definition and relevant properties of $\cP$---the category of finite reflexive linear graphs. By $\bbP$ we denote the projective Fra\"issé limit of $\cP$.
	In this case, the canonical continuous homomorphism ${\rm pr}: {\rm Aut}(\bbP) \to {\rm Homeo}(P)$ is injective as the class $\cP$ obviously fulfills condition (iv) of Proposition~\ref{P:small}; so it fulfills the conclusions of Proposition~\ref{P:inj}.
We identify ${\rm Aut}(\bbP)$ with its image, that is, 
	$$ {\rm Aut}(\bbP) \leq \textrm{Homeo}(P).$$

\subsection{Density in ${\rm Homeo}(\bbP)^\bbN$} 

	The main theorem in this section is the following application of Theorem \ref{mcor}. Note that it implies Theorem~\ref{T:homeoprec} from the introduction. 
	
	\begin{theorem}\label{S3mainT} {}\
	\begin{enumerate}
	\item[(i)]	 For every $n \in \mathbb{N}$, the set of all $\bar \gamma \in {\rm Aut}(\bbP)^n$ such that the orbit of $\bar \gamma$ under the diagonal conjugacy action by ${\rm Aut}(\bbP)$ is dense in ${\rm Homeo}(P)^n$ is comeager in ${\rm Aut}(\bbP)^n$. 
		 
	\item[(ii)]	 The set of all $\bar \gamma \in {\rm Aut}(\bbP)^\bbN$ such that the orbit of $\bar \gamma$ under the diagonal conjugacy action by ${\rm Aut}(\bbP)$ is dense in ${\rm Homeo}(P)^\bbN$ is comeager in ${\rm Aut}(\bbP)^\bbN$. 
	\end{enumerate}
	\end{theorem}
	
	 The key to the proof of the theorem above is the following lemma. 
	 
	\begin{lemma} \label{JPPmod}
		Suppose that  $A,B,A',B' \in \cP$, 
		there exist $A^+,B^+\in \cP$ and $\alpha\in {\rm Epi}(A^+,A)$, $\alpha'\in {\rm Epi}(A^+, A')$, $\beta\in {\rm Epi}(B^+,B)$, and $\beta'\in {\rm Epi}(B^+,B')$ with the following property:
		
		\noindent for all $f\in {\rm Epi}(A,B)$, $f'\in {\rm Epi}(A',B')$, there exists $f^+\in {\rm Epi}(A^+, B^+)$ such that   
		 \begin{equation}\notag
		 	f\circ \alpha = \beta \circ f^+  \ \hbox{ and }\ ( f'\circ \alpha')  \ R\ ( \beta' \circ f^+).
		 \end{equation}
		\end{lemma}

%
%
%
%

	The following corollary is an immediate consequence of Lemma~\ref{JPPmod}. 
	
	\begin{corollary} \label{JPP}
		Suppose that  $A,B,A',B' \in \cP$, $f_i,g_i,f_i',g'_i$, $i = 1, \ldots, n$, are such that $(A,B,f_i,g_i),$ $(A',B', f_i',g_i') \in \cP_{\textrm{p}}$.
		Then, for some $A^+,B^+$, $f_i^+, g_i^+$, with $1\leq i \leq n$, and $\alpha: A^+ \to A$, $\alpha': A^+ \to A'$, we have that, for each $1\leq i\leq n$,
		\begin{itemize}
			\item the map $\alpha$ is an approximate epimorphism between $(A^+,B^+,f_i^+,g_i^+)$ and $(A,B,f_i,g_i)$,
			\item  the map $\alpha'$ is an epimorphism between $(A^+,B^+,f_i^+,g_i^+)$ and $(A',B',f'_i,g'_i)$.
		\end{itemize}
	\end{corollary}

	\begin{PROOF}[Proof of Lemma~\ref{JPPmod}]{Lemma \ref{JPPmod}}	
	By possibly extending $A$ with an exact epimorphism, we can assume that
	 \begin{enumerate}[label = $(*)_{\arabic*}$, ref = $(*)_{\arabic*}$]
	 	\item \label{fcond}   $(f')^{-1}(b)$ is the union of intervals each is of length at least $2$, 
	 	\end{enumerate} 
	 since exact epimorphisms are closed under composition.

		Let 
	\[
	A= \{ 0, \dots, m\}, \; A' = \{ 0, \dots, m'\},\; B= \{ 0, \dots, n\}, \; B'= \{ 0, \dots, n'\}, 
	\]
	we define $A^+$ with $R^{A^+}$. 
	Let 
	\[
	\begin{split} 
		A'_i &= \{ (i, j)\mid j\in A'\} \; \hbox{ for } i\in A,
	\end{split} 
	\]
	and $A^+ = \bigcup_{i \in A} A'_i$.
	We define $R^{A^+}$ on these sets to be the smallest reflexive and symmetric relation with 
	\[
	(i,j) \,R^{A^+} (i,j') \Leftrightarrow |j-j'|\leq 1, \;\hbox{ for }i\in A, 
	\]
	(so that with these definitions each $A_i'$ is a reflexive linear graph with $m'+1$ many elements),
		and 
	\[
	(i, m')\, R^{A^+} (i+1, m'),\;\hbox{ for }i <m,\ i \hbox{ is even}.
	\]
		\[
	(i, 0)\, R^{A^+} (i+1, 0),\;\hbox{ for }i <m, \ i \hbox{ is odd}.
	\]
	Similarly, we let 
 	\[
 	\begin{split} 
 		B'_k &= \{ (k, l)\mid l\in B'\} \; \hbox{ for } k\in B,
 	\end{split} 
 	\]
 	and $B^+ = \bigcup_{i \in B} B'_i$,
 	 where we stipulate that $R^{B^+}$ satisfy
 	\[
 	(k,l) \,R^{B^+} (k,l') \Leftrightarrow |l-l'|\leq 1, \;\hbox{ for }k\in B, 
 	\]
 	and 
 	\[
 	(k, n')\, R^{B^+} (k+1, n'),\;\hbox{ for }k <n,\ k \hbox{ is even},
 	\]
 	\[
 	(k, 0)\, R^{B^+} (k+1, 0),\;\hbox{ for }k <n, \ k \hbox{ is odd}.
 	\]
 	
Now we let $\alpha\colon A^+\to A$ and $\alpha'\colon A^+\to A'$ be defined by  
 \[
 \alpha(i, j)=  i\;\hbox{ and }\;\alpha'(i,j)=j.
 \]
 Similarly,  $\beta\colon B^+\to B$ and $\beta'\colon B^+\to B'$ is defined by 
 \[
 \beta(k,  l)=  k\;\hbox{ and }\;\beta'(k,l)=l.
 \]
 It is immediate that $\alpha, \alpha', \beta, \beta'$ are epimorphisms.

		 Fix $f \in {\rm Epi}( A,B)$ and $f' \in {\rm Epi}(A',B')$.
		 We are going to construct $f^+ \in {\rm Epi}(A^+,B^+)$
		 such that
		 \begin{align} 
		 	\beta \circ f^+ \  & R \ \ f \circ \alpha, \label{fcoh}\\
		 	\beta' \circ f^+ \  &=\ f' \circ \alpha'.\label{f'R}
		 \end{align}
		
		Let $a^*_0 = (0,0)$, and for $0<i \in A$ define $a^*_i \in A'_i$ be the element adjacent to $A_{i-1}'$.
		We define $f^+$ by induction so that
		for each $i$
		\begin{align} 
			\beta \circ f^+(a^*_i) \  & = \ f \circ \alpha(a_i^*), \label{fa*} \\
			\beta \circ f^+(a) \  & R \ \ \beta \circ f^+(a_i^*), \ (a \in A_i') \label{faAi}  \\
			\beta' \circ f^+ \rest A_i' \  &=\ f' \circ \alpha' \rest A_i' .\label{f'}
		\end{align}
		 Observe that $f \circ \alpha \rest A_i'$ is constant, since $\alpha \rest A_i'$ is constant for each $i \in A$, therefore \eqref{fa*} and \eqref{faAi} together imply \eqref{fcoh}.
	
		 Recall that $a_0^* = (0,0)$, let 
		 $$f^+(a_0^*) = (f(0), f'(0)),$$
		 so clearly $\beta(f^+(a_0^*)) = f(0) = f(\alpha(a_0^*))$, similarly with $\beta'$, $\alpha'$.
		 
		 Suppose that $f^+ \rest A_0' \cup \ldots \cup A_{i-1}' \cup \{a_i^*\}$ has been defined with the relevant parts of \eqref{fa*}-\eqref{f'} being satisfied.
		 
		 Let $\varepsilon = f(i+1) - f(i) \in \{-1,0,1\}$.
		 If $\varepsilon = 0$, then $f \circ \alpha \rest A'_i \cup \{a_{i+1}^*\} \equiv f(i)= f(i+1)$, and we can define
		  $$f^+(a) = (f(i), f'(\alpha'(a))) \hbox{ for }a \in  A'_i \cup \{a_{i+1}^* \}.$$
		 
		 Otherwise, if $\varepsilon \neq 0$, then we distinguish eight cases depending on the parity of $f(i)$ and $i$.
		 Suppose for simplicity that $\varepsilon = 1$ (so $f(i+1)) = f(i)+1$) and $f(i)$, $i$ are even.
		 
		 By our hypothesis, $f^+(a_i^*) \in B'_{f(i)}$, therefore it suffices to define $f^+ \rest A_i \cup \{a_{i+1}^*\}$ so that 
		  \begin{equation} \label{cB'} \begin{array}{rl}
		  		f^+[A_i \cup \{a_{i+1}^*\}] & \subseteq B'_{f(i)} \cup B'_{f(i)+1}, \\
		  		f^+(a_{i+1}^*) & \in B'_{f(i)+1}
		  		\end{array}
		  	 \end{equation}
  		alongside the condition \eqref{f'}.
		  (This is because $\beta \rest B'_k \equiv k$).
		  Since $f(i)$ is even we note that 
		  \begin{equation} \label{related}  (f(i),n') \ R^{B^+} \ (f(i)+1,n'). \end{equation}
		  
		  Now $a_i^*$ is the endpoint of the finite linear graph $A_i' = \{(i,j): \ j \in A'\}$ that is connected with an element of $A'_{i-1}$,
		  and since $i$ is even, we have $a_i^* = (i,0)$.
		   Let $j_0 \in A'$ be the smallest such that $f'(j_0) = n'$, so that by \ref{fcond} 
		   $$f'(j_0+1) = n'.$$
		   Letting
			\[
			f^+(i, j) = 
			\begin{cases} 
				(f(i),  f'(j)), &\hbox{ for }\;  j \leq j_0;\\
				(f(i)+1, f'(j)), &\hbox{ for }\; j > j_0.
			\end{cases}
			\]
		 and
		 $$f^+(a_{i+1}^*) = f^+(i+1,m') =  (f(i)+1, f'(m')) \in B'_{f(i)+1}$$
		 so by the choice of $j_0$ (and \eqref{related}) we get an epimorphism and \eqref{fa*} for $a_{i+1}^*$.
		It is also clear that $f^+(i,j), f^+(i+1,m') \in B'_{f(i)} \cup B'_{f(i)+1}$ (implying \eqref{cB'}, so that \eqref{faAi} follows). Moreover,	 
		 $\beta'(f^+(i,j)) = f'(j) = f'(\alpha'(i,j))$ for $j \in A'$ and
		 $\beta'(f^+(i+1,m')) = f'(m') = f'(\alpha'(i+1,m'))$, which implies \eqref{f'}.
		
	\end{PROOF}
	
	\begin{PROOF}[Proof of Theorem \ref{S3mainT}]{Theorem \ref{S3mainT}}
	 (i) follows from Corollary \ref{JPP} and Corollary \ref{m2cor}.
	 
	 (ii) follows from Corollary \ref{JPP} and Corollary \ref{m2cor2}.
	\end{PROOF}

	\subsection{Density in ${\rm Aut}(\bbP, \bullet)^\bbN$}\label{Su:densp}

	Consider the class $\cP_\bullet$ of structures in the language consisting of one binary symbol $R$ and a constant symbol $c$. The structures of $\cP_\bullet$ are finite sets with $R$ interpreted as a symmetric reflexive linear binary relation and the constant symbol $\bullet$ interpreted as one of the endpoints of the structure. So elements of $\cP_\bullet$ are isomorphic  to structures of the form	
	\[
	I= \{ 0, \dots, n\}\, \hbox{ with }\, xR^Iy \Leftrightarrow |x-y|\leq 1, \hbox{ for }x,y\in I, \hbox{ and }\bullet^I=0,
	\]
	where $n\in \bbN$. Note that the enumeration of $I$ is not part of the structure. 
	Morphisms in $\cP_\bullet$ are all epimorphisms among structures in $\cP_\bullet$. It follows from Lemma~\ref{amalgamp} that the family $\cP_\bullet$ forms a projective Fra{\"i}ss{\'e} class. Using the argument from \cite{IS} that $\cP$ is transitive, one checks that $\cP_\bullet$ is transitive as well. Let $\bbP_\bullet \in (\cP_\bullet)^*$ be its projective limit. We denote the underlying space of it by $\bbP$. On this space we have interpretations $R^\bbP$ and $\bullet^\bbP$ of $R$ and $\bullet$, respectively. Then it is not difficult to check that the reduct $(\bbP,R^\bbP)$ is the pre-pseudoarc, that is,  it is isomorphic to the Fra{\"i}ssé limit of $\cP$ (in the category $\cP^*$ derived from $\cP$). Therefore, the canonical quotient $\bbP/R^{\bbP}$ is homeomorphic to the pseudoarc and $\bullet^P =\bullet^\bbP/R^\bbP$ is a point in the pseudoarc $P$. Furthermore, the argument from \cite{IS} showing that ${\rm Aut}(\bbP)$ is dense in ${\rm Homeo}(P)$, gives that 
	${\rm Aut}(\bbP_\bullet)$ is dense in the set $\{ h\in {\rm Homeo}(P)\mid h(\bullet_P)=\bullet_P\}$. 	
	\begin{theorem}\label{S4mainT}  {}\
	\begin{enumerate}
	\item[(i)]	 For every $n \in \mathbb{N}$, the set of all $\bar \gamma \in {\rm Aut}(\bbP_\bullet)^n$ such that the orbit of $\bar \gamma$ under the diagonal conjugacy action by ${\rm Aut}(\bbP_\bullet)$ is dense in ${\rm Aut}(P_\bullet)^n$ is comeager in ${\rm Aut}(\bbP_\bullet)^n$. 
		 
	\item[(ii)]	 The set of all $\bar \gamma \in {\rm Aut}(\bbP_\bullet)^\bbN$ such that the orbit of $\bar \gamma$ under the diagonal conjugacy action by ${\rm Aut}(\bbP_\bullet)$ is dense in ${\rm Aut}(P_\bullet)^\bbN$ is comeager in ${\rm Aut}(\bbP_\bullet)^\bbN$. 
	\end{enumerate}
	\end{theorem}
	
	 The theorem will follow from the lemma below. 
	 
	\begin{lemma} \label{JPPaut}
		Suppose that  $A,B,A',B' \in \cP_\bullet$, 
		there exist $A^+,B^+\in \cP_\bullet$ and $\alpha\in {\rm Epi}(A^+,A)$, $\alpha'\in {\rm Epi}(A^+, A')$, $\beta\in {\rm Epi}(B^+,B)$, and $\beta'\in {\rm Epi}(B^+,B')$ with the following property:
		
		\noindent for all $f\in {\rm Epi}(A,B)$, $f'\in {\rm Epi}(A',B')$, there exists $f^+\in {\rm Epi}(A^+, B^+)$ such that   
		 \begin{equation}\notag
		 	f\circ \alpha = \beta \circ f^+  \ \hbox{ and }\ f'\circ \alpha'= \beta' \circ f^+.
		 \end{equation}
		\end{lemma}

	\begin{PROOF}{Lemma \ref{JPPaut}}
	Let 
	\[
	A= \{ 0, \dots, m\}, \; A' = \{ 0, \dots, m'\},\; B= \{ 0, \dots, n\}, \; B'= \{ 0, \dots, n'\}, 
	\]
	with the constant in all these structures interpreted as $0$. 
	
	We define $A^+$ with $R^{A^+}$. 
	Let 
	\[
	\begin{split} 
	A'_0 &= \{ (0,-1,j)\mid j\in A'\}\\
	A'_i &= \{ (i-1,1, j)\mid j\in A'\} \cup \{ (i,-1, j)\mid j\in A'\},\; \hbox{ for }0<i\in A.
	\end{split} 
	\]
	and we define $R^{A^+}$ on these sets to be the smallest reflexive and symmetric relation with 
	\[
	(i,\epsilon,j) \,R^{A^+} (i,\epsilon,j') \Leftrightarrow |j-j'|\leq 1, \;\hbox{ for }i\in A\hbox{ and }\epsilon\in \{ -1,1\}, 
	\]
	and 
	\[
	(i-1, 1, 0)\, R^{A^+} (i, -1, 0),\;\hbox{ for }0<i\in A.
	\]
	Note that with these definitions $A_0'$ is a reflexive linear graph with $m'+1$ many elements and $A_i'$, $0<i\in A$, is a reflexive linear graph with $2(m'+1)$ many elements. 
	
	Let now $A^+$ be the union 
	\[
	\bigcup_{i\in A} A'_i
	\]
	with the point $(i-1, -1, m')\in A'_{i-1}$ identified with the point $(i-1,1,m')\in A'_{i}$ for each $0<i\in A$. 
	Let $R^{A^+}$ be the relation on $A^+$ that is induced from $R^{A^+}$ defined on $A'$ and $A'_i$ for $i\in A$, $i<m$. 
	We interpret the constant as $(0,-1,0)\in A'_0$.

	The definitions of $B^+$ and $R^{B^+}$ are similar. For each $k\in B$, let 
	\[
	\begin{split} 
	B'_0 &= \{ (0,-1,l)\mid l\in A'\}\\
	B'_k &= \{ (k-1,1, l)\mid l\in B'\} \cup \{ (k,-1, l)\mid l\in B'\},\; \hbox{ for }0<k\in B, 
	\end{split}
	\]
	and let $R^{B^+}$ be the smallest reflexive and symmetric relation with 
	\[
	\begin{split}
	(k,\epsilon ,l) \,&R^{B^+} (k,\epsilon ,l') \Leftrightarrow |l-l'|\leq 1,\;\hbox{ for }k\in B\hbox{ and }\epsilon\in \{ -1,1\}, \\
	(k-1, 1, 0) \, &R^{B^+} (k, -1, 0),\;\hbox{ for }0<k\in B.
	\end{split}
	\]
	
	Let $B^+$ be the union 
	\[
	\bigcup_{k\in B} B'_k
	\]
	with the points $(k-1, -1, n')\in B'_{k-1}$ and $(k-1,1,n')\in B'_k$ identified for each $0<k\in B$. Let $R^{B^+}$ be the relation on $B^+$ that is induced from $R^{B^+}$ defined on $B'_k$ for $k\in B$.  We interpret the constant as $(0,-1,0)$. 
	
	It is easy to check that the structures $A^+$ and $B^+$ defined above are in $\cP_\bullet$.

	Let $\alpha\colon A^+\to A$ and $\alpha'\colon A^+\to A'$ be defined by  
	\[
	\alpha(i, \epsilon, j)=  i\;\hbox{ and }\;\alpha'(i,\epsilon,j)=j.
	\]
	Similarly, let $\beta\colon B^+\to B$ and $\beta'\colon B^+\to B'$ be defined by 
	\[
	\beta(k, \epsilon, l)=  k\;\hbox{ and }\;\beta'(k,\epsilon,l)=l
	\]
	It is easily checked that $\alpha, \alpha', \beta, \beta'$ are epimorphisms in $\cP_\bullet$.

	Assume now that $f\colon A\to B$ and $f'\colon A'\to B'$ are epimorphisms in $\cP_\bullet$. We need to find an epimorphism $f^+\colon A^+\to B^+$ in $\cP_\bullet$ 
	such that 
	 \begin{equation}\label{E:comfp} 
		 	f\circ \alpha = \beta \circ f^+  \ \hbox{ and }\ f'\circ \alpha'= \beta' \circ f^+.
		 \end{equation}
		 
	By induction on $i\in A$, we define $f^+$ on $A_0'\cup \cdots \cup A'_{i}$. This will be done so that 
	the value of $f^+$ at the last vertex of $A_0'\cup \cdots \cup A'_{i}$, which is $(i, -1, m')$ is given by 
	\begin{equation}\label{E:finv}
	f^+(i,-1,m') = (f(i), \epsilon_i, f'(m')), 
	\end{equation} 
	for some $\epsilon_i\in \{ -1,1\}$.

	We start with defining $f^+$ on $A'_0$ by letting 
	\[
	f^+(0,-1,j) = (0,-1,f'(j)) \;\hbox{ for }j\in A'. 
	\] 
	We note that \eqref{E:finv} holds with $i=0$. 
	
	Assume that $f^+$ is defined on $A_0'\cup \cdots \cup A'_{i-1}$ for some $0<i \in A$ and $\epsilon_{i-1}$ is determined by \eqref{E:finv}. We need to define $f^+$ on $A'_i$. Let $\delta_i\in \{ -1,0,1\}$ be such that 
	\[
	f(i) = f(i-1)+\delta_i. 
	\]
	We divide the definition of $f^+$ on $A'_i$ into two cases. Keep in mind that we need to give the values of $f^+$ on triples of the form $(i-1, 1, j)$ and $(i, -1, j)$ with $j\in A'$. 	
	
	First assume $\epsilon_{i-1}\cdot \delta_i\geq 0$, that is, either $\delta_i=0$ or $\delta_i=\epsilon_{i-1}$ . In this case, let 
	\[
	f^+(i-1, 1, j)  = (f(i-1), \epsilon_{i-1}, f'(j))
	\]
	and 
	\[
	f^+(i, -1, j) =
	\begin{cases}
	 (f(i), \epsilon_{i-1}, f'(j)), &\hbox{ if }\delta_i=0;\\
	 (f(i), -\epsilon_{i-1}, f'(j)), &\hbox{ if }\delta_i=\epsilon_{i-1}.
	\end{cases} 
	\]
	Note that with these definitions \eqref{E:finv} holds for $i$. 
	
	The second case is $\epsilon_{i-1}\cdot \delta_i< 0$, that is, $\delta_i=-\epsilon_{i-1}$. Let 
	\[
	j_{f'}= \hbox{ the smallest }j\in A'\hbox{ with } f'(j_{f'})= n'. 
	\]
	Define 
	\[
	f^+(i-1, 1, j) = 
	\begin{cases} 
	(f(i-1), \epsilon_{i-1}, f'(j)), &\hbox{ for }\; j_{f'}\leq j;\\
	(f(i-1), -\epsilon_{i-1}, f'(j)), &\hbox{ for }\; j\leq j_{f'}.
	\end{cases}
	\]
	and 
	\[
	f^+(i,-1,j) = (f(i), \epsilon_{i-1}, f'(j)). 
	\]
	Again \eqref{E:finv} holds for $i$. 
	
	Since $f^+$ is given by formulas, it is easy (if somewhat tedious) to check, using condition \eqref{E:finv}, that $f^+$ is an epimorphism in $\cP_\bullet$. We leave this check to the reader. It is immediate that \eqref{E:comfp} holds. 
		\end{PROOF} 
	
	For the corollary below, we recall the result of Hamilton \cite{HA} that each homeomorphism of the pseudoarc has a fixed point. (In fact, each continuous map from the pseudoarc to itself has a fixed point.)  The following theorem gives another strengthening of the result in \cite{B-M}. It gives a homeomorphism of the pseudoarc with a dense conjugacy class with additional control on fixed points.

	\begin{corollary}\label{C:fixed} There exists $f\in {\rm Homeo}(P)$ with the following property. For each $h\in {\rm Homeo}(P)$ and $x\in P$ with $x$ being a fixed point of $h$, there exists a sequence $(g_n)$ in ${\rm Homeo}(P)$ such that 
	\begin{enumerate}
	\item[---] the sequence $(g_n f g_n^{-1})$ converges to $h$ in ${\rm Homeo}(P)$; 
	
	\item[---] $x$ is a fixed point of $g_nfg_n^{-1}$ for each $n$. 
	\end{enumerate} 
	\end{corollary}
	
	\begin{PROOF}{Corollary \ref{C:fixed}} Let $\tilde{f}\in {\rm Aut}(\bbP_\bullet)$ have a dense conjugacy class in ${\rm Aut}(\bbP_\bullet)$ as guaranteed by Theorem~\ref{S4mainT}. Let $f={\rm pr}(\tilde{f})$. Fix now $h\in {\rm Homeo}(P)$ and $x\in P$ with $h(x)=x$. Note that $\bullet_P$ is a fixed point of $f$. By homogeneity of $P$, see \cite{Le}, there exists $g\in {\rm Homeo}(P)$ such that 
	$g(\bullet_P)=x$. It follows that $g^{-1}hg$ is in ${\rm Homeo}(P)$ and has $\bullet_P$ as a fixed point. By the choice of $f$, there exists a sequence $(g_n')$ in ${\rm Homeo}(P)$ such that $g_n'(\bullet_P)=\bullet_P$ and the sequence $\big(g_n'f g_n'^{-1}\big)$ converges to $g^{-1}hg$. It follows that the sequence $g_n = gg_n'$ is as required. 
	\end{PROOF}

	\section{A homeomorphism of the pseudoarc that is not conjugate to an automorphism of the pre-pseudoarc} \label{s4}
	
	This section concerns $\cP$---the class of finite connected linear reflexive graphs with all edge-preserving surjective homomorphisms. 
	\begin{theorem}\label{T:homeout}
		Let $P$ be the pseudoarc, which we identify with the natural quotient of the projective Fra\"issé limit $\bbP$ of the class of finite linear graphs. There exists a homeomorphism in ${\rm Homeo}(P)$ that is not conjugate to an element of ${\rm Aut}(\bbP)$.
	\end{theorem}
	
	We note that $\cP$ obviously fulfills condition (iv) of Proposition~\ref{P:small}; so it fulfills conditions (i)--(iii) in Proposition~\ref{P:small} and the conclusions of Propositions~\ref{P:cloint} and \ref{P:inj}. It was proved in \cite{IS} that each equivalence class of the compact equivalence relation $R^\bbP$ on $\bbP$ has at most two elements. This fact will be used in the proof below.

	The following is a key lemma; it provides a sufficient condition for a homeomorphism not to be conjugate to any automorphism of $\bbP$. Our argument proving the lemma hinges on the rigidity of automorphisms---they preserve the algebra of regular open  sets associated with the projections from $\bbP$ onto finite linear graphs.
	Our example, constructed in Proposition \ref{mainpro}, will satisfy the premises of the lemma.
	
	\begin{lemma} \label{keyle}
		Suppose that $\sigma: P \to P$ is a homeomorphism, for which there exist a subcontinuum $C \subseteq P$ and $x \neq y \in C$ with the following properties: 
		\begin{enumerate}
			\item[---] $\sigma \rest C = {\rm id}_C$;
			\item[---] if $U_1,U_2,U_3$ are regular open subsets of $P$ such that 
			\[
			\begin{split} 
			&U_1\cap U_2= U_2\cap U_3=U_1\cap U_3=\emptyset,\\
				&{\rm cl}(U_1) \cup {\rm cl}(U_2)\cup {\rm cl}(U_3)= P,\;\; {\rm cl}(U_1) \cap {\rm cl}(U_2)\cap {\rm cl}(U_3) = \emptyset,\; \hbox{ and }\\
				&{\rm cl}(U_1) \cap C = \emptyset,\;\; x \in U_2,\;\;  y \in U_3,
			\end{split}
			\]
			then 
			\begin{equation}\notag 
			\sigma(U_2) \cap U_3 \neq \emptyset\; \textrm{ or } \; \sigma(U_3) \cap U_2 \neq \emptyset. \end{equation}			
		\end{enumerate}
	
	Then $\sigma$ is not conjugate in ${\rm Homeo}(P)$ to any member of ${\rm Aut}(\bbP)$.	
	\end{lemma}
	
	\begin{PROOF}{Lemma \ref{keyle}}
		Note that the above property of $\sigma$ is invariant under conjugation by a homeomorphism, therefore it suffices to show that representing $P$ as a quotient of the prespace $\bbP$, we necessarily have $\sigma \notin {\rm Aut}(\bbP$). Let $\pr: \bbP \to P$ be the canonical continuous surjection. Suppose that $\sigma$ is an automorphism, that is, there exists $\tilde{\sigma}\in {\rm Aut}( \bbP)$ with $\pr(\tilde{\sigma}) = \sigma$, where $\pr: {\rm Aut}(\bbP) \to {\rm Homeo}(P)$ is the canonical embedding; in other words, we have 
		\[
		\sigma \circ \pr = \pr \circ \tilde{\sigma}. 
		\]
				
		We set 
		\[
		\bbC =  \pr^{-1}(C), 
		\]
		Note that $\bbC$ is a compact zero-dimensional space.
		We claim that 
		\begin{equation}\label{E:restid} 
		\tilde{\sigma}\rest \bbC = {\rm id}_{\bbC}.  
		\end{equation} 
		Indeed, since $\sigma\rest C= {\rm id}_C$, we see that $\bbC$ is invariant under $\tilde {\sigma}$ and 
		\begin{equation}\label{E:sigov}
		\tilde{\sigma}(u)R^{\bbP} u,\;\hbox{ for all } u\in \bbC.
		\end{equation}  
		Since $\tilde{\sigma}$ is injective and each $R^{\bbP}$-equivalence class has at most two elements, we see that $\tilde{\sigma}\circ\tilde{\sigma}(u)= u$, 
		for all $u\in \bbC$. So, if there is $u \in  \bbC$ with $\tilde{\sigma}(u) \not=u$ for some $u \in \bbC$, then by taking $U$ to be a small enough clopen neighborhood of $u$ (in $\bbC$) and $V=\tilde{\sigma}(U)$, we see that there are relatively clopen subsets $U, V$ of $\bbC$ with 
		\begin{equation}\label{E:unna}
		U\cap V=\emptyset,\; \; \tilde{\sigma}(U)=V.
		\end{equation} 
		and 
		\begin{equation}\label{E:unna2}
		U\cup V\not= \bbC,
		\end{equation} 
		so the complementer clopen set
		$$Z = \bbC \setminus (U \cup V) \neq \emptyset.$$
		Since each $R^{\bbP}$-equivalence class has at most two elements, it follows from \eqref{E:sigov} and \eqref{E:unna} that $R^{\bbP}(U\cup V)= U\cup V$. Since $\bbC$ is $\overline{\sigma}$-invariant, necessarily $R^{\bbP}(Z) = Z$. 
		Then $\pr(U \cup V) \cap \pr(Z) = \emptyset$ contradicting that $C = \pr(\bbC)$ is connected.

			Fix a finite linear graph $I_0$ and $\varphi \in {\rm Epi}(\bbP,I_0)$ such that $\varphi \circ \pr^{-1}(x) \cap \varphi \circ \pr^{-1}(y) = \emptyset$.			
			Using Claim \ref{pdensity} we can find a finite linear graph $I_1$ and epimorphism $\psi \in {\rm Epi}(\bbP,I_1)$ such that  for some $\alpha, \alpha' \in {\rm Epi}(I_1,I_0)$ we have
			$$
				\begin{array}{rcl}
				 \varphi & = & \alpha \circ \psi, \\
				 \varphi \circ \overline{\sigma} & = & \alpha' \circ \psi,
				\end{array}
			$$ 
			in other words,
			\begin{equation} \label{ids} 
				\alpha \circ \psi \circ \overline{\sigma} = \alpha' \circ \psi.
			\end{equation}
			Let $K_1 = \psi[\bbC]$ (which will, roughly speaking represent $\bbC$ in $I_1$) we note that condition \eqref{E:restid} (together with \eqref{ids})  implies
			\begin{equation}\label{E:aleq} 
					\alpha \rest K_1 = \alpha' \rest K_1.
			\end{equation} 
			Similarly, we let $K_0 = \alpha \circ \psi[\bbC]$.
			We also let	 		 
			\[
			X_0 = (\alpha \circ \psi) \circ \pr^{-1}(x)\;\hbox{ and }\; Y_0 = (\alpha \circ \psi) \circ \pr^{-1}(y), 
			\]
		 we have $X_0,Y_0 \subseteq K_0 \subseteq  I_0$, and
			\begin{equation} \label{X_I} X_0 \cap Y_0 = \emptyset \end{equation}
			by the choice of $\varphi = \alpha \circ \psi$ above.
			Similarly, let		 		 
			\[
			X_1 = \psi \circ \pr^{-1}(x)\;\hbox{ and }\; Y_1 = \psi \circ \pr^{-1}(y).
			\]

		
		Pick disjoint subsets $X',Y' \subseteq I_0$ with  $X' \cup Y' = K_0$, $X_0 \subseteq X'$ and $Y_0 \subseteq Y'$.  
		Set 
		\[
		X^+ = \alpha^{-1}(X') \cap K_1\;\hbox{ and }\; Y^+ = \alpha^{-1}(Y') \cap K_1.
		\] 
		Then we have
		\begin{align}
		&X^+ \cup Y^+ = K_1, \label{Xu}\\
		& X^+ \cap Y^+ = \emptyset, \\
		&X_{1} \subseteq X^+, \ Y_{1} \subseteq Y^+.\label{Xk}
		\end{align}
		Set 
		\[
		V_1= I_{1} \setminus K_1. 
		\]
		
		Let 
		 \[
		 W_1 = \pr(\psi^{-1}(V_1)),\;\;  W_2 = \pr(\psi^{-1}(X^+)),\;\; W_3 = \pr(\psi^{-1}(Y^+)).
		 \]
		 Then, by \eqref{Xu} and the definition of $V_1$, we have 
		 \begin{equation}\label{E:cov3}
		 W_1\cup W_2\cup W_3= P.
		 \end{equation} 
		 Note furthermore that 
		 \begin{equation}\label{E:int3} 
		  W_1\cap W_2\cap W_3=\emptyset.
		 \end{equation}
		 Otherwise, there exist $w_1\in \psi^{-1}(V_1)$, $w_2\in \psi^{-1}(X^+)$, and $w_3\in \psi^{-1}(Y^+)$ such that 
		 ${\rm pr}(w_1)= {\rm pr}(w_2)= {\rm pr}(w_3)$, that is, 
		 \[
		 w_1R^\bbP w_2R^\bbP w_3 R^\bbP w_1.
		 \]
		 Since each $R^\bbP$-equivalence class has at most two elements, we see that $w_i=w_j$ for some $1\leq i<j\leq 3$, which entails that the sets 
		 $\psi^{-1}(V_1),\, \psi^{-1}(X^+),\, \psi^{-1}(Y^+)$ are not pairwise disjoint, which contradicts pairwise disjointness of $V_1, X^+, Y^+$.
		 The equality $K_1 = \psi[\pr^{-1}(C)]$ together with $V_1  \cap K_1 = \emptyset$ imply 
		 \begin{equation}\label{E:rag}
		 C \cap W_1  = \emptyset.
		 \end{equation} 
		 Also, it follows from \eqref{Xk} and the definition of $X_1, Y_1$ that
		 \begin{equation}\label{E:rag2}
		 x \notin W_1 \cup W_3\;\hbox{ and }\; y \notin W_1 \cup W_2.
		 \end{equation} 
		By Proposition~\ref{P:cloint}, we have 
		\begin{equation}\label{E:ggr} 
		{\rm int}(W_i) = P \setminus \bigcup_{j\not=i} W_j\;\hbox{ and }\;  {\rm cl}\big({\rm int}(W_i)\big) = W_i,\;\hbox{ for }i=1, 2, 3. 
		\end{equation} 		
		It follows from \eqref{E:cov3}--\eqref{E:ggr} 
		 that 
		$U_i = {\rm int}(W_i)$, $i=1,2,3$, satisfy the requirements of the condition in the statement of the lemma. 
		
		Now, assume that $\sigma(U_2) \cap U_3 \neq \emptyset$ as the other case can be dealt with the same way.
		Pick $z \in \bbP$ such that $\pr(z) \in U_2$ and $\sigma(\pr(z)) = \pr(\tilde{\sigma}(z)) \in U_3$.
		By the definition of $U_2$, we have $\pr(z) \in W_2 \setminus( W_1 \cup W_3)$, so 
		\begin{equation}\label{E:inxpl2}
		\psi(z) \in X^+,
		\end{equation}
		and, therefore, 
		\begin{equation}\label{E:inxpl}
		\alpha(\psi(z)) \in X'.
		\end{equation}
		By a similar argument, we have 
		$\psi(\tilde{\sigma}(z))\in Y^+$, 
		and so 
		\begin{equation}\label{E:inypl}
		\alpha (\psi(\tilde{\sigma}(z)))\in Y'. 
		\end{equation} 
		Now \eqref{ids} and \eqref{E:aleq} (and \eqref{Xk}) together with \eqref{E:inxpl2} and \eqref{E:inxpl} give 
		\begin{equation} \label{e2} 
		\alpha (\psi(\tilde{\sigma}(z))) = \alpha'(\psi(z)) = \alpha(\psi(z)) \in  X'. 
		\end{equation}
		But \eqref{E:inypl} and \eqref{e2} contradict the disjointness of $X'$ and $Y'$. 
	\end{PROOF}
	
		To construct a homeomorphism that is not conjugate to any element of Aut$(\bbP$) it remains to find one that satisfies the premises of Lemma \ref{keyle}.
	
	\begin{proposition} \label{mainpro}
		There exists $\sigma \in {\rm Homeo}(P)$ as in Lemma \ref{keyle}.
	\end{proposition} 
	\begin{PROOF}{Proposition \ref{mainpro}}
		We are going to 
		\begin{itemize}
			\item define a generic sequence $J_0,J_1, \ldots, J_n, \ldots$ of finite linear graphs by hand (together with the bonding maps $\pi_{i,j}$, $i \leq j < \omega$),
			\item identify $\bbP$ with $\projlim_{i} J_i$, and let $\pi_j: \projlim_{i} J_i \to J_j$ be the natural projection,
			\item and  construct an increasing sequence $(n_i)_i$ and epimorphisms $h_i: J_{n_{i}+1} \to J_{n_i}$ in such a way that for any choice of $\tilde{\sigma}_i \in $ Aut($\bbP$) with $\pi_{n_i} \circ \tilde{\sigma}_i = h_i \circ \pi_{n_i+1}$
			we have that $\tilde{\sigma}_i$ converges uniformly (in $\cC(P,P)$), moreover the limit is an element of ${\rm Homeo}(P)$,
		\end{itemize}
		and we will let $\sigma$ be the obtained function $\lim_{i} \tilde{\sigma}_i$.
		
	The following definitions are particular cases of notions from an unpublished work of Solecki and Tsankov.	
\begin{definition}[Solecki--Tsankov]\label{tpdf}
	Let $L$ be a finite linear graph. A family $t$ of sets is an {\em $L$-type} if  $t$ is a maximal linearly ordered by $\subseteq$ family of connected subsets of $L$.
\end{definition}
		\begin{definition}[Solecki--Tsankov]
			Suppose that $f: J \to L$ is an epimorphism between the finite linear graphs $J$, $L$, and let $a \in J$ be a node, $M \subseteq J$ be a subinterval with $a$ being an endpoint of $M$.
			Then we define the type of the pair $a, M$ with respect to $f$, in symbols, ${\rm tp}^{a,M}(f)$ to be the $f[M]$-type
			$$ {\rm tp}^{a,M}(f) := \{ f[M']: \ M' \subseteq M \text{ is an interval, } a \in M' \}.$$
			(So if $M = \{a,a', a'', \ldots,a^{(|M|-1)}\}$ is an enumeration of $M$ that is a walk, then ${\rm}{tp}^{a,M}(f)$ codes the order in which  the walk $f(a), f(a'), \ldots f(a^{(|M|-1)})$ visits the nodes of $J$.)
		\end{definition}

		By induction on $i$ we define
		\begin{itemize}
			
			\item $J_i$, $K_i$,  $\pi_{i,i+1}$ for $i=-1,0,1, \ldots$,
			\item $h_{-1}$,
			\item $h^{\bullet}_{k}$, $k \in \omega$,
			\item $L_{4k+2}, L_{4k+3}$ ($k \in \omega$),
			\item $g_{k}$, $f_{k}$, $f'_k$, $h_{k}$  $(k \in \omega)$,
			
		\end{itemize}   
		keeping the following outline in mind. The idea is that $\projlim_i K_i \subseteq \projlim_i J_i$ will represent the subcontinuum $C$ on which the prospective homeomorphism $\sigma$ is the identity.
		$\sigma$ will be approximated with automorphisms represented by the $h_i'$s, or $f_i$'s, $g_i$'s. Each $J_i$ corresponds to a chain-like decomposition of the space, the pieces of which will not be invariant under  $\sigma$, even though the condition $\sigma \rest C = {\rm id}_C$ would dictate that pieces that intersect $C$ are fixed.
		
		 Basically the functions $h_i'$, $f_i$, $g_i$ are liftings of each other, the only purpose to give them different names is an attempt to ease the notational awkwardness later, e.g.\ in the proof of Lemma \ref{justifl}. More concretely,
	 $h_i: J_{4i} \to J_{4i-1}$ is an epimorphism and $h^\bullet_i: J_{4i+1} \to J_{4i}$ is an epimorphism with $h_i \circ h^\bullet_i = \pi_{4i-1,4i+1}$, where the existence of $h^\bullet_i$'s will ensure that the limit is (left-)invertible, so injective. Then $g_i: J_{4k+2, 4k+1}$ will be a lifting of $h_i$, which will be ensured by the condition $h^\bullet_i \circ g_i = 	\pi_{4i,4i+2}$. In the recursive construction $K_{4i+1}$ will not only have $K_{4i+2}$ as a $g_i$- and $\pi_{4i+1,4i+2}$-preimage (on which  $g_i$ and $\pi_{4i+1,4i+2}$ coincide), but also a copy of it $L_{4i+2}$. In the next step of the recursion this is  followed by the construction of $J_{4i+3}$, and $f_i, \pi_{4i+2, 4i+3} : J_{4i+3} \to J_{4i+2}$, where we can also guarantee that $f_i, \pi_{4i+2, 4i+3}$ map $L_{4i+3}$ to $L_{4i+2}$, as well as $K_{4i+2}$ to $K_{4i+2}$. But instead of $f_i$ being a lifting of $g_i$, it will be only an almost lifting, because we will arrange so that $f_i \rest L_{4i+3}$ is roughly speaking a shift of $\pi_{4i+3,4i+2} \rest L_{4i+3}$ by one. (This is the crucial step to guarantee that regular open partitions are never $\sigma$-invariant.) Finally, $h_{i+1}: J_{4i} \to J_{4i-1}$ will be a lifting of $f_i$ guaranteed by the condition $f_i \circ \pi_{4i+3,4i+4} = \pi_{4i+2,4i+3} \circ h_{i+1}$. Each $h_i$ will be responsible for an amalgamation task to ensure genericity of the sequence. We will have further technical conditions, e.g.\ the conditions on types, which will be necessary to carry out our main tasks.
	 
	Formally we require that the $J_i$, $K_i$, $f_i$, $f_i'$, $g_i$, $h_i$, $h^\bullet_i$ satisfy the following.
		\begin{enumerate}[label = $\blacksquare_{\arabic*}$, ref = $\blacksquare_{\arabic*}$]
			\item \label{J1} $J_i$ is a finite linear graph,
			\item $K_i\subseteq J_i$ is a subinterval,
			\item $K_{-1} = L_{-1} = J_{-1}$ are the one element linear graph,
			\item $h_i$, $h_i^\bullet$, $g_i$, $f_i$, $f_i'$ are epimorphisms
			\begin{itemize}
				\item $h_i \in {\rm Epi}(J_{4i}, J_{4i-1})$, 
				\item $h^\bullet_i \in {\rm Epi}(J_{4i+1}, J_{4i})$,
				\item $g_i \in {\rm Epi}(J_{4i+2},J_{4i+1})$,
				\item $f_i,f_i' \in {\rm Epi}(J_{4i+3}, J_{4i+2})$,
			\end{itemize}
			\item \label{Kendp} $\pi_{i,i+1}[K_{i+1}] = K_i$, and the $\pi$'s agree with the $h_i$/$h_i^\bullet$/$g_i$/$f_i$ on $K_k$ whenever it is appropriate (and defined), that is, 
			\begin{itemize}
				\item  $\pi_{4k-1,4k} \rest K_{4k} = h_k \rest K_{4k}$,
				\item $\pi_{4k,4k+1} \rest K_{4k+1} = h^\bullet_k \rest K_{4k+1}$,
				\item $\pi_{4k+1,4k+2} \rest K_{4k+2} = g_k \rest K_{4k+2}$,
				\item $\pi_{4k+2,4k+3} \rest K_{4k+3} = f_k \rest K_{4k+3}$,
			\end{itemize} moreover they map endpoints to endpoints,

			\item \label{cont} for each $a, b \in J_{i+1}$, if $a R^3 b$, then $\pi_{i,i+1}(a) R \pi_{i,i+1}(b)$ (e.g.\ this is ensured if for $c \in J_i$, each connected component in $\pi_{i,i+1}^{-1}(c)$ is an interval of length at least $3$),
		
			\item for $i = 4k-1$ (including $i=-1$):
			\begin{enumerate}[label =  $\blacksquare_7(\alph*)$, ref =  $\blacksquare_7(\alph*)$]
				
				\item   for each maximal connected component (subinterval)  $C$ of $J_{4k} \setminus K_{4k}$ we have $\pi_{4k-1,4k}[C] = J_{4k-1} = h_{k}[C]$, 
				and 
				$$ {\rm{tp}}^{c^*,C}(h_{k}) = {\rm tp}^{c^*,C}(\pi_{4k-1,4k}) \ \textrm{ for } c^* \in C \textrm{ with } c^* R K_{4k},$$
				\item \label{4k+3comm} (and if $k \geq 0$:) $h_{4k-2} \circ \pi_{4k-1,4k} = \pi_{4k-2,4k-1} \circ h_{4k-1}$,
			\end{enumerate}  
			
			\item for $i = 4k$:
			\begin{enumerate}[label = $\blacksquare_8(\alph*)$, ref =  $\blacksquare_8(\alph*)$]
				\item \label{invh} $h_{k} \circ h^\bullet_{k} = \pi_{4k-1,4k} \circ \pi_{4k,4k+1}$,
				\item   for each maximal connected component (subinterval)  $C$ of $J_{4k+1} \setminus K_{4k+1}$ we have $\pi_{4k,4k+1}[C] = J_{4k} = h_{k}[C]$, 
				and 
				$$ {\rm tp}^{x,C}(h_{k}) = {\rm tp}^{x,C}(\pi_{4k,4k+1}) \ \textrm{ for }x \in C \textrm{ with } x R^{J_{4k+1}} K_{4k+1},$$
					\item \label{contbullet} for each $a, b \in J_{4k+1}$, if $a R^3 b$, then $h^\bullet_k(a) R h_k^\bullet(b)$,
				
			\end{enumerate}
			\item for $i = 4k+1$:
			\begin{enumerate}[label =  $\blacksquare_9(\alph*)$, ref = $\blacksquare_9(\alph*)$]
				\item \label{Lendp} 	 $\pi_{4k+1,4k+2}[L_{4k+2}] = K_{4k+1}$, and $\pi_{4k+1,4k+2} \rest L_{4k+2}$ maps endpoints to endpoints,
				\item $\pi_{4k+1,4k+2} \rest L_{4k+2} = g_{k} \rest L_{4k+2}$,
				\item $\neg (L_{4k+2} R K_{4k+2})$,
				\item   for each  of the three maximal connected component (interval)  $C$ in $J_{4k+2} \setminus (K_{4k+2} \cup L_{4k+2})$ we have 
				$$\pi_{4k+1,4k+2}[C] = J_{4k+1} = g_{k}[C],$$ 
				and 
				$$ \begin{array}{l}
					\textrm{for }x \in C \textrm{ with } x R^{J_{4k+2}} (K_{4k+2}\cup L_{4k+2}) \\
					\ \ {\rm tp}^{x,C}(g_{k}) = {\rm tp}^{x,C}(\pi_{4k+1,4k+2}),
					 \end{array}$$
				\item \label{commn} $h^\bullet_{k} \circ g_{k} = \pi_{4k,4k+1} \circ  \pi_{4k+1,4k+2}$,
			\end{enumerate}
			
			\item for $i = 4k+2$:
			\begin{enumerate}[label =  $\blacksquare_{10}(\alph*)$, ref =  $\blacksquare_{10}(\alph*)$]
				\item $\pi_{4k+2,4k+3}[L_{4k+3}] = L_{4k+2}$, $\pi_{4k+2,4k+3} \rest L_{4k+3}$ maps endpoints to endpoints,
					\item $\pi_{4k+2,4k+3} \rest L_{4k+3} = f_{k}' \rest L_{4k+3}$, 
				\item letting $C_1,C_2,C_3$ denote the connected components in $J_{4k+2} \setminus (K_{4k+2} \cup L_{4k+2})$ if we let
				$$C_i' := \pi^{-1}_{4k+2,4k+3}(C_i), \ i =1,2,3,$$
				then we have that
				$C_1',C_2',C_3'$ are all connected (intervals), and
				$$C_j' = (f_{k}')^{-1}(C_j) \ \textrm{ for }j=1,2,3.$$
				\item  $\pi_{4k+1,4k+2} \circ f_{k}' = g_{k} \circ \pi_{4k+2,4k+3}$,

				\item for $a \in J_{4k+3} \setminus L_{4k+3}$ we have $f_{k}(a) = f'_{k}(a)$,
				\item \label{slide} for each $a \in L_{4k+3}$ we have ${\rm dist}(f_{k}(a),f'_{k}(a)) \leq 2$,
				and if $a \in L_{4k+3}$ is such that $\pi_{4k+2,4k+3}(a) \in L_{4k+2}$ is neither an endpoint nor is related to an endpoint of $L_{4k+2}$, then 
				$${\rm dist}(f_{k}(a),f'_{k}(a)) = 2 \ (= {\rm dist}(\pi_{4k+2,4k+3}(a),f'_{k}(a))$$
					(note that the above imply
					\begin{equation} \label{4k+2acomm} (\pi_{4k+1,4k+2} \circ f_{k}) R (g_{k} \circ \pi_{4k+2,4k+3}) \end{equation}
					(here ${\rm dist}(a,b) \leq n$ means $a R^n b$),
			\end{enumerate}

			\item \label{generic} whenever $M$ is a linear graph and $\phi: M \to J_i$ for some $i$, then there exists $j \geq i$, $\phi^+: J_j \to M$, such that $\pi_{i,j} = \phi \circ \phi^+$ (where $\pi_{i,j}$ denotes the composition $\pi_{i,i+1} \circ \ldots \circ \pi_{j-1,j}$),
		\end{enumerate}
		Before the construction we argue that this will result the required homeomorphism.
		\begin{lemma}\label{justifl}
			Assume that the induction maintaining \ref{J1}- \ref{generic} can be carried out. Then 
			\begin{enumerate}[label = $(\arabic*)$, ref = $(\arabic*)$]
	\item \label{bbP}$\projlim_{i} J_i$ (with the bonding maps $\pi_{i,j}$, $i \leq j$ defined as in \ref{generic}) is isomorphic to $\bbP$,
	\item letting $K = \projlim_{i} K_i (\subseteq \projlim_{i} J_i)$, its projection $C:=\pr[K]$ is a subcontinuum,
	\item \label{xendp} there exists $(x_i)_i$, $(y_i)_i \in K$ such that $\pi_j((x_i)_i) = x_j$, $\pi_j((y_i)_i) = y_j$ are the two endpoints of $K_j$,
	\item \label{Hk} if for each $k$ the map $\sigma_k \in {\rm Aut}(\bbP)$ is such that $\pi_{4k-1} \circ \sigma_k = h_{k} \circ \pi_{4k}$, then $\sigma_k$ is convergent in ${\rm Homeo}(P)$, and their limit $\sigma$ is as in Lemma \ref{keyle} witnessed by the subcontinuum $C$ and $x = \pr((x_i)_i)$, $y = \pr((y_i)_i)$.
\end{enumerate} 
		\end{lemma}
		\begin{PROOF}{Lemma \ref{justifl}}(Lemma \ref{justifl})
			It is standard that \ref{generic} suffices for \ref{bbP}. 
			
			Similarly, by a routine argument it follows from the fact that the $K_i$'s are interval that $C$ is a continuum, but for the sake of completeness we elaborate.
			Since $K_i$ is an interval, any clopen decomposition $C_0 \cup C_1$ of $K$ is of the form $C_0 = \pi_{i}^{-1}(K^0_i)$, $C_1 = \pi_{i}^{-1}(K^1_i)$, for a partition $K^0_i \cup K^1_i$ of $K_i$. Now there must exist $(z^0_i)_i, (z^1_i)_i \in K$ with $z^0_{i} \in  K^0_i$, $z^1_{i} \in K^1_i$ and $(z^0_i)_i R (z^1_i)_i$. But then necessarily $\pr((z^0_i)_i) = \pr((z^1_i)_i)$, so $C$ is a continuum, indeed.
			
			For \ref{xendp} note that $\{x_i,y_i\}$ can only be the set of the two endpoints of $K_i$. Now clause \ref{Kendp} clearly implies \ref{xendp}.
			
			Pick such an $\sigma_k$ for each $k$. First we are going to check that $(\sigma_k)_k$ is convergent in $\mathcal{C}(P,P)$ (so $\lim_k \sigma_k$ is a surjective continuous function), and that $\lim_k \sigma_k$ is  injective.
			In light of the second part of  Theorem \ref{cauchychar} it suffices to verify that for all finite connected linear graph $A$ and $\varphi \in {\rm Epi}(\bbP,A)$ we have that for all large enough $n,n'$ the relation $(\varphi \circ \sigma_n)  R (\varphi \circ \sigma_{n'})$ holds. Using the fact that every epimorphism factors through a $\pi_j$, it is enough to check that 
			\begin{equation} \label{convprop}  (\pi_{4k-1} \circ \sigma_k) R (\pi_{4k-1} \circ \sigma_{k'}) \ \textrm{ if }k \leq k'. \end{equation}
			A standard induction argument gives that the following claim implies \eqref{convprop}.
			\begin{claim}\label{contract}
				If $\psi \in {\rm Epi}(\bbP, J_{4k+3})$ is such that $\psi R (\pi_{4k+3} \circ \sigma_{k+1})$,
				then 
			$$(\pi_{4k-1,4k+3} \circ \psi) R (\pi_{4k-1} \circ \sigma_{k}).$$
			\end{claim}
			\begin{PROOF}{Claim \ref{contract}}(Claim \ref{contract})
				Using that  $\pi_{4k+3} \circ \sigma_{k+1} = h_{k+1} \circ \pi_{4k+4}$ and $\pi_{4k-1} \circ \sigma_{k} = h_{k} \circ \pi_{4k}$ we need to argue that if $\psi$ is $R$-related to $h_{k+1} \circ \pi_{4k+4}$, then 
				the composition $(\pi_{4k-1,4k+3} \circ \psi)$ is $R$-related to $h_{k} \circ \pi_{4k}$.
				
				We are going to argue that (roughly speaking) as $h_{k+1}$ is a lifting of $f_k$, and $f_k$ is an almost lifting of $g_k$ (within distance one), so $(g_{k} \circ \pi_{4k+2})$ and $(\pi_{4k+1, 4k+3} \circ \psi)$ are within distance two. At this point we will compose these two with $\pi_{4k,4k+1}$, use the contractive nature of these mappings together with the fact that $g_k$ is a lifting of $h_k$.

				More concretely we first argue  that
				\begin{equation} \label{liftdiagram}
					(\pi_{4k+1,4k+3} \circ h_{k+1})  R  (g_k \circ \pi_{4k+2,4k+4}),
				\end{equation}
				which will then immediately imply 
				\begin{equation} 
					 (g_k \circ \pi_{4k+2}) R (\pi_{4k+1,4k+3} \circ h_{k+1} \circ \pi_{4k+4}) R (\pi_{4k+1,4k+3}  \circ \psi).
				\end{equation}

				To see \eqref{liftdiagram}, we note that clause \ref{4k+3comm} implies
				\begin{equation}\label{acomm} \pi_{4k+1,4k+3} \circ h_{k+1} = \pi_{4k+1,4k+2} \circ  f_{k} \circ \pi_{4k+3,4k+4}  
				 \end{equation}
				and \eqref{4k+2acomm} implies
				$$(\pi_{4k+1,4k+2} \circ f_{k} \circ \pi_{4k+3,4k+4}) R (g_{k} \circ \pi_{4k+2,4k+4}),$$
				so \eqref{liftdiagram} holds, indeed.

				By \ref{invh} we can apply $h_{k} \circ h^\bullet_{k} = \pi_{4k-1,4k} \circ \pi_{4k,4k+1} =\pi_{4k-1,4k+1} $ to both sides, to get
					$$ (h_{k} \circ h^\bullet_{k} \circ g_{k} \circ \pi_{4k+2}) R  (\pi_{4k-1,4k+3} \circ \psi)$$
					(here we can write $R$ instead of $R^2$ by \ref{cont}).
					Now using \ref{commn}, the LHS can be simplified to $ (h_{k}  \circ \pi_{4k})$, so
					$$  (h_{k}  \circ \pi_{4k}) R  (\pi_{4k-1,4+3} \circ \psi),$$
				and we are done. \end{PROOF}

			This means that \eqref{convprop} holds, and $\sigma:=\lim_k \sigma_k$ exists (in $\mathcal{C}(P,P)$). We now check that $\sigma$ is a homeomorphism, for which it is enough to see that $\sigma$ is injective.
			\begin{claim}\label{injcl}
				The map	$\sigma=\lim_k \sigma_k \in \mathcal{C}(P,P)$ is injective.
			\end{claim}
			\begin{PROOF}{Claim \ref{injcl}}(Claim \ref{injcl})
			 Pick $(z_i)_i$, $(z'_i)_i$ in $\bbP$ with $\pr((z_i)_i) \neq \pr((z'_i)_i)$. This means that $\neg (z_i R z_i')$ for large enough $i$. With \eqref{convprop} in order to argue that $\sigma(\pr((z_i)_i)) \neq \sigma(\pr((z'_i)_i))$ it is enough to show that for large enough $k$ we have 
			 $$\neg (\pi_{4k-1} \circ \sigma_k((z_i)_i)) \ R^3 \ (\pi_{4k-1} \circ \sigma_k((z'_i)_i)),$$, that is, (by the definition of $\sigma_k$), we need that
			$$\neg (h_{k}(z_{4k}) R^3 h_{k}(z'_{4k})).$$
			Suppose that $k$ is large enough so that 
			\begin{equation} \label{klargeenough} \neg (z_{4k} R z_{4k}'), \end{equation} we claim that $\neg (h_{k+1}(z_{4k+4}) R^3 h_{k+1}(z'_{4k+4}))$. Assume otherwise, so 
			$$h_{k+1}(z_{4k+4}) R^3 h_{k+1}(z'_{4k+4}),$$
			 and therefore
			\begin{equation} \label{apppr1}(\pi_{4k+1,4k+3} \circ h_{k+1}(z_{4k+4})) R (\pi_{4k+1,4k+3} \circ h_{k+1}(z'_{4k+4})) \end{equation}
			by \ref{cont}.
			
			We are going to argue that approximating the $\pi_{4k+1,4k+3}$-image of $h_{k+1}(z_{4k+4})$ and $h_{k+1}(z_{4k+4})$ by $g_{k}(z_{4k+2})$ and $g_{k}(z'_{4k+2})$, the latter two are $R^3$-related, to which we can apply $h^\bullet_k$, and use $h^\bullet_k \circ g_k = \pi_{4k,4k+2}$ to get a contradiction with \eqref{klargeenough}.
			
			To formalize this idea we use \eqref{liftdiagram} to get
			\begin{equation} \label{appppr2} (\pi_{4k+1,4k+3} \circ h_{k+1}(z_{4k+4})) R (g_{k} \circ \pi_{4k+2,4k+4})(z_{4k+4}) = g_{k}(z_{4k+2}), \end{equation} 
			and similarly for $z'$, that is, 
			\begin{equation} \label{appppr3}  (\pi_{4k+1,4k+3} \circ h_{k+1}(z'_{4k+4})) R (g_{k} \circ \pi_{4k+2,4k+4})(z'_{4k+4}) = g_{k}(z'_{4k+2}). \end{equation} 
			So combining\eqref{apppr1}-\eqref{appppr3} we obtain 
			$$g_{k}(z_{4k+2}) R^3 g_{k}(z'_{4k+2}).$$ Then by \ref{contbullet} 
			$$(h^\bullet_{k} \circ g_{k}(z_{4k+2})) R (h^\bullet_{k} \circ g_{k}(z'_{4k+2})).$$
			By \eqref{commn}, $h^\bullet_{k} \circ g_{k}(z_{4k+2}) = z_{4k}$, and $h^\bullet_{k} \circ g_{k}(z'_{4k+2}) = z'_{4k}$, so clearly $z_{4k} R z_{4k}'$, which contradicts \eqref{klargeenough}.
			\end{PROOF}
			
			So $\sigma \in {\rm Homeo}(P)$, indeed, and it remains to check that $C$, $x$, $y$ witness that $\sigma$ has the properties from Lemma \ref{keyle}. First we  recall Proposition~\ref{Ktop}, that is, that sets of the form
			$$ \{ \pr((z_i)_i): \ (z_i)_i \in \bbP, \  z_j R x_j \} \  (j \in \mathbb{N}) $$
			form a neighborhood basis of $x  = \pr((x_i)_i)$, and similarly with $y =\pr((y_i)_i)$. 
			
			Now we can fix the pairwise disjoint regular open sets $U_1,U_2, U_3$ of $P$ as in Lemma \ref{keyle} with $x \in U_2$, $y \in U_3$, $\bigcup_{i=1}^{3} {\rm cl}(U_i) = P$ and $C \cap {\rm cl}(U_1)  = \emptyset$.
			Let $W_i = \pr^{-1}(U_i)$ for $i=1,2,3$.
			By the observation above there exists $i_0$ such that 
			\begin{enumerate} [label = $\boxdot_{\arabic*}$, ref = $\boxdot_{\arabic*}$]
				\item \label{end2} for every $(z_i)_i \in \bbP$, $z_{i_0}Rx_{i_0}$ implies $\pr((z_i)_i) \in U_2$, so $(z_i)_i \in W_2$.
			\end{enumerate}
				Similarly,
			\begin{enumerate}[label = $\boxdot_{\arabic*}$, ref = $\boxdot_{\arabic*}$]
				\stepcounter{enumi}
				\item \label{end3} 	 $z_{i_0}Ry_{i_0}$ implies that $(z_i)_i \in W_3$.
			\end{enumerate}

			Since each point admits a neighborhood which at most two $U_i$ can intersect, we can also assume that $i_0$ is large enough so that for $a \in J_{i_0}$
			$$ \pi^{-1}_{i_0}(a) \textrm{ intersects at most 2 of } \{W_1,W_2,W_3\}.$$
			W.l.o.g.\ we can assume that $i_0 = 4k$ for some  $k \in \mathbb{N}$.
			
			Since $C$ is of positive distance from $U_1 $, we may also assume that whenever $a \in K_{4k}$ we have $\pi_{4k}^{-1}(a)$ intersects only $W_2 \cup W_3$, but cannot intersect $W_1$. Moreover, since $\pr(\pi_{4k}^{-1}(a))$ has nonempty interior in $P$ and $\bigcup_{i=1}^3 U_i$ is dense, at least one of $W_2$ and $W_3$ must intersect $\pi_{4k}^{-1}(a)$. 
			This means that 
				\begin{enumerate}[label = $\boxdot_{\arabic*}$, ref = $\boxdot_{\arabic*}$]
					\setcounter{enumi}{2}
				\item \label{U2U3} if $ a \in L_{4k+2}$, then  $\pi_{4k+2}^{-1}(a)$ can only intersect $W_2$ and $W_3$, and it intersects at least one of them.
			\end{enumerate}

			Moreover, if $L_{4k+2}=\{l_j: \ j<|L_{4k+2}|\}$ (with $l_j R l_{j+1}$), then 
			$$\{ \pi_{4k,4k+2}(l_0), \pi_{4k,4k+2}(l_{|L_{4k+2}|-1})\} = \{x_{4k},y_{4k}\}$$
			by \ref{Kendp}, \ref{Lendp}. Therefore, it follows from \ref{end2}, \ref{end3} together with \ref{cont} that
			\begin{enumerate}[label = $\boxdot_{\arabic*}$, ref = $\boxdot_{\arabic*}$]
				\setcounter{enumi}{3}
				\item $(\pi^{-1}_{4k+2}(l_0) \cup \pi^{-1}_{4k+2} (l_1) \cup \pi^{-1}_{4k+2} (l_2)) \cap (W_1 \cup W_3) = \emptyset$, 
				\item \label{RE} $(\pi^{-1}_{4k+2}(l_{|L_{4k+2}|-1}) \cup \pi^{-1}_{4k+2} (l_{|L_{4k+2}|-2}) \cup \pi^{-1}_{4k+2} (l_{|L_{4k+2}|-3})) \cap  (W_1 \cup W_2) = \emptyset$ 
			\end{enumerate}
			which we can assume  by possibly flipping the order and numbering (and in fact this is not even strict, but we  won't need more than these $3$-$3$ elements).
				We introduce the informal notation 
			$$l_n \dotplus  m = l_{n+m}$$
			if $0 \leq n,n+m < |L_{4k+2}|$.
			Now we recall \ref{slide}, that is, for $a \in L_{4k+3}$ we have ${\rm dist}(f_{k}(a),\pi_{4k+2,4k+3}(a))=2$ unless $\pi_{4k+2,4k+3}(a) \in \{l_0,l_1,l_{|L_{4k+2}|-1}, l_{|L_{4k+2}|-2}\}$.
		This can only happen if  either
			\begin{enumerate}[label = $\boxdot_{\arabic*}$, ref = $\boxdot_{\arabic*}$]
				\setcounter{enumi}{5}
				\item \label{+} for every $a$ with $\pi_{4k+2,4k+3}(a) \notin \{l_0,l_1,l_{|L_{4k+2}|-1}, l_{|L_{4k+2}|-2}\}$ we have 
				$$f_{k}(a) = \pi_{4k+2,4k+3}(a) \dotplus 2,$$
			\end{enumerate} or the other way around, for all but $4$ $a$'s
				$f_{k}(a) = \pi_{4k+2,4k+3}(a) \dotplus (-2)$ holds. We will assume \ref{+}, as the other case is essentially the same.
				We
					\begin{enumerate}[label = $\boxdot_{\arabic*}$, ref = $\boxdot_{\arabic*}$]
					\setcounter{enumi}{6}
					\item \label{mdf}let $m< |L_{4k+2}|$ be largest such that $\pi_{4k+2}^{-1}(l_m) \cap W_2 \neq \emptyset$, 
					\end{enumerate}
				and note that 
				\begin{equation} \label{mbound} 2 \leq m < |L_{4k+2}|-3 \end{equation} holds by \ref{RE} (since the $W_j$'s are pairwise disjoint).
				Pick $\bar z = (z_i)_i \in W_2 \cap \pi_{4k+2}^{-1}(l_m)$ (so $\pr(\bar z) \in U_2$). In the rest of the proof we will show that $ \sigma(\pr(\bar z)) \in U_3$.
				
				 Now $\pi_{4k+3,4k+2}(z_{4k+3}) = z_{4k+2} = l_{m}$, so necessarily 
				\begin{equation} \label{lm+2} f_{k}(z_{4k+3}) = l_{m+2}. \end{equation}
				It is easily seen that $\pi_{4k+2,4k+3} \circ h_{k+1} = f_k \circ \pi_{4k+3,4k+4}(z_{4k+4}) = l_{m+2}$, so 
				$\sigma_{k+1}(\bar z)  \notin W_2$. The purpose of the following claim is to show that not only
				do we have $\sigma_{k+1}(\bar z) \notin (W_1 \cup W_2)$, but points close enough to $\sigma_{k+1}(\bar z)$ lie outside $W_1 \cup W_2$.
					\begin{claim}\label{final}
					
					If $b \in J_{4k+3}$ is such that $b R^3 h_{k+1}(z_{4k+4})$, then $\pi_{4k+3}^{-1}(b) \cap  W_2 = \emptyset$.
				\end{claim}
				\begin{PROOF}{Claim \ref{final}}
					By \ref{mdf} it suffices to show that for any such $b$ 
					$$\pi_{4k+2,4k+3}(b) =l_j \ \textrm{ for some }j>m.$$

					We note that (by \eqref{4k+3comm})
					$$ \pi_{4k+2,4k+3} (h_{k+1}(z_{4k+4})) = f_{k}(z_{4k+3}) = l_{m+2},$$
				 so by \ref{cont} and \eqref{mbound}
					\begin{equation}  h_{k+1}(z_{4k+4}) R^3 b \ \Rightarrow \ \pi_{4k+2,4k+3}(b) \in  \{l_{m+1}, l_{m+2}, l_{m+3} \} \subseteq L_{4k+2}. \end{equation}
				\end{PROOF}
				
					We are going to argue that
					 \begin{equation} \label{inclusion} (b R^2 h_{k+1}(z_{4k+4})) \ \Rightarrow \  \pi_{4k+3}^{-1}(b) \subseteq W_3 \end{equation}
					 (in fact we will only use the statement for $b$'s satisfying $b R h_{k+1}(z_{4k+4})$).
				By the definition of the $W_j$'s (that is, $W_i = \pr^{-1}(U_i)$)  
				\begin{equation} \label{disj1}
					\pr(\pi_{4k+3}^{-1}(b)) \cap U_2 = \emptyset \ \text{ if } b R^3 w,
				\end{equation}
				similarly, by \ref{U2U3}
				\begin{equation} \label{disj2}
					\pr(\pi_{4k+3}^{-1}(b)) \cap U_1  = \emptyset \ \text{ if }b R^3 w.
				\end{equation}
								
				Let $B = \{ b \in J_{4k+3}: \ bR^3 h_{k+1}(z_{4k+4})\}$, $A = J_{4k+3} \setminus B$, and let 
				$$D_X =  \bigcup_{x \in X}  \pr(\pi_{4k+3}^{-1}(x)) \ \textrm{ for } X \subseteq J_{4k+3},$$
				so \eqref{disj1}, \eqref{disj2} can be summarized as
				$$ D_B \cap (U_1 \cup U_2) = \emptyset.$$
				Since $\bigcup_{j=1}^{3} U_j$ is dense in $P$, and  $P \setminus D_A \subseteq D_B$ is open ((i) of Proposition \ref{P:cloint}), $U_3$ is dense in  $P \setminus D_A$.
				But $U_3$ is regular open, so 
				$$P \setminus D_A  \subseteq {\rm int}({\rm cl}(U_3)) =  U_3.$$
	
				Now using that  $D_X \cap D_Y = \emptyset$ if $\neg (X R Y)$, we obtain				 	 
				$$ \textrm{if } bR^2 h_{k+1}(z_{4k+4}), \textrm{ then }   \pr(\pi_{4k+3}^{-1}(b)) \subseteq P \setminus D_A \subseteq U_3,	$$
				which yields \eqref{inclusion} as desired.

				It remains to check that the above statements imply $\sigma(\pr(\bar z)) \in U_3$.
				Recalling \eqref{convprop} and the fact that  $\sigma(\pr(\bar z))$ is the limit of $\pr(\sigma_i(\bar z))$ (in $P$), by compactness of $\bbP$  there exists a sequence $\bar z^* = (z^*_i)_i \in \bbP$ that  is a an accumulation point of $\{(\sigma_i(\bar z)): \ i \in \omega\}$ (in $\bbP$), so necessarily represents $\sigma(\pr(\bar z)) \in P$, that is,  $\pr(\bar z^*) = \sigma(\pr(\bar z))$ and has the property
				 $$z^*_{4k+3} \ R  \ \pi_{4k+3}( \sigma_{k+1}(\bar z)) = h_{k+1}(z_{4k+4})$$
				 (where the equality follows from the way we picked $\sigma_k$, \ref{Hk}).
				 This means that \eqref{inclusion} implies $\bar z^* = (z^*_i)_i \in W_3$, so $\sigma(\pr(\bar z)) =\pr(\bar z^*)   \in U_3$, and we are done.

		\end{PROOF}
		
		The inductive steps (depending on the remainder of $i$ modulo $4$) will rely on Lemmas \ref{step4k}-\ref{step4k+3} below, with Lemma \ref{step4k} handling the case of stepping from $4k$ to $4k+1$ (constructing $J_{4k+1}$, $\pi_{4k,4k+1}$, $h^\bullet_k$ from the inputs $h_k, \pi_{4k-1,4k} \in {\rm Epi}(J_{4k}, J_{4k-1})$), and the last one handling the case of going from $4k+3$ to $4k+4$ (constructing $J_{4k+4}$, $\pi_{4k+3,4k+4}$, $h_{k+1}$).
		Each lemma will rely on Lemma \ref{amalglem}.

				\begin{lemma} \label{step4k}
			Suppose that $J \supseteq K$, $J' \supseteq K'$ are finite linear graphs, $\pi, h: J' \to J$ are epimorphisms, such that 
			\begin{enumerate}
				\item $\pi \rest K' = h \rest K'$, 
				\item $\pi [K'] = K$ ($= h[K']$), $\pi$ maps endpoints to endpoints,
				\item $J'$ is the disjoint union of the distinct nonempty intervals $C'_1$, $C'_2$, $K'$,
				\item $\pi[C'_1] = h[C'_1] = \pi[C'_2] = h[C'_2] = J$.
				\item for $c \in C'_1 \cup C'_2$, if $c R K'$ (in particular, $c$ must be an endpoint of $C_j'$ for $j=1$ or $j=2$), then
				$${\rm tp}^{c,C'_j}(h) = {\rm tp}^{c,C'_j}(\pi).$$ 
			\end{enumerate}
			
			Then, there exist $J'' \supseteq  K''$, and  $\pi', h^\bullet \in {\rm Epi}(J'',J')$ that satisfy the following requirements:
			\begin{enumerate}[label = $\boxplus_\arabic*$, ref = $\boxplus_\arabic*$]
				\item $h \circ h^\bullet = \pi \circ \pi'$,
				\item $J''$ is the disjoint union of the nonempty intervals $K''$, $C_1''$, $C_2''$,
				\item $h^\bullet \rest K'' = \pi' \rest K''$ with $\pi'[K''] = K'$, and these ($\pi' \rest K''$) map endpoints to endpoints,
				\item $\pi'[C_1'']  = h^\bullet[C_1''] = J'$,  $\pi'[C_2''] = h^\bullet[C_2''] = J'$,
				\item for $c \in C''_1 \cup C''_2$, if $c R K''$ (in particular, $c$ is an endpoint of $C_j$ for $j=1$ or $j=2$), then
				$${\rm tp}^{c,C''_j}(h^\bullet) = {\rm tp}^{c,C''_j}(\pi').$$

				\item for each $d \in J'$, $(\pi')^{-1}(d)$ is the disjoint union of intervals, each of them is of length at least $3$, in particular, $a R^3 b$ implies $\pi'(a) R \pi'(b)$.
			\end{enumerate}
		\end{lemma}
		\begin{PROOF}{Lemma \ref{step4k}}(Lemma \ref{step4k})
			First we are going to amalgamate the pair $h$, $\pi$ piecewise on $C_1$, on $K$ and on $C_2$ to get $J^+$ which will be the disjoint union of $K^+$, $C_1^+$, $C_2^+$ (with $K^+ R C_i^+$), the mappings $h^+,\pi^+ \in {\rm Epi}(J^+,J')$ with $\pi \circ \pi^+ = h \circ h^+$, such that $\pi^+[C_i^+] = h^+[C_i^+] = C'_i$,
			$h^+[K^+] = \pi^+[K^+] = K'$, and $h^+ \rest K^+ = \pi^+ \rest K^+$.
			This can be done by invoking the moreover part of Lemma \ref{amalglem} two times, with the roles
			\begin{itemize}
				\item $M = C'_1 = M'$, $f = \pi \rest C'_1$, $f' = h \rest C'_1$, and setting $m_- = m_-'$ to be the node connected with $K$, after which we can let $C_1^+$ be the resulting $O$, $\pi^+ \rest C_1^+$ be the resulting $g$, $h^+ \rest C_1^+$ be the resulting $g'$, 
				\item $M = C_2' = M'$, $f = \pi \rest C'_2$, $f' = h \rest C'_2$, and setting $m_- = m_-'$ to be the node connected with $K$, after which $C_2^+$ will be the resulting $O$.
			\end{itemize}
			We let $K^+$ be isomorphic to $K$ and $\pi \rest K = h^+ \rest K$ be an isomorphism.
			
			Now let $J^+$ be a linear graph which is the disjoint union of the intervals  $C_1^+$, $K^+$, $C_2^+$ (where $K^+$ lies in the middle). Pick a finite linear graph $J''$ which is the disjoint union of the intervals $C_1''$, $K''$, $C_2''$ and $\pi^{++} \in {\rm Epi}(J'', J^+)$ such that $\pi^{++}[K''] = K^+$, $\pi^{++}[C_i''] = J^+$, moreover, 
			\begin{itemize}
				\item  if $C''_1 = \{a_0,a_1, \ldots, a_{|C_1''|-1}\}$ where $K'' R a_0$, $a_i R a_{i+1}$,  then the walk 
				$$\pi^{++}(a_0), \pi^{++}(a_1), \ldots, \pi^{++}(a_{|C''_1|-1})$$ visits all nodes of $C^+_1$ before reaching $K^+$, that is,
					$$ {\rm tp}^{a_0, C_1''}(\pi^{++}) \supseteq \{C^+_1, C^+_1 \cup K^+ \},$$
				\item similarly, if $b_0 \in C''_2$, $b_0 RK''$, then 
					$$ {\rm tp}^{b_0, C_2''}(\pi^{++}) \supseteq \{C^+_2, C^+_2 \cup K^+ \},$$
				\item for every $d \in J^+$, each connected component of $(\pi^{++})^{-1}(d)$ is an interval including at least $3$ nodes.
			\end{itemize}
			One easily checks that $\pi' = \pi^+ \circ  \pi^{++}$, $h^\bullet =h^+ \circ \pi^{++}$ work.

		\end{PROOF}

		\begin{lemma} \label{step4k+1}
			Suppose that $J \supseteq K$, $J' \supseteq K'$ are finite linear graphs, $\pi, h^\bullet: J' \to J$ are epimorphisms, such that 
				\begin{enumerate}
					\item $\pi \rest K' = h^\bullet \rest K'$, 
					\item $\pi [K'] = K$ ($= h^\bullet[K']$), $\pi$ maps endpoints to endpoints,
					\item $J'$ is the disjoint union of the distinct nonempty intervals $C'_1$, $C'_2$, $K'$,
					\item $\pi[C'_1] = h^\bullet[C'_1] = \pi[C'_2] = h^\bullet[C'_2] = J$.
					\item for $c \in C'_1 \cup C'_2$, if $c R K'$ (so then necessarily $c$ is an endpoint of $C'_j$ for $j = 1$ or $2$), then
					$${\rm tp}^{c,C'_j}(h^\bullet) = {\rm tp}^{c,C'_j}(\pi).$$ 
				\end{enumerate}
				
				Then, there exist $J'' \supseteq  K'', L''$, and  $\pi', h \in {\rm Epi}(J'',J')$ that satisfy the following requirements:
				\begin{enumerate}[label = $\boxplus_\arabic*$, ref = $\boxplus_\arabic*$]
					\item $h^\bullet \circ g = \pi \circ \pi'$,
					\item $J''$ is the disjoint union of the nonempty intervals $C_1''$, $L''$, $C_2''$,  $K''$, $C_3''$,
					\item $g \rest K'' = \pi' \rest K''$ with $\pi'[K''] = K'$, and this epimorphism ($\pi' \rest K''$) maps endpoints to endpoints,
					\item $g \rest L'' = \pi' \rest L''$ with $\pi'[L''] = K'$, and this epimorphism ($\pi' \rest L''$) maps endpoints to endpoints,
										
					\item $\pi'[C_i'']  = g[C_i''] = J'$ for $i=1,2,3$,
					\item for $c \in C''_1 \cup C''_2 \cup C_3''$, if $c R (K'' \cup L'')$ (so then necessarily $c$ is an endpoint of $C''_j$ for $j  \in \{1,2,3\}$), then
					$${\rm tp}^{c,C''_j}(g) = {\rm tp}^{c,C''_j}(\pi').$$ 				
					\item  for each $d \in J'$, $(\pi')^{-1}(d)$ is the disjoint union of intervals, each of them is of length at least $3$, in particular, $a R^3 b$ implies $\pi'(a) R \pi'(b)$.
				\end{enumerate}
			\end{lemma}

		\begin{PROOF}{Lemma \ref{step4k+1}}
		Notice that the main difference compared to Lemma \ref{step4k} is that we ``double" $K''$ (and we have $3$ complementer intervals instead of $2$).
		Let $J^+$, $K^+$, $C_1^+$, $C_2^+$, $\pi^+$, $g^+$ be $J''$, $K''$, $C_1''$, $C_2''$, $\pi''$ and $h^\bullet$ given by Lemma \ref{step4k}. Our $J''$ will be an appropriate  extension of $J^+$. 
		
		So let $J''$ be a finite linear graph with an epimorphism $\pi^{++}: J'' \to J^+$ that satisfies 
			\begin{itemize}
				\item $J''$ is the disjoint union of the connected subgraphs $C_1''$, $L''$, $C_2''$, $K''$, $C_3''$, where consecutive ones are connected,
			\item  $\pi^{++}[C''_1]  = \pi^{++}[C''_3] = C^*_1$,
			\item $\pi^{++}[C''_2] = C_2^+$,
			\item $\pi^{++}[K''] = \pi^{++}[L''] = K^+$,
			\item for every $d \in J^+$, each connected component of $(\pi^{++})^{-1}(d)$ is an interval including at least $3$ nodes.
		\end{itemize}
		It is straightforward to check that $\pi' = \pi^+ \circ \pi^{++}$, $g =  g^+ \circ \pi^{++}$ work.
		\end{PROOF}

		\begin{lemma} \label{step4k+2}
			Suppose that $J \supseteq K,L$, $J' \supseteq K', L'$ are finite linear graphs, $\pi, g: J' \to J$ are epimorphisms, such that 
		\begin{enumerate}
			\item $\pi \rest K' = g \rest K'$, 
			\item $\pi [K'] = K$ ($= g[K']$), $\pi$ maps endpoints to endpoints,
			\item $\pi \rest L' = g \rest L'$, 
				\item $\pi [L'] = K$ ($= g[L']$), $\pi$ maps endpoints to endpoints,
			\item $J'$ is the disjoint union of the distinct nonempty intervals $C'_1$, $L'$, $C'_2$, $K'$, $C_3'$ (with consecutive ones being connected to each other),
			\item $\pi[C'_i] = g[C'_i] = J$ for $i=1,2,3$,
			\item  for $x \in C'_i$, if $x R K'$, or $x R L'$ (so then necessarily $x$ is an endpoint of $C_i'$), then
			$${\rm tp}^{x,C'_i}(g) = {\rm tp}^{x,C'_i}(\pi).$$ 
		\end{enumerate}
		
		Then, there exist $J'' \supseteq  K'', L''$, and  $\pi', f, f' \in {\rm Epi}(J'',J)$ that satisfy the following requirements:
			\begin{enumerate}[label = $\boxplus_\arabic*$, ref = $\boxplus_\arabic*$]
			\item \label{comm} $g \circ  \pi' =  \pi \circ f' $,
			\item  $J''$ is the disjoint union of the distinct nonempty intervals $C''_1$, $L''$, $C''_2$, $K''$, $C_3''$, with the consecutive ones being connected,
			
				\item $f' \rest K'' = \pi' \rest K''$ with $\pi'[K''] = K'$, 
			\item $f' \rest L'' = \pi' \rest L''$ with $\pi'[L''] = L'$, 		
			\item \label{Ci'} $\pi'[C_i''] = C_i' = f'[C''_i]$ for $i=1,2,3$,
				\item \label{prei} for each $d \in J'$, $(\pi')^{-1}(d)$ is the disjoint union of intervals, each of them is of length at least $3$, in particular, $a R^3 b$ implies $\pi'(a) R \pi'(b)$.
			\item  for $a \in J'' \setminus L''$ we have $f(a) = f'(a)$,
	
			\item for each $a \in L''$ we have ${\rm dist}(f(a),f'(a)) \leq 2$,
			and if $a \in L''$ is such that $f'(a) = \pi'(a)$ is neither an endpoint of $L'$ nor is related to one then 
			${\rm dist}(f(a),f'(a)) = 2$,

		\end{enumerate}
		\end{lemma}
		\begin{PROOF}{Lemma \ref{step4k+2}}
		  We are going to construct 
		  \begin{itemize}
		  	\item $K''$, $\pi' \rest K''$, $f' \rest K'''$,
		  	\item $L''$, $\pi' \rest L''$, $f' \rest L''$,
		  	\item $C_i''$, $\pi' \rest C_i''$, $f' \rest C_i''$ for $i=1,2,3$ 
		  \end{itemize}
		  separately, in such a way that they map endpoints to endpoints. If these restrictions satisfy \ref{comm}, then (gluing the graphs in the obvious way) we will automatically have \ref{comm}-\ref{Ci'}. 
		  
		  We can apply the moreover part of Lemma \ref{amalglem} separately to the pairs $g \rest C'_i$ and $\pi \rest C'_i$  where $i=1,3$, to obtain $C_i''$, $f' \rest C_i''$, $\pi' \rest C_i''$ satisfying $\pi \circ f' \rest C_i'' = g \circ \pi' \rest C_i''$.
		  
		  Next we apply the main part of Lemma \ref{amalglem}  to $C_2$. 
		  
		  Finally, since  $\pi\rest K' = g \rest K'$, and $\pi\rest L' = g \rest L'$ we can take $K'' = K'$, $L'' = L'$ (and $f' \rest K'' =\pi' \rest K'' = {\rm id}_{K''}$). 
		  
		  Now  $J'' = C''_1 \cup L'' \cup C''_2 \cup K'' \cup C_3''$ (connecting the pieces in this order) will satisfy \ref{comm}-\ref{Ci'}. If $\pi^*: J^* \to J''$ is an epimorphism where each point's preimage is exactly an interval containing $3$ nodes, then replacing $f'$, $\pi'$ with   $f' \circ \pi^*$, $\pi' \circ \pi^*$ and replacing $J''$ with $J^*$ will ensure  \ref{prei}, too.
		  
		  Finally, to get $f$ we first note that $(f')^{-1}(L') = L''$ implies that the endpoints of $L''$ are sent to those of $L'$, that is, if $l''_-, l''_+ \in L''$ are such that $l''_- R C_1''$, $l''_+ R C_2''$, and the endpoints $l'_-,l'_+ \in L'$ are such that $l'_- R C_1'$, $l'_+ R C_2'$, then $f'(l''_-) = l'_-$, $f'(l''_+) = l'_+$ (for this we used also $f'[C_i''] =C_i'$). We also note that by \ref{prei} there is a connected subgraph of $L''$ containing $l''_-$, consisting of at least $3$ elements all of which are mapped to $l'_-$ by $f'$.
		  
		   Identifying $L'$ with $\{0,1,2, \ldots, |L'|-1\}$ (where $l'_- = 0$, $l'_+ = |L'|-1$,  we can define $f \rest L''$ as follows. $$f \rest L'' (a) = \left\{ \begin{array}{ll} f'(a)\ (=0) & \text{ if }a = l''_- \\
		   																f'(a)+1 \ (=1)	& \text{ if } a \neq l''_-, \ a R l''_-	\\
		   																\max(f'(a)+2, |L'|-1), & \text{ otherwise.}									
		   											\end{array}							 \right\}$$
		   	Finally, letting 
		   	$$f \rest J'' \setminus L'' = f' \rest J'' \setminus L'',$$
		   	it is straightforward to check that $f \in {\rm Epi}(J'', J')$ with the desired properties.
		   											
		\end{PROOF}

		\begin{lemma} \label{step4k+3}
		Suppose that $J \supseteq K,L$, $J' \supseteq K', L'$, $\pi, f: J' \to J$ are epimorphisms, such that 
			\begin{enumerate}[label = $(\arabic*)$, ref = $(\arabic*)$]
				\item $\pi \rest K' = f \rest K'$, 
				\item $\pi [K'] = K$ ($= f[K']$), $\pi$ maps endpoints to endpoints,
				\item $\pi [L'] = K$ ($= f[K']$), $\pi$ maps endpoints to endpoints,
				\item \label{J'decomp}  $J'$ is the disjoint union of the distinct nonempty intervals $C'_1$, $L'$, $C'_2$, $K'$, $C_3'$ (with consecutive ones being connected to each other), similarly, $J$ is the disjoint union of  $C_1$, $L$, $C_2$, $K$ and $C_3$,
				\item $\pi[C'_i] = f[C'_i] = C_i$ for $i=1,2,3$,
				
			\end{enumerate}
			Assume moreover, that
			\begin{enumerate}[label = $(\arabic*)$, ref = $(\arabic*)$]
				\setcounter{enumi}{6}
				\item \label{Mf} $M$ is a finite linear graph, $\phi \in {\rm Epi}(M,J')$.
			\end{enumerate}
			
			Then, there exist $J'' \supseteq  K''$, and  $\pi', h \in {\rm Epi}(J'',J')$ that satisfy the following requirements:
				\begin{enumerate}[label = $\boxplus_\arabic*$, ref = $\boxplus_\arabic*$]
				\item $f \circ \pi' = \pi \circ h$,
				\item there exists $\phi^* \in {\rm Epi}(J'',M)$ with $\phi \circ \phi^* = \pi'$,
				\item $J''$ is the disjoint union of $C_1''$, $K''$, $C_2''$ (consecutive ones are connected),
				\item $\pi' \rest K'' = h \rest K''$, mapping endpoints to endpoints,
				\item $\pi'[K''] = K' (= h[K''])$,
				\item \label{ctp} for $i=1,2$ 
				$$\pi'[C_i''] = h[C_i''] = J',$$ and if $c \in C_i''$ satisfies $c R K''$, then
				$$ {\rm tp}^{c,C_i''}(h) =  {\rm tp}^{c,C_i''}(\pi').$$
				\item \label{preii} for each $d \in J'$, $(\pi')^{-1}(d)$ is the disjoint union of intervals, each of them is of length at least $3$, in particular, $a R^3 b$ implies $\pi'(a) R \pi'(b)$.
			\end{enumerate}
			
		\end{lemma}
		\begin{PROOF}{Lemma \ref{step4k+3}}
			The finite linear graph $J''$ will be the result of three successive extension of $J'$.
			We are going to define
			\begin{itemize}
				\item $J^+$ and $\pi^+, h^+ \in {\rm Epi}(J^+,J')$ with $f \circ \pi' = \pi \circ h^+$,
				\item $J^{++}$ and $\pi^{++} \in {\rm Epi}(J^{++},J^+)$,
				\item and finally $J''$, $\pi'' \in {\rm Epi}(J'',J^{+})$,
			\end{itemize} 
			and we will let $\pi' =  \pi^{+} \circ \pi^{++} \circ \pi''$, $h = h^+ \circ \pi^{++} \circ \pi''$.
			
			The extension $J^+$ will be the disjoint union of the finite linear graphs $C_1^+$, $L^+$, $C_2^+$, $K^+$, $C_3^+$, which are connected to each other in this order. 
			$C_1^+$, $\pi^+ \rest C_1^+$, $h^+ \rest C_1^+$ are gotten by applying Lemma \ref{amalglem} to $g \rest C_1$ and $\pi \rest C_1$, and similarly, we obtain 	$D \in \{L^+,C_2^+,C_3^+ \}$, $\pi^+ \rest D$, $h^+ \rest D$ by the same way, while we can let $K^+$ be isomorphic to $K'$ (since $f \rest K' = \pi \rest K'$), so 
			\begin{enumerate}[label = $\blacksquare_\arabic*$, ref = $\blacksquare_\arabic*$]
				\item \label{D} for $D \in \{C_1^+,L^+,C_2^+,K^+,C_3^+ \}) \ \pi^+[D]= h^+[D]$, 	and $\pi^+$, $h^+$ map endpoints of $D$ to endpoints of $\pi^+[D]$ (which is $C_1^+$ if $D = C_1^+$, etc.),
				\item $\pi^+ \rest K^+ = h^+ \rest K^+$,
				\item \label{DE} in particular, if $D \neq E \in \{C_1^+,L^+,C_2^+,K^+,C_3^+ \}$, then $\pi^+[D] \cap \pi^+[E] =  \emptyset$,
			\end{enumerate}
				
			Next, we define an extension $J^{++}$ of $J^+$ as follows. The pair $J^{++}$ and $\pi^{++} \in {\rm Epi}(J^{++},J^+)$ will be an unfolding of $J^+$, which is in fact uniquely determined (up to isomorphism) by the requirements that
			\begin{itemize}
				\item $J^{++}$ is the disjoint union of $D_i$, $i=1,2, \ldots, 8$ with $D_i R D_{i+1}$,
				\item for each $i$ the map $\pi^{++} \rest  D_i$ is a bijection,
				\item $\pi^{++} [D_1] = \pi^{++}  [D_8] = J^{+}$,
				\item $\pi^{++} [D_2] = \pi^{++}[D_7] = C_3^{+}$,
				\item $\pi^{++} [D_i] = K^+$ for $i=3,4,5,6$.
			\end{itemize} 
			For future reference we remark that
			\begin{equation} \label{D8}
				(\pi^{++} \rest D_8 {\textrm{ is an isomorphism with }} J^+) \ \& \ (d \in D_8 \wedge   d R D_7 \ \to \pi^{++}(d) \in C_3^+),
			\end{equation}
			\begin{equation} \label{D1}
				(\pi^{++} \rest D_1 {\textrm{ is an isomorphism with }} J^+) \ \& \ (d \in D_1 \wedge  d R D_2 \ \to \pi^{++}(d) \in C_3^{+}),
			\end{equation}

			Next, we are going to use $M$ and $f$ from \ref{Mf} to construct $J''$.
			W.l.o.g.\ we can assume that $f$ factors through $\pi^{+} \circ \pi^{++}$, or simply $f$ maps onto $J^{++}$, so we can pick $J''$, $f^+ \in {\rm Epi}(J'',M)$, $\pi'' \in {\rm Epi}(J'',J^{++})$ with $\pi'' = f \circ f^+$. We can assume that $(\pi'')^{-1}(a)$ consists of intervals each of length at least $3$ whenever $a \in J^{++}$.
			
			Finally, we claim that letting $\pi' = \pi^+ \circ \pi^{++} \circ \pi''$, $h = h^+ \circ \pi^{++} \circ \pi''$, and $K''$ be any interval in $J''$ with
			\begin{itemize}
				\item $\pi''[K''] = D_4$, and $K''$ is maximal such interval (that is, $a \notin K''$, $aR K''$ implies $\pi''(a) \in D_3 \cup D_5$),
				\item if $C \subseteq J'' \setminus K''$ is maximal connected, then $\pi'' \rest C$ is surjective
			\end{itemize} 
			will work (that is, \ref{ctp} is satisfied). We also remark that passing to a further extension of $J''$, if necessary and replacing $J''$, $\pi''$, without loss of generality such $K''$ exists.
			
			So let $C \subseteq J'' \setminus K''$ be a maximal connected subgraph,
			 $c \in C$ be such that $c R K''$, we need to check that
				$$ {\rm tp}^{c,C}(h) =  {\rm tp}^{c,C}(\pi').$$
			It is easy to see that $\pi''(c) \in D_3 \cup D_5$.
			First we assume that $\pi''(c) \in D_3$, and	we let $a$ denote the unique node in $K^+$ with $a R C_2^+$. By the way $\pi^{++}: J^{++} \to J^+$ is defined one can check that
			$$\pi^{++} \circ \pi''(c) = a.$$
			
				We claim that 	
			\begin{equation}\label{cC} {\rm tp}^{c,C}(\pi^{++} \circ \pi'') \supseteq \{\{a\},K^+, K^+ \cup C_3^+, K^+ \cup C_3^+ \cup C_2^+, K^+ \cup C_3^+ \cup C_2^+ \cup L^+\}. \end{equation}
			Clearly ${\rm tp}^{c,C}(\pi'')$ contains a set $H$, such that $H$ is either the entire $D_3$ and a subset of $D_4 \cup D_5 \cup D_6$, or $H$ contains the entire $D_6$, and a subset of $D_3 \cup D_4 \cup D_5$. Either case $\pi^{++}[H] = K^+$.
			Similarly, ${\rm tp}^{c,C}(\pi'')$ contains a set $H'$ that is either the union of $D_2 \cup D_3$ and a subset of $D_4 \cup D_5 \cup D_6 \cup D_7$, or $H$ is the union of the full $D_6 \cup D_7$ and a subset $D_2 \cup D_3 \cup D_4 \cup D_5$. In either case, $\pi^{++}[H] = K^+ \cup C_3^+$. This implies that 
			$${\rm tp}^{c,C}(\pi^{++} \circ \pi'') \supseteq \{\{a\},K^+, K^+ \cup C_3^+\}.$$
			Now recalling \eqref{D8} and \eqref{D1} it is easy to check that \eqref{cC} holds, indeed.
			Finally, $k^+_-$ is an endpoint of $K^+$, so by \ref{D} and \ref{DE} the node $\pi^+(k^+_-) = \pi'(c)= h^+(k^+_-) = h'(c)$ is an endpoint of $K'$, so applying $\pi^+$ and $h^+$ to \eqref{cC},
			$${\rm tp}^{c,C}(\pi^+\circ \pi^{++} \circ \pi'') \supseteq \{\{\pi^+(a)\},K', K' \cup C_3', K' \cup C_3' \cup C_2', K' \cup C_3' \cup C_2' \cup L'\},$$
			and 
			$${\rm tp}^{c,C}(h^+ \circ \pi^{++} \circ \pi'') \supseteq \{\{\pi^+(a)= h^+(a)\},K', K' \cup C_3', K' \cup C_3' \cup C_2', K' \cup C_3' \cup C_2' \cup L'\},$$ which by our assumption \ref{J'decomp} on the structure of $J'$ uniquely determine ${\rm tp}^{c,C}(\pi^+\circ \pi^{++} \circ \pi'') $ and ${\rm tp}^{c,C}(h^+\circ \pi^{++} \circ \pi'')$, in particular these two coincide.
			
			On the other hand, if $\pi''(c) \in D_5$, and we let $b$ be the unique node in $K^+$ with $b R C_3^+$, then the same argument shows 
			\begin{equation} {\rm tp}^{c,C}(\pi^{++} \circ \pi'') \supseteq \{\{b\},K^+, K^+ \cup C_3^+, K^+ \cup C_3^+ \cup C_2^+, K^+ \cup C_3^+ \cup C_2^+ \cup L^+\}, \end{equation}
			to which  one can similarly apply $h^+$ and $\pi^+$, and conclude
			$${\rm tp}^{c,C}(\pi^+\circ \pi^{++} \circ \pi'') = {\rm tp}^{c,C}(h^+\circ \pi^{++} \circ \pi'').$$		\end{PROOF}

	\end{PROOF}

		\appendix

\section{Projective Fra{\"i}ss{\'e} theory}\label{A:projf}

 We recall here the framework of projective Fra{\"i}ss{\'e} theory. The theory was introduced in \cite{IS}. The description below is a generalized version of \cite{IS}. Propositions~\ref{P:small}--\ref{P:inj} are new.

 \subsection*{Category $\cK$}

		We fix a symbol $R$. By an {\bf interpretation of} $R$ on a set $X$ we understand a binary relation 
$R^X$ on $X$, that is, $R^X\subseteq X\times X$. 
We say that $R^X$ is a {\bf reflexive graph} if it is reflexive and symmetric as a binary relation. Assume $X$ and $Y$ are equipped with interpretations $R^X$ 
and $R^Y$ of $R$. 
Then a function $f\colon X\to Y$ is called a {\bf strong homomorphism} 
if
\begin{enumerate}
\item[---] for all $x_1, x_2\in X$, $x_1R^X x_2$ implies $f(x_1)R^Y f(x_2)$;  

\item[---] for all $y_1, y_2\in Y$, $y_1R^Y y_2$ implies that there exist $x_1, x_2\in X$ with $y_1= f(x_1)$, $y_2=f(x_2)$, and $x_1 R^X x_2$. 
\end{enumerate}
By reflexivity of $R^Y$, a strong homomorphism is a surjective function from $X$ to $Y$. We will often skip the superscript $X$ in $R^X$ trusting that the context determines which interpretation of $R$ we have in mind.

	We have a category $\mathcal K$ each of whose objects is a finite set equipped with an interpretation of $R$ as a reflexive graph, all of whose morphisms are strong homomorphisms, and the following condition holds 
\begin{enumerate}
\item[(o)] if $A, B, C$ are objects in $\mathcal K$, $f\colon A\to B$ and $h\colon A\to C$ are morphisms in $\mathcal K$, and $g\colon B\to C$ is a function such that 
$h=g\circ f$, then $g$ is a morphism in $\mathcal K$. 
\end{enumerate}

	We say that $\mathcal K$ is a {\bf projective Fra{\"i}ss{\'e} class} if it fulfills the following two conditions: 
\begin{enumerate}[label = $({\rm F}\arabic*)$, ref = $({\rm F}\arabic*)$]
\item\label{F1} for any two objects $A,B$ in $\mathcal K$, there exist an object $C$ in $\mathcal K$ and morphisms $f\colon C\to A$ and $g\colon C\to B$ in $\mathcal K$; 

\item\label{F2} for any two morphisms $f, g$ in $\mathcal K$ with the same codomain, there exist morphisms $f', g'$ in $\mathcal K$ such that $f\circ f'= g\circ g'$. 
\end{enumerate} 
Condition (i) is called the {\bf joint projection property} and condition (ii) is called the {\bf projective amalgamation property}.

\subsection*{Category $\cK^*$ and projective Fra{\"i}ss{\'e} limit}
	
	We consider a category $\cK^*$ whose objects are totally disconnected compact metric spaces $\mathbb{K}$ taken together with sets ${\rm Epi}(\bbK, A)$ of continuous surjective functions from $\bbK$ to $A$, for $A\in \cK$, with the following properties:
	\begin{enumerate}[label = $({\rm A}\arabic*)$]
	\item \label{factor}
	for each continuous function $\psi\colon \bbK\to X$, where $X$ is a finite topological space, there exist $\varphi\in {\rm Epi}(\bbK, A)$ and a function $f\colon A\to X$, for some $A\in \cK$, such that 
	\[
	\psi = f\circ \varphi; 
	\] 
	
	\item\label{cofin} for all $\phi\in {\rm Epi}(\bbK, A)$ and $\psi\in {\rm Epi}(\bbK, B)$, for some $A, B\in \cK$, there exists $\chi\in {\rm Epi}(\bbK, C)$, $f\in {\rm Epi}(C, A)$, and $g\in {\rm Epi}(C,B)$, for some $C\in \cK$,  such that 
	\[
	\phi= f\circ\chi\;\hbox{ and }\; \psi=g\circ \chi; 
	\]
		 \item \label{good} for a strong homomorphism $f\colon A\to B$, for some $A,B\in {\mathcal K}$, we have 
		\[
		\begin{split} 
		\big(f\circ\varphi \in {\rm Epi}(\bbK, B),\,&\hbox{ for some $\varphi \in {\rm Epi}(\bbK, A)$}\big) \; \Rightarrow\;  f\in {\rm Epi}(A, B)\\		
	&\Rightarrow\;\big(f\circ\varphi \in {\rm Epi}(\bbK, B),\hbox{ for all $\varphi \in {\rm Epi}(\bbK, A)$}\big). 
	\end{split} 
		\]
	\end{enumerate} 
	Conditions \ref{factor}, \ref{cofin}, \ref{good} assert that elements of $\cK^*$, that is, $\bbK$ together with ${\rm Epi}(\bbK, A)$, for $A\in \cK$, can be viewed as inverse limits of sequences consisting of structures in $\cK$ with bonding maps being morphisms in $\cK$. 

A morphism in $\cK^*$ is a continuous surjection $\sigma\colon \bbK\to \bbK'$ such that $\varphi\circ \sigma\in {\rm Epi}(\bbK, A)$, for each $\varphi\in {\rm Epi}(\bbK', A)$ with $A\in \cK$. By 
\[
{\rm Aut}(\bbK)
\]
we denote the group of all invertible morphisms $\bbK\to\bbK$, that is, all homeomorphisms $\sigma\colon \bbK\to\bbK$ such that both it and $\sigma^{-1}$ are morphisms in $\cK^*$.

	\begin{proposition} \label{hype} Let $\cK$ be a countable projective Fra{\"i}ss{\'e} class. There exists an object $\bbK_\infty$ in $\cK^*$ such that 
		\begin{enumerate}[label = $({\rm P}\arabic*)$, ref = $({\rm P}\arabic*)$]
			\item\label{P1} (projective universality) ${\rm Epi}(\bbK_\infty,A)\not=\emptyset$, for each $A \in \cK$, 
			\item\label{P2} (projective ultrahomogeneity) for each $A \in \cK_\infty$, $\varphi, \psi  \in {\rm Epi}(\bbK_\infty,A)$ there exists $\sigma \in  {\rm Aut}(\bbK_\infty)$ such that $\varphi = \psi \circ \sigma$.
		\end{enumerate}		
		The object in $\cK^*$ with properties \ref{P1} and \ref{P2} is unique up to isomorphism. 
	\end{proposition}
	
	The unique object $\bbK_\infty$ in Proposition~\ref{hype} above is called the {\bf projective Fra{\'i}ss{\'e} limit} of $\cK$. 
	
	The proposition is proved by constructing  structures $A_n$ in $\cK$ and $\pi_n\colon A_{n+1}\to A_n$ in ${\rm Epi}(A_{n+1}, A_n)$, $n\in \bbN$, such that 
	\begin{enumerate}[label = $({\rm G}\arabic*)$, ref = $({\rm G}\arabic*)$]
	\item\label{G1} for each $A\in \cK$, there exists $m\in \bbN$ with ${\rm Epi}(A_m, A)\not=\emptyset$; 
	
	\item\label{G2} for each $A\in \cK$ and $f\in {\rm Epi}(A, A_m)$, for some $m\in \bbN$, there exists $n> m$ and $g\in {\rm Epi}(A_n, A)$ such that 
	\[
	f\circ g = \pi_m\circ \cdots \circ \pi_{n-1}.
	\]
	\end{enumerate} 
	The construction is done by induction using countability of $\cK$ and properties \ref{F1} and \ref{F2}. 
	One lets $\bbK_\infty = \projlim_n (A_n, \pi_n)$. One then has the canonical projection maps $\pi^\infty_n\colon \bbK_\infty \to A_n$, and defines, for $A\in \cK$, 
	\[
	{\rm Epi}(\bbK_\infty, A) = \{ f\circ \pi^\infty_n\mid f\in {\rm Epi}(A_n, A) \hbox{ for some } n\in \bbN\}. 
	\]
	A sequence with properties \ref{G1} and \ref{G2} is called a {\bf generic sequence} for $\cK$. One checks properties \ref{factor}--\ref{good} and \ref{P1}, \ref{P2}. Property (o) of the Fra{\"i}ss{\'e} class $\cK$ is not used in the construction; it is used to check the first implication in \ref{good}.

	Define a binary relation $R^\bbK$ on any $\bbK\in \cK^*$ by letting 
\[
x R^\bbK y \;\hbox{ iff }\; \big( \varphi(x)R^A\varphi(y), \hbox{ for all } A\in \cK \hbox{ and all } \varphi\in {\rm Epi}(\bbK, A)\big). 
\]
We note that $R^\bbK$ is a compact symmetric and reflexive binary relation on $\bbK$. We also note that all elements of ${\rm Aut}(\bbK)$ are isomorphisms of the structure 
$(\bbK, R^\bbK)$.

The following proposition is useful in applications.

\begin{proposition}\label{P:small} 
Conditions (i)--(iii) below are equivalent to each other. 
\begin{enumerate} 
\item[(i)] For each non-empty clopen set $V\subseteq \bbK$, there is $x\in V$ such that $R^\bbK(x)\subseteq V$.

\item[(ii)] The set $\{ x\in \bbK\mid R^\bbK(x) = \{ x\}\}$ is dense in $\bbK$.

\item[(iii)] For each clopen set $V\subseteq \bbK$, the set $R^\bbK(V)\setminus V$ is nowhere dense in $\bbK$.
\end{enumerate}


Condition (iv) below implies conditions (i)--(iii). 

\begin{enumerate} 
\item[(iv)] for all $a\in A\in \cK$, there exist $b\in B\in \cK$ and an epimorphism $f:B\to A$ in $\cK$ such that $f(R^B(b)) = \{ a\}$.  
\end{enumerate} 
\end{proposition} 

\begin{proof} We start by showing that conditions (i)--(iii) are equivalent to each other. These arguments use only the fact that $R^\bbK$ is a compact symmetric relation on $\bbK$. 

(i)$\Rightarrow$(ii). Fix a metric on $\bbK$. If (ii) fails, then there exists a non-empty clopen set $U\subseteq\bbK$ such that ${\rm diam}(R^\bbK(x))>0$ for all $x\in U$. Observe that, by compactness of $\bbK$, for each $r>0$, the set 
\[
\{ x\in \bbK\mid {\rm diam}(R^\bbK(x))\geq r\}
\]
is closed. So, by the Baire category theorem, there exists a non-empty clopen set $U'\subseteq U$ and $r>0$ such that ${\rm diam}(R^\bbK(x))\geq r$ for all $x\in U'$. Let $V\subseteq U'$ be a non-empty clopen set with 
${\rm diam}(V)<r$. It is clear that (i) fails for this $V$. The implication is proved. 

(ii)$\Rightarrow$(iii). Let $V\subseteq \bbK$ be a clopen set. Let $U=\bbK\setminus V$. By (ii) the set $U \setminus R^\bbK(V)$ is dense in $U$. Since $R^\bbK(V)$ is closed, we see that $R^\bbK(V)\cap U$ is nowhere dense, but this is (iii). 

(iii)$\Rightarrow$(i). Let $V\subseteq \bbK$ be non-empty and clopen set. An application of (iii) to $\bbK\setminus V$ gives 
\[
R^{\bbK}(\bbK\setminus V)\cap V = R^{\bbK}(\bbK\setminus V)\setminus (\bbK\setminus V) \not\supseteq V.   
\]
So we have $R^{\bbK}(\bbK\setminus V)\not\supseteq V$. Any point $x\in V\setminus R^{\bbK}(\bbK\setminus V)$ satisfies (i) since $R^\bbK$ is symmetric.

We now show that (iv) implies conditions (i)--(iii). 

(iv)$\Rightarrow$(i). Let $V\subseteq \bbK$ be a non-empty clopen set. There exist $a\in A\in \cK$ and an epimorphism $\phi\colon\bbK\to A$ in $\cK^*$ such that
\[
V\supseteq \{ x\in \bbK\mid \phi(x)=a\}. 
\]
Let $b\in B\in \cK$ and an epimorphism  $f\colon B\to A$ in $\cK$ be chosen for $a$ and $A$ as in point (iv). Now, let $\psi\colon \bbK\to B$ be 
an epimorphism in $\cK^*$ such that 
\[
\phi= f\circ \psi. 
\]

Pick $x\in \bbK$ such that $\psi(x)=b$. We claim that $x$ is as required by the conclusion of (i). 
Let $y\in \bbK$ be such that $xR^\bbK y$. We need to show that $y\in V$, for which it is enough to prove that $\phi(y)=a$. We have 
\[
\psi(x) R^B \psi(y)\;\hbox{ and }\; \psi (x)=b, 
\]
so, by the choice of $f$,
\[
f(\psi (y)) = a. 
\]
Since $f(\psi(y)) = \phi(y)$, it follows that $\phi(y)=a$, and the conclusion follows. 
\end{proof}

\subsection*{The canonical quotient space of a transitive class $\cK$}

We will abandon the subscript in the notation $\bbK_\infty$ for the projective Fra{\"i}ss{\'e} limit of $\cK$.

We say that the Fra{\"i}ssé class $\mathcal{K} $ is {\bf transitive} if $R^{\bbK}$ is a transitive relation on the projective Fra{\"i}ss{\'e} limit $\bbK$ of $\cK$. Transitivity of $R^{\bbK}$ implies that it is a compact equivalence relation on $\bbK$ since $R^{\bbK}$ is compact, symmetric, and reflexive by its very definition. 

Assume $\cK$ is a transitive projective Fra{\"i}ss{\'e} class. Then
\[
K= \bbK/R^\bbK
\]
with the quotient topology is a compact metric space. We call it the {\bf canonical quotient space} of $\cK$. Let 
\[
{\rm pr}\colon \bbK\to K
\]
be the quotient map, which we call the {\bf canonical projection}.

\begin{proposition} \label{Ktop} Assume $\cK$ is a countable transitive projective Fra{\"i}ss{\'e} class. 
	If $x \in \bbK$, then sets of the form
	$$\{ \pr(y): \ y \in \bbK, \ \varphi(x) R \varphi(y)\}, \ \varphi \in {\rm Epi}(\bbK,A), \ A \in \cK$$
	form a neighborhood basis of $\pr(x) \in K$.
\end{proposition}

\begin{PROOF}{Proposition \ref{Ktop}}     For a fixed $\varphi \in {\rm Epi}(\bbK,A)$ let $A_0 \subseteq A$ be $A_0 = \{a \in A: \ \varphi(x) R a\}$.
	Clearly 
	$$\{ \pr(y): \ y \in \bbK, \ \varphi(y) \in A_0 \} \supseteq K \setminus \{ \pr(y): \ y \in \bbK, \ \varphi(y) \notin A_0\} \ {\not\owns} \ \pr(x)$$
	since for any $y \in \bbK$ $\neg (\varphi(y) R \varphi(x))$ implies $\neg x  R y$, so $\pr(x) \neq \pr(y)$. As $\{ \pr(y): \ y \in \bbK, \ \varphi(y) \notin A_0\}$ is closed, 
	$\{ \pr(y): \ y \in \bbK, \ \varphi(y) \in A_0 \}$ is a neighborhood of $\pr(x)$, indeed.
	
	Suppose that $U \subseteq K$ is open and $\pr(x) \in U$. Since there are only countably many epimorphisms in $\cK$, by the amalgamation property of $\cK$ (and by \ref{P1}, \ref{P2}), we can enumerate a cofinal system $(\varphi_i)_i$ of epimorphisms, in the sense that each epimorphism factors through all but finitely many $\varphi_i$.  
	
	Suppose that $U$ does not contain any subset of the prescribed form, in particular for each $\varphi_i$ there exists  $y_i \in \bbK$ with $\pr(y_i) \notin U$, $\varphi_i(y_i) R \varphi_i(x)$.
	Since every epimorphism from $\bbK$ into an object in $\cK$ factors through a $\varphi_j$, w.l.o.g.\ for each $i \leq j$ there exists $\pi$ such that $\varphi_i = \pi \circ  \varphi_j$, and by a standard compactness argument  we can assume $\varphi_i(y_{j}) = \varphi_i(y_i))$ for $i\leq j$. By compactness of $\bbK$, we can assume that $(y_{i})_i$ converges to some $y \in \bbK$ (in fact the convergence of $(\varphi_i(y_j))_j$ already implies this), so by continuity, for each $i$, 
	 $$\varphi_i(y) = \varphi_i(\lim_j(y_{j})) = \lim_j(\varphi_i(y_j)) = \varphi_i(y_i),$$
	 in particular, $\varphi_i(y) R \varphi_i(x)$. Since every $\varphi \in {\rm Epi}(\bbK,A)$ factors through some $\varphi_i$ via an epimorphism we obtain that  $\varphi(x) R \varphi(y)$ holds (for every $\varphi$).
	So $x R^\bbK y$, and $\pr(x) = \pr(y) = \lim_i \pr(y_i) \in {\rm cl}(K \setminus U)= K \setminus U$,  contradicting that $\pr(x) \in U$ is open.
\end{PROOF}

\begin{proposition}\label{P:cloint} Assume $\cK$ is a countable transitive projective Fra{\"i}ss{\'e} class. Assume that $\cK$ fulfills condition (iv) (or just conditions (i)--(iii)) of Proposition~\ref{P:small}. 
Let $\phi\colon \bbK\to A$ be an epimorphism in $\cK$ and let $S\subseteq A$. Then 
\begin{enumerate} 
\item[(i)] ${\rm int}\big({\rm pr}\big(\phi^{-1}(S)\big)\big) = K\setminus {\rm pr}\big(\phi^{-1}(A\setminus S)\big)$;

\item[(ii)] ${\rm pr}\big(\phi^{-1}(S)\big)$ is a regular closed subset of $K$.  
\end{enumerate} 
\end{proposition} 

\begin{proof} Set $D_S= \phi^{-1}(S)$. 

(i) We have 
		\begin{equation}\label{E:fundcl} 
		{\rm int}({\rm pr}(D_S))\cap {\rm pr}(D_{A\setminus S})=\emptyset.
		\end{equation} 
		Otherwise there is $z\in D_{A\setminus S}$ with ${\rm pr}(z)\in {\rm int}({\rm pr}(D_S))$. By continuity of $\rm pr$, if $U\subseteq \bbK$ is a small enough clopen set containing $z$, then  
		\[
		\emptyset\not= U\subseteq D_{A\setminus S}\;\hbox{ and }\;{\rm pr}(U)\subseteq {\rm pr}(D_S). 
		\]
		So, for this clopen set $U$, we have 
		\[
		\emptyset\not= U\subseteq D_{A\setminus S}\;\hbox{ and }\; R^{\bbK}(y)\cap D_S\not=\emptyset,\hbox{ for each }y\in U.
		\]
		Since $D_S\cap D_{A\setminus S}=\emptyset$, this implies
		\[
		R^\bbK( \bbK\setminus U)\supseteq U\not=\emptyset, 
		\]
		which, with $V=\bbK\setminus U$, contradicts condition (iii) of Proposition~\ref{P:small}. 
		We get (i), that is, 
		\begin{equation}\notag
		{\rm int}(\pr(D_S)) = K \setminus \pr(D_{A \setminus S}), 
		\end{equation} 
		with the inclusion $\subseteq$ being a rephrasing of \eqref{E:fundcl} and the inclusion $\supseteq$ following from $D_S\cup D_{A\setminus S}= \bbP$, continuity of the function $\rm pr$, and compactness of $D_{A\setminus S}$. 
		
		(ii)  Applying (i) to obtain the first and third equalities, we get 
		\[
		{\rm cl}\big({\rm int}(\pr(D_S))\big) = {\rm cl}(K \setminus \pr(D_{A \setminus S}))= K\setminus  {\rm int}( \pr(D_{A \setminus S}))={\rm pr}(D_S).
		\]
\end{proof}

There is a natural continuous homomorphism 
\[
\pr\colon {\rm Aut}({\mathbb K})\to {\rm Homeo}(K) 
\]
induced by the projection $\pr\colon \bbK \to K$. Namely, given 
	$f\in {\rm Aut}(\bbK)$ and $x\in K$, we fix $p\in \bbK$ with $x=\pr(p)$ and let 
	\[
	\pr(f)(x) = \pr\big( f(p)\big).
	\]
It is now easy to check that, since $f$ is an automorphism of $\bbK$, the value of $\pr(f)(x)$ does not depend on the choice of $p$. It is also easy to see that $\pr(f)$ is continuous and bijective, so it is a homeomorphism of $K$. Continuity of $\pr\colon  {\rm Aut}({\mathbb K})\to {\rm Homeo}(K)$ is then easy to check. It is clear that, for $f,g\in {\rm Aut}(\bbK)$, we have 
\begin{equation}\label{E:preq} 
\pr(f)=\pr(g) \;\hbox{ iff }\; \forall x,y\in \bbK \; \big(xR^{\bbK} y \Rightarrow f(x)R^\bbK g(y)\big). 
\end{equation}

We register the following proposition. 

\begin{proposition}\label{P:inj} 
Let $\cK$ be a countable transitive projective Fra{\"i}ss{\'e} class. Assume $\cK$ fulfills condition (iv) (or just conditions (i)--(iii)) of Proposition~\ref{P:small}. Then the canonical homomorphism 
${\rm pr}\colon {\rm Aut}(\bbK)\to {\rm Homeo}(K)$ is injective. 
\end{proposition} 

\begin{proof} We will use condition (ii) of Proposition~\ref{P:small}. 
Let $f,g\in {\rm Aut}(\bbK)$ be such that ${\rm pr}(f)={\rm pr}(g)$. Let 
\[
T= \{ x\in \bbK\mid R^\bbK(x)=\{ x\}\}.
\]
Since $f$ and $g$ are automorphisms of $(\bbK, R^\bbK)$, we see that $f(T)=g(T)=T$. So, by the assumption ${\rm pr}(f)={\rm pr}(g)$ and condition \eqref{E:preq}, we have that $f(x)=g(x)$ for all $x\in T$. Since, by our assumption (ii), $T$ is dense in $\bbK$, we get $f=g$. 
\end{proof}

\section{An amalgamation lemma}\label{A:amal} 

In this section, we state and prove two amalgamation lemmas. The first one of which was known to the second author for a long time and both are special cases of a result from an unpublished work by Solecki and Tsankov. For the sake of completeness, we provide proofs of these particular cases here.

The first lemma will be needed in Section~\ref{Su:densp}. 

\begin{lemma}[Solecki]\label{amalgamp} 
Supposed that $I$,  $J$, $J'$ be finite linear graphs (objects in $\cP$), 
 $f: J \to I$, $f': J' \to I$ are epimorphisms (in $\cP$), and $a_J$ and $a_{J'}$ are endpoints of $J$ and $J'$, respectively. If 
 \[
 {\rm tp}^{a_J,J}(f) = {\rm tp}^{a_{J'},J'}(f'),
 \]
then there exist epimorphisms $g: L \to J$, $g': L \to J'$ and an endpoint $a_L$ of $L$ with 
	$$f \circ g = f' \circ g',$$
	and 
	 \[
	 g(a_L) = a_J,\; g'(a_L) = a_{J'}.
	 \]
\end{lemma} 

Lemma~\ref{amalgamp} will follow from Lemma~\ref{amalglem} below. We will need Lemma~\ref{amalglem} also for the construction in Section~\ref{s4}. 

	\begin{lemma}[Solecki--Tsankov]\label{amalglem}
	Suppose that 
	\begin{enumerate}[label = $(\arabic*)$, ref =  $(\arabic*)$ ]
		\item[]  $I$,  $J$, $J'$ are finite linear graphs (objects in $\cP$),
		\item[]  $f: J \to I$, $f': J' \to I$ are epimorphisms (in $\cP$), 
		\item[] \label{L3} $a_J$, $b_J$ ($a_{J'}$, $b_{J'}$, resp.) are endpoints of $J$ ($J'$, respectively) which 
		satisfy
		$$ {\rm tp}^{a_J,J}(f) = {\rm tp}^{a_{J'}, J'}(f')\;\hbox{ and }\;   {\rm tp}^{b_J,J}(f) = {\rm tp}^{b_{J'}, J'}(f').$$
	\end{enumerate} 
		Then there exist a finite linear graph $L$ with endpoints $a_L$, $b_L$ and epimorphisms $g: L \to J$, $g': L \to J'$ with 
		$$f \circ g = f' \circ g',$$
		and 
		\[
		g(a_L) = a_J,\; g'(a_L) = a_{J'},\; g(b_L) = b_J,\; g'(b_L) = b_{J'}.
		\]	
\end{lemma}

	First, we deal with the following particular case of Lemma~\ref{amalglem}.
	
	\begin{lemma} \label{lift0}
		Suppose that $I$, $J$, $J'$ are finite linear graphs, $f \in {\rm Epi}(J, I)$, $f' \in {\rm Epi}(J', I)$, and the endpoints $a_J,b_J$ of $J$, $a_{J'},b_{J'}$ of $J'$ and $a_I,b_I$ of $I$  satisfy $f(a_J) = a_I$,  $f(b_J) = b_I$, $f'(a_{J'}) = a_I$,  $f'(b_{J'}) = b_I$.
		
		Then, for some finite linear graph $L$ and epimorphisms $g: L \to J$, $g': L \to J'$ we have $f \circ g = f' \circ g'$, and one endpoint of $L$ is mapped to $a_J$  by $g$ ($a_{J'}$ by $g'$, resp.), while the other endpoint is mapped to $b_J$ and $b_{J'}$.
	\end{lemma}
	\begin{PROOF}{Lemma \ref{lift0}}
		Extend the linear graphs $J$, $J'$, $I$ by two points as follows. We let 
		$$J_{ext} = J \cup \{c_J,d_J\},$$
		such that $c_J$ is connected with $a_J$, and  $d_J$ is connected with $b_J$. Similarly,
		$$J'_{\rm ext} = J' \cup \{c_{J'},d_{J'}\},$$
		and
		$$I_{\rm ext} = I \cup \{c_I,d_I\}.$$
		It is easy to see that $\tilde{f}: J_{\rm ext} \to I_{\rm ext}$, defined by the identities $\tilde{f} \rest J = f$, 
		\begin{equation} \label{f'def} 
			\begin{array}{l}
				\tilde{f}(c_J) = c_I,  \\
				\tilde{f}(d_J) = d_I
			\end{array}
		\end{equation} is an epimorphism.
		Similarly, letting $\tilde{f}': J'_{\rm ext} \to I_{\rm ext}$ denote the mapping that satisfies $\tilde{f}' \rest J' = f'$, 
		\begin{equation} \label{f''def} 
			\begin{array}{l}
				\tilde{f}'(c_{J'}) = c_I,  \\
				\tilde{f}'(d_{J'}) = d_I
			\end{array}
		\end{equation}
		is an epimorphism extending $f'$.
		Now let $L$ be a linear graph for which for some epimorphisms
		$\tilde{g}: L \to J_{\rm ext}$, and 	$\tilde{g}': L \to J'_{\rm ext}$ we have $\tilde{f} \circ \tilde{g} = \tilde{f}' \circ \tilde{g}'$.
		(It is straightforward to see that) by passing to a minimal subinterval of $L$ with  $\tilde{f} \circ \tilde{g} (= \tilde{f}' \circ \tilde{g}')$ mapping onto $I^{\rm ext}$ we can assume that endpoints of $L$ are mapped to endpoints of $I_{\rm ext}$ (that is, to $c_I,d_I$) by $\tilde{f} \circ \tilde{g}$, and only those are mapped to $c_I,d_I$.
		
		Now recalling \eqref{f'def}, \eqref{f''def} and $c_I,d_I \in I_{\rm ext} \setminus I$,
		it is easily checked that for every $l \in L$
		$$ \tilde{f} \circ \tilde{g}(l) = c_I \ \iff \  \tilde{g}(l)= c_J, $$
		and
		$$ \tilde{f}' \circ \tilde{g}'(l) = c_I \ \iff \  \tilde{g}'(l)= c_{J'}$$
		(and similarly with $d_I$, $d_J$, $d_{J'}$).
		Again by minimality one easily checks that if $l \in L$ is adjacent to an endpoint of $L$, then \begin{itemize}
			\item either $\tilde{f}\circ \tilde{g}(l) =\tilde{f}' \circ \tilde{g}'(l) = a_I$, and $\tilde{g}(l) = a_J$, $\tilde{f}'(l) = a_{J'}$,
			\item or $\tilde{f}\circ \tilde{g}(l) =\tilde{f}' \circ \tilde{g}'(l) = b_I$, and $\tilde{f}(l) = b_J$, $\tilde{f}'(l) = b_{J'}$.
		\end{itemize} 
		So replacing $L$ with the subgraph that remains upon removal of the  two endpoints (and letting $g$, $g'$ be the appropriate restrictions of $\tilde{g}$, $\tilde{g}'$) works.					
		\end{PROOF}
	
	\begin{PROOF}[Proof of Lemma~\ref{amalglem}]{Lemma \ref{amalglem}}
	The proof is by induction on $|I|$ using Lemma~\ref{lift0} above.
	
	If $|I| = 1$, then any $L$ that admits epimorphisms to $J$ and $J'$ works.
	Suppose that Lemma~\ref{amalglem} holds  if $|I| <n$, and fix 	$I$, $J$, $J'$, $f$, $f'$ where $|I| = n$.
	Enumerate $I$ as $\{a_0,a_1, \ldots, a_{n-1}\}$ such that
	$$ \{\{a_j: \ j \leq k  \}: k < n\} = {\rm tp}^{a_j,J}(f) = {\rm tp}^{a_{J'}, J'}(f'),$$
	(in particular, $f(a_J) = f'(a_{J'}) = a_0$),  and pick an enumeration $\{b_0, b_1, \ldots, b_{n-1}\}$ with
	
	$$  \{\{b_j: \ j \leq k  \}: k < n\}  = {\rm tp}^{b_J,J}(f) = {\rm tp}^{b_{J'}, J'}(f').$$
	We note that both $a_{n-1}$ and $b_{n-1}$ must be (possibly different endpoints) of $I$.
	We assume that $a_{n-1}= b_{n-1}$ (the other case is similar but simpler,  and is  left to the reader).
	
	Now we reduce the task of finding $I$ and $g$, $g'$ to four amalgamation task.
	Enumerate 
	$$J = \{a_J = u_0, u_1, \ldots, u_{|J|-1} = b_J\} \ \textrm {  so that } u_k R u_{k+1},$$
	and
	$$J' = \{a_{J'} = u'_0, u'_1, \ldots, u'_{|J'|-1} = b_{J'}\} \ \textrm {  so that } u'_k R u'_{k+1}.$$
	Let
	\begin{itemize}
		\item $s_0 < |J|$ be minimal such that $f(u_{s_0}) = a_{n-1}$, and
		\item $s'_0 < |J'|$ be minimal such that $f'(u'_{s_0'}) = a_{n-1}$,
	\end{itemize} 
	and set
	$$K_0 = \{u_0,u_1, \ldots, u_{s_0-1}\},$$
	$$K_0' = \{u'_0,u'_1, \ldots, u'_{s'_0-1}\}.$$
	While it is clear that 
	\begin{equation} \label{condK0}
		f[K_0] = f'[K_0'] = I \setminus \{a_{n-1}\},
	\end{equation}
	\begin{equation} \label{typeK0}
		 {\rm tp}^{u_0,K_0}(f) = {\rm tp}^{u'_0, K_0'}(f'),
	\end{equation}
	we argue that
	\begin{equation} \label{typeK0'}
		{\rm tp}^{u_{s_0-1},K_0}(f) = {\rm tp}^{u'_{s'_0-1}, K_0'}(f'),
	\end{equation}
	holds as well.
	Indeed, by the way we defined $s_0$ and $s'_0$ we have $f(u_{s_0-1}) = f'(u'_{s_0'-1}) = a_{n-1}$ is and endpoint of $f[K_0] = f'[K_0'] = I \setminus \{a_{n-1}\}$, so it is enough to recall that there is only one type starting with a fixed endpoint.
	
	Similarly, we let
	\begin{itemize}
		\item $t_0 < |J|$ be maximal such that $f(u_{t_0}) = b_{n-1}$, and
		\item $t'_0 < |J'|$ be maximal such that $f'(u'_{t_0'}) = b_{n-1}$,
	\end{itemize} 
	and set
	$$K_3 = \{u_{t_0+1}, u_{t_0+2}, \ldots, u_{|J|-1}\},$$
	$$K_3' = \{u'_{t'_0+1}, u'_{t'_0+2}, \ldots, u'_{|J'|-1}\}.$$
	So again,
	\begin{equation}\label{condK3}
		f[K_3] = f'[K_3'] = I \setminus \{b_{n-1}\},
	\end{equation}
	\begin{equation}
		{\rm tp}^{u_{|J|-1},K_3}(f) = {\rm tp}^{u'_{|J'|-1}, K_3'}(f'),
	\end{equation}
	\begin{equation} \label{typeK3'}
		{\rm tp}^{u_{t_0+1},K_3}(f) = {\rm tp}^{u'_{t'_0+1}, K_3'}(f').
	\end{equation}
	Now using \ref{condK0}-\ref{typeK0'} and applying the induction hypothesis
	\begin{enumerate}[label = $({\rm R}0)$, ref = $({\rm R}0)$]
		\item \label{amK0} $f \rest K_0$, $f \rest K_0'$ can be amalgamated in a way that endpoints are mapped to endpoints, more concretely one endpoint is mapped to $u_0$ and $u'_0$, while the other is mapped to $u_{s_0-1}$ and $u'_{s'_0-1}$,
	\end{enumerate}
	and by \eqref{condK3}-\eqref{typeK3'},
	\begin{enumerate}[label = $({\rm R}3)$, ref = $({\rm R}3$]
		\item \label{amK3} $f \rest K_3$, $f\rest K_3'$ can be amalgamated in a way that endpoints are mapped to endpoints, more concretely one endpoint is mapped to $u_{|J|-1}$ and $u'_{|J'|-1}$, while the other is mapped to $u_{t_0+1}$ and $u'_{t'_0+1}$.
	\end{enumerate}
	Recall that we assumed that $a_{n-1} = b_{n-1}$ (which is an endpoint of $I$),
	pick 
	 $q_0 < |J|$ and  $q'_0 < |J'|$ such that 
	 \begin{equation} \label{q0} f(u_{q_0}) = f'(u'_{q_0'}) \textrm{ is the other endpoint.} \end{equation}
	 Set
	$$K_1 = [u_{s_0}, u_{q_0}],$$
	$$K_1' = [u'_{s'_0}, u'_{q'_0}],$$
	and
	$$K_2 = [u_{q_0}, u_{t_0}],$$
	$$K_2' = [u'_{q'_0}, u'_{t'_0}].$$
	So by the way $s_0$ and $s_0'$ are picked $f(u_{s_0}) = f'(u'_{s'_0})$ we note that $f \rest K_1$ and $f'_\rest K_1'$ map endpoints to endpoints, so Lemma~\ref{lift0} applies, so
	\begin{enumerate}[label = $({\rm R}1)$, ref = $({\rm R}1)$]
		\item \label{amK1} $f \rest K_1$, $f \rest K_1'$ can be amalgamated in a way that endpoints are mapped to endpoints, more concretely one endpoint is mapped to $u_{q_0}$ and $u'_{q_0'}$, while the other is mapped to $u_{s_0-1}$ and $u'_{s'_0-1}$.
	\end{enumerate}
	Similarly,
	\begin{enumerate}[label = $({\rm R}2)$, ref = $({\rm R}2)$]
		\item \label{amK2} $f \rest K_2$, $f \rest K_2'$ can be amalgamated such that endpoints are mapped to endpoints, more concretely one endpoint is mapped to $u_{q_0}$ and $u'_{q_0'}$, while the other is mapped to $u_{t_0+1}$ and $u'_{t'_0+1}$.
	\end{enumerate}
	Putting together the maps in \ref{amK0}, \ref{amK1}, \ref{amK2}, \ref{amK3} we get the conclusion.
	\end{PROOF}

	\begin{PROOF}[Proof of Lemma~\ref{amalgamp}]{Lemma \ref{amalgamp}}
	We reduce Lemma~\ref{amalgamp} to Lemma~\ref{amalglem}.	
	
	Suppose that $I,J,J'$, $f,f'$ and $a_J \in J$, $a_{J'} \in J'$ are as in our assumptions.
	Pick $c_J \in J$, $c_{J'} \in J'$ such that $f(c_J) = f'(c_{J'})$ is an endpoint of $I$.
	Let $\alpha: K \to J$ be an epimorphism which maps the endpoint $a_K$ ($b_K$, resp.) of $K$ to $a_J$ ($c_J$, resp.) and similarly, $\alpha': K' \to J'$ is such that $a_{K'}$ goes to $a_{J'}$, and $b_{K'}$ is sent to $c_{J'}$. Since 
	$$f \circ \alpha(b_K) = f(c_J) = f'(c_{J'}) = f \circ \alpha(b_{K'}) \textrm{ is an endpoint of }I,$$
	clearly 
	\begin{equation} \label{tpe} {\rm tp}^{b_K,K}(f\circ \alpha ) = {\rm tp}^{b_{K'}, K'}(f' \circ \alpha'). \end{equation}
	On the other hand,  $\alpha(a_K) = a_{J}$ is an endpoint, so it is not hard to see that
	$$ {\rm tp}^{a_K,K}(f \circ \alpha) = {\rm tp}^{a_J, J}(f),$$
	and similarly
	$$ {\rm tp}^{a_{K'},K'}(f' \circ \alpha') = {\rm tp}^{a_{J'}, J'}(f'),$$
	so by our assumptions 
	\begin{equation} \label{tpe'} {\rm tp}^{a_K,K}(f \circ \alpha) = {\rm tp}^{a_{K'}, K'}(f' \circ \alpha'), \end{equation}
	So by \eqref{tpe}, \eqref{tpe'} Lemma~\ref{amalglem} applies. Replacing the resulting $g$ and $g'$ by $\alpha \circ g$ and $\alpha' \circ g'$, respectively, we are done. 
	\end{PROOF}

\end{document}